\documentclass[a4paper,11pt]{article}
\usepackage[T1]{fontenc}
\usepackage{rotating}

\usepackage[title]{appendix}
\usepackage{upgreek}
\usepackage{amsmath} 
\usepackage{amsthm}
\usepackage{amsfonts}
\usepackage{mathtools}
\usepackage{amssymb}
\usepackage{enumitem}
\usepackage{gensymb}
\usepackage{arydshln}
\usepackage[noadjust]{cite}

\def\x{\mathbf{x}}
\def\y{\mathbf{y}}

\def\bx{\overline{\mathbf{x}}}
\def\by{\overline{\mathbf{y}}}
\def\z{\mathbf{z}}
\def\bz{{\overline{\mathbf{z}}}}
\def\q{\mathbf{q}}
\def\bq{{\overline{\mathbf{q}}}}
\def\w{\mathbf{w}}

\newcommand{\RR}{\mathbb{R}}
\newcommand{\ZZ}{\mathbb{Z}}
\newcommand{\nul}{\mathbf{0}}

\DeclareMathOperator{\rank}{rank}
\DeclareMathOperator{\dem}{dem}
\DeclareMathOperator{\Ker}{ker}
\DeclareMathOperator{\im}{im}
\DeclareMathOperator{\CoKer}{coker}

\DeclareMathOperator{\Circ}{Circ}

\DeclareMathOperator{\cart}{\Box}

\DeclareMathOperator{\Su}{S}
\usepackage{graphicx}
\usepackage{url}
\usepackage[hyperfootnotes=true,colorlinks=true,pagebackref=false]{hyperref}
\hypersetup{
    colorlinks=true,
    linkcolor=blue,
    citecolor=magenta,      
    urlcolor=black
}

\usepackage{tikz}
\usetikzlibrary{calc}
\usetikzlibrary{backgrounds}
\usetikzlibrary{intersections}
\usetikzlibrary{decorations.markings}
\usetikzlibrary{arrows.meta}
\usetikzlibrary{decorations.pathreplacing}
\usetikzlibrary{patterns}
\usetikzlibrary{patterns.meta}

\usepackage{authblk}
\usepackage{subcaption}
 
\usepackage[text={170mm,254mm}]{geometry}

\theoremstyle{plain}
\newtheorem{theorem}{Theorem}
\newtheorem{lemma}[theorem]{Lemma}
\newtheorem{corollary}[theorem]{Corollary}
\newtheorem{proposition}[theorem]{Proposition}

\theoremstyle{definition}
\newtheorem{definition}[theorem]{Definition}

\newtheorem{example}[theorem]{Example}

\usepackage{amsfonts}
\usepackage{amsmath}
\usepackage{ifxetex,ifluatex}
\if\ifxetex T\else\ifluatex T\else F\fi\fi T%
  \usepackage{fontspec}
\else
  \usepackage[T1]{fontenc}
  \usepackage[utf8]{inputenc}
  \usepackage{lmodern}
\fi

\usepackage{hyperref}

\title{Nut digraphs}

\author[1,2,7]{Nino Ba\v{s}i\'{c}}
\author[3]{Patrick~W.~Fowler}
\author[4,5]{Maxine M. McCarthy}
\author[6,7]{Primo\v{z} Poto\v{c}nik}

\affil[1]{FAMNIT, University of Primorska, Koper, Slovenia}
\affil[2]{IAM, University of Primorska, Koper, Slovenia}
\affil[3]{Chemistry, School of Mathematical and Physical Sciences, \linebreak University of Sheffield, Sheffield S3 7HF, UK}
\affil[4]{Physics and Astronomy, School of Mathematical and Physical Sciences, \linebreak University of Sheffield, Sheffield S3 7RH, UK}
\affil[5]{Max Planck Institute for the Science of Light, Staudtstra{\ss}e 2, 91058 Erlangen, Germany}
\affil[6]{Faculty of Mathematics and Physics, University of Ljubljana, Slovenia}
\affil[7]{Institute of Mathematics, Physics and Mechanics, Ljubljana, Slovenia}

\begin{document}

\maketitle

\begin{abstract}
A \emph{nut graph} is a simple graph whose kernel is spanned by a single full vector (i.e.\ the adjacency matrix has a single zero eigenvalue and all non-zero 
kernel eigenvectors have no zero entry). 
We classify generalisations of nut graphs to nut digraphs:
a digraph whose kernel (resp.\ co-kernel) is spanned by a full vector is \emph{dextro-nut} (resp.\ \emph{laevo-nut});
a \emph{bi-nut} digraph is both laevo- and dextro-nut;
an \emph{ambi-nut} digraph is a bi-nut digraph where  kernel and co-kernel are spanned by the same vector;
a digraph is \emph{inter-nut} if the intersection of the kernel and co-kernel is spanned by a full vector.
It is known that a nut graph is connected, leafless and non-bipartite. It is shown here that an ambi-nut digraph is
strongly connected, non-bipartite (i.e.\ has a non-bipartite underlying graph) and has minimum in-degree and minimum 
out-degree of at least $2$. Refined notions of core and core-forbidden vertices apply to singular digraphs. Infinite families of
nut digraphs and systematic coalescence, cross-over and multiplier constructions are introduced.
Relevance of nut digraphs to topological physics is discussed.

\vspace{\baselineskip}
\noindent
\textbf{Keywords:} Nut graph, core graph, nullity, directed graph, nut digraph,
dextro-nut, laevo-nut, bi-nut, ambi-nut, inter-nut, dextro-core vertex, laevo-core vertex, graph spectra.

\vspace{\baselineskip}
\noindent
\textbf{Math.\ Subj.\ Class.\ (2020):} 
05C50, 
05C20, 
05C92 
\end{abstract}

\section{Introduction}

\emph{Nut graphs}~\cite{SciGut1998} are the graphs that have a one-dimensional nullspace where the non-trivial
kernel eigenvector $x = [x_1\, \ldots\, x_n ]^\intercal \in \Ker A (G)$ is 
\emph{full} (i.e.\ it  has no zero entries).
This class of graphs is the subject of an extensive mathematical literature~\cite{SciGut1998, ScirihaCoalescence, Basic2021regular, Basic2020, ScMaxCorSzSing09, Jan2020, GPS, GutmanSciriha-MaxSing, Sc1998,
CoolFowlGoed-2017, Sc2007, Damnjanovic2022a, DamnjanStevan2022, Damnjanovic2022b, Damnjanovic2022c, NutOrbitPaper, GroupNutPaper, DegreeNutPaper, DaBaPiZi2024, Sciriha2021, ScOnRnkGr99,QuarticBicirculantNuts,CubicPolyNuts}. 
Nut graphs also feature in mathematical chemistry in applications of graph theory to 
models of electron distribution, radical reactivity and ballistic conduction in molecular $\pi$ systems \cite{ScFo2008, nuttybook, PWF_JCP_2014, feuerpaper, Fowler2013}.
The definition has been extended from graphs to signed graphs \cite{Basic2020singular}.
 
Recently, after a talk on nut graphs given by one of us, 
a member of the audience raised the question of whether analogues could be found amongst the directed graphs. 
In other words, the questioner wanted to know whether \emph{nut digraphs} exist.  
This proved to be a fruitful line of enquiry.
The (undirected) nut graphs are connected,
leafless (have no vertices of degree~$1$),  non-bipartite, 
not edge-transitive, and non-trivial examples exist for all orders
$n \geq 7$~\cite{GutmanSciriha-MaxSing, NutOrbitPaper}. 
A natural follow-up question would be on how these properties transfer to non-trivial nut digraphs.
The present note reports our work to provide a full answer to both questions, which has entailed
consideration of alternative definitions of nut digraphs, cataloguing small examples and
devising systematic constructions for large nut digraphs.

\section{Definitions}
\label{sec:2}

Nut graphs sit within the larger class of core graphs. A \emph{core graph} is a singular graph for which every vertex has a non-zero entry in some non-trivial kernel eigenvector. 
It is traditional to exclude $K_1$ from both nut and core graph classes, discounting it as a trivial case~\cite{SciGut1998}.  Nut graphs are then
precisely the core graphs of nullity $1$.

 A {\em digraph} is a pair $(V,\to)$, where $V$ is a finite non-empty set of vertices and $\to$ is an
 arbitrary binary relation on $V$. The set of vertices of a digraph $G$ is  denoted by $V(G)$ and
 the corresponding relation $\to$ is denoted by $\to_G$.
If $u\to_G v$ for some $u,v \in V(G)$, then we say that $u$ is an {\em in-neighbour} of $v$, 
that $v$ is an {\em out-neighbour} of $u$ and that $(u,v)$ is an {\em arc} (also called a {\em directed edge})
of $G$. The set of all arcs of $G$ will sometimes be denoted by $E(G)$.
For a vertex $u\in V(G)$, let $G^+(u) $ denote the set of all out-neighbours and $G^-(u)$ the set of all in-neighbours of $u$.
The digraph $G^R$ whose vertex set is $V(G)$ and where $u \to_{G^R} v$ if and only if $v\to_G u$ is 
called the \emph{reverse digraph} 
(sometimes known as the \emph{opposite digraph} or \emph{converse digraph}) 
of $G$. We denote the in- and out-degrees of a vertex $u \in V(G)$ by $d^{-}(u)$ and $d^{+}(u)$, respectively.
A vertex $u$ is called a \emph{source} if $d^{-}(u) = 0$ and $d^{+}(u) > 0$. Similarly, a vertex $u$ is called a \emph{sink} if $d^{+}(u) = 0$ and $d^{-}(u) > 0$.
We use $d(u)$ to denote $d^{-}(u) + d^{+}(u)$.
The minimal in- and out-degrees of a digraph $G$ will be denoted by $\delta^{-}(G)$ and $\delta^{+}(G)$, respectively.
The minimal degree of an (undirected) graph $\Gamma$ will be denoted $\delta(\Gamma)$.

If the relation $\to_G$ is irreflexive (in the sense that no $v\in V(G)$ satisfies $v\to_G v$) and symmetric (in the sense that for all $u,v\in V(G)$,
 $u\to_G v$ implies $v\to_G u$),
 then we call $G$ a (simple) graph and the relation $\to_G$ becomes the usual adjacency relation.
Similarly, if $\to$ is asymmetric (in the sense that for no pair of possibly equal vertices $u$ and $v$ of $G$
do both $u\to_G v$ and $v\to_G u$ hold),
 then $G$ is said to be an \emph{oriented} (simple) \emph{graph}, as it can be obtained by orienting each edge of a simple graph.
The {\em underlying graph} of a directed graph $G$ is defined as the digraph with the same vertex-set as $G$ and with 
the adjacency relation being the symmetric closure of the adjacency relation of $G$. Note that the underlying graph is always a simple graph.

Our definition of bipartite digraphs is the same as the one used by Brualdi \cite{Brualdi2010}.
We will call a digraph $G$ \emph{bipartite} if its vertex-set $V(G)$ can be partitioned into subsets $V_1$ and $V_2$, such that for every arc $(u, v) \in E(G)$,
either $u \in V_1$ and $v \in V_2$ or $v \in V_1$ and $u \in V_2$. 
Let $U$ and $W$ be two non-empty disjoint sets of vertices of a digraph $G$.
Then $G[U]$ is the digraph with vertex-set $U$ with $u_1 \to_{G[U]} u_2$ if and only if $u_1 \to_{G} u_2$, i.e. the digraph induced on $U$.
Similarly,
$G[U,W]$ is the digraph with vertex-set $U \cup W$ and with $v_1\to_{G[U,W]} v_2$ if and only if $|\{v_1,v_2\} \cap U| = |\{v_1,v_2\} \cap W| = 1$ and $v_1\to_G v_2$.
Note that the digraph $G[U,W]$ is bipartite.

 A digraph $G$ can be viewed as a linear operator on
the $|V(G)|$-dimensional vector space $\RR^V$ of all functions $\x \colon V(G) \to \RR$, where the action of $G$ is given by
\begin{equation}
\label{eq:operator}
G\colon \x \mapsto (u \mapsto \sum_{v\in G^+(u)} \x(v)).
\end{equation}
An element of $\RR^V$ is {\em full} if $\x(v) \not = 0$ for every $v\in V$.
For simplicity, we write $\x_v$ instead of $\x(v)$, and when the vertices of $G$ are in some order, say $v_0, v_1, \ldots, v_{n - 1}$, then we identify
a function $\x \in \RR^V$ with the column vector 
$[\begin{matrix}\x_{v_0} & \x_{v_1} & \ldots & \x_{v_{n - 1}}\end{matrix}]^\intercal$.
Furthermore, for $v_i\in V$, let $\chi_i \in \RR^V$ denote the characteristic function of $v$, and observe that $\{\chi_i : i\in \ZZ_n\}$ is a basis for $\RR^V$.
The matrix that corresponds to $G$ in this basis is then precisely the adjacency matrix $A(G)$ of $G$, for which $A(G)_{ij}=1$
if and only if $v_i \to_G v_j$ (and is $0$ otherwise). Note that
\begin{equation}
A(G^R) = A(G)^\intercal.
\end{equation}
The kernel $\Ker G$ of the digraph $G$, viewed as an operator on $\RR^V$, consists of all the elements $\x\in \RR^V$ satisfying
\begin{equation}
\label{eq:zerosum}
\sum_{u\in G^+(v)}\x(u) = 0\> \hbox{ for all } v\in V(G).
\end{equation}
Similarly, the co-kernel $\CoKer G$ of the digraph $G$, defined as the kernel of the operator $G^R$, consists of all the elements $\x\in \RR^V$ such that 
\begin{equation}
\label{eq:zerosumn}
\sum_{u\in G^-(v)}\x(u) = 0\>  \hbox{ for all } v\in V(G).
\end{equation}

If $\x\in \RR^V$ is identified with the corresponding vector in $\RR^n$, then 
$\Ker G = \{ \x \in \RR^n : A(G) \x = \nul\}$ and $\CoKer G = \Ker G^R = \{ \x \in \RR^n : A(G)^\intercal \x = \nul\} = \{ \x \in \RR^n : \x^\intercal A(G) = \nul^\intercal\}$.
In other words, the kernel and the cokernel of a digraph are then precisely the kernel and the cokernel of the adjacency matrix of $G$.  
The nullity of a digraph $G$, denoted $\eta(G)$,
 is defined as $\eta(G) = \dim \Ker G$. 
Note that $\eta(G) = \dim \Ker G = \dim \CoKer G$.
A digraph $G$ is called \emph{singular} if $\eta(G) > 0$ and \emph{non-singular} otherwise. In other words, $G$ is non-singular if it is invertible as an operator in the sense of \eqref{eq:operator}.
Recall that for an undirected graph, its adjacency matrix is symmetric, and the algebraic multiplicity of an eigevalue matches its geometric multiplicity. In particular, the nullity of an
undirected graph equals the algebraic multiplicity of the $0$ eigenvalue. In the directed case, this is no longer true. With the definition adopted in this paper, the nullity of a directed graph is the geometric multiplicity of the $0$ eigenvalue,
but not necessarily its algebraic multiplicity.

We are now in a position to discuss plausible generalisations of the notion of a nut graph to digraphs.
 Each of the generalisations that we propose below will have the property that when applied to a digraph which is a graph (that is, the adjacency relation of which
 is symmetric) it will coincide with the standard definition of a nut graph. 

\begin{definition}
A digraph is called a \emph{dextro-nut} provided it has a one-dimensional kernel spanned by a full vector. Similarly, a digraph is a 
\emph{laevo-nut} if its co-kernel is one-dimensional and spanned by a full vector.
\end{definition}

\noindent
Observe that $G$ is a laevo-nut digraph if and only if its reverse $G^R$ is a dextro-nut digraph. Note that the conditions (\ref{eq:zerosum}) and (\ref{eq:zerosumn}) 
are refinements of the usual \emph{ local condition} for kernel eigenvectors in the setting of undirected graphs;
see Figure~\ref{fig:schematicLocal}.

\begin{figure}[!htbp]
\centering
\begin{tikzpicture}[scale=1]
\tikzstyle{vertex}=[draw,circle,font=\scriptsize,minimum size=5pt,inner sep=1pt,fill=orange!80!white]
\tikzstyle{edge}=[draw,thick,decoration={markings,mark=at position 0.5 with {\arrow{Stealth}}},postaction=decorate]
\node[vertex,label={180:$b_1$}] (b1) at (-1, 0.6) {};
\node[vertex,label={180:$b_2$}] (b2) at (-1, 0) {};
\node[fill=none,draw=none] (dots1) at (-1, -0.6) {$\cdots$};
\node[vertex,label={180:$b_\ell$}] (bl) at (-1, -1.2) {};
\node[vertex,label={0:$a_1$}] (a1) at (4, 0.6) {};
\node[vertex,label={0:$a_2$}] (a2) at (4, 0) {};
\node[fill=none,draw=none] (dots2) at (4, -0.6) {$\cdots$};
\node[vertex,label={0:$a_k$}] (ak) at (4, -1.2) {};
\node[vertex,label={-90:$v$}] (v) at (1.5, -0.3) {};
\path[edge] (b1) to[bend left=20] (v);
\path[edge] (b2) -- (v);
\path[edge] (bl) to[bend right=20] (v);
\path[edge] (v) to[bend left=20] (a1);
\path[edge] (v) -- (a2);
\path[edge] (v)  to[bend right=20] (ak);
\draw [thick, decorate, decoration={brace, amplitude=5pt}]  (-1.8, -1.5) -- (-1.8, 0.9) node[pos=0.5,left=6pt,black]{$G^-(v)$};
\draw [thick, decorate, decoration={brace, amplitude=5pt}]  (4.8, 0.9) -- (4.8, -1.5) node[pos=0.5,right=6pt,black]{$G^+(v)$};
\end{tikzpicture}
\caption{Local conditions in nut digraphs. Entries on $G^+(v)$ of $\x \in \Ker G$ are labeled $a_1, \ldots, a_k$ and entries on $G^-(v)$ of $\y \in \CoKer G$ are labeled $b_1, \ldots, b_\ell$.
For a dextro-nut digraph, the local condition is $a_1 + \cdots + a_k = 0$ and for a laevo-nut $b_1 + \cdots + b_\ell = 0$. For the underlying graph the local
condition would simply be the sum of the two: $a_1 + \cdots + a_k + b_1 + \cdots + b_\ell = 0$.}
\label{fig:schematicLocal}
\end{figure}
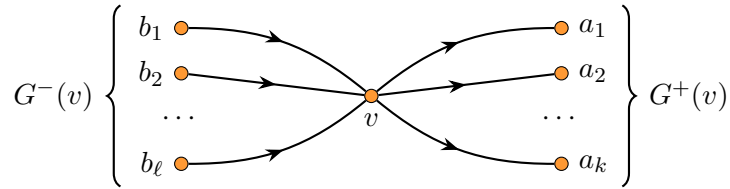

\begin{definition}
A digraph which is both a dextro-nut and a laevo-nut is a {\em bi-nut} digraph. 
\end{definition}

\begin{definition}
A bi-nut digraph whose kernel and the co-kernel are spanned by \emph{the same vector}
is an \emph{ambi-nut} digraph.
\end{definition}

\noindent
Both  kernel and co-kernel of a dextro-nut digraph are one-dimensional. Hence, if a dextro-nut digraph is not a bi-nut digraph, 
the vector that spans its co-kernel must contain some zero entries. Likewise, \emph{mutatis mutandis}, for laevo-nut digraphs.

Finally, one might consider digraphs for which the intersection of the kernel and the co-kernel is one-dimensional and is spanned by a full vector.
Such an object could be called an {\em inter-nut} digraph. 

\begin{definition}
A digraph $G$ is an \emph{inter-nut} digraph if $\Ker G \cap \CoKer G$  is one-dimensional and spanned by a full vector.
\end{definition}

\noindent
Note that a digraph is an ambi-nut if and only it is both
a bi-nut and an inter-nut. Moreover, an inter-nut that is a dextro-nut is automatically also a laevo-nut and therefore an ambi-nut. Figure~\ref{fig:nutUniverse} gives an overview of the  relations between these nut digraph classes.

In this paper, we will be concerned mainly with the strongest of these generalisations of nut graphs, namely with the ambi-nut digraphs.
However, we will prove results in as general a form as possible, thus providing results for the weaker forms of nut digraph \emph{en passant}.
\medskip

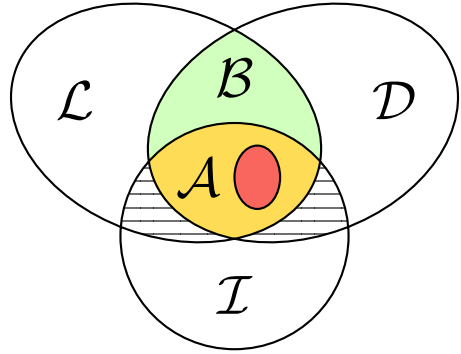
\begin{figure}[!htbp]
\centering
\begin{tikzpicture}[scale=0.6]
\definecolor{biColor}{RGB}{209,255,189}
\definecolor{ambiColor}{RGB}{254,220,86}
\definecolor{nutColor}{RGB}{248,102,94}
\def\laevoNuts{(-0.5,0) ellipse [x radius=3.5,y radius=2.5,rotate=160]}
\def\dextroNuts{(2.5,0) ellipse [x radius=3.5,y radius=2.5,rotate=20]}
\def\interNuts{(1,-2.5) ellipse [x radius=2.5,y radius=2.5,rotate=90]}
\def\nutGraphs{(1.5,-1.2) ellipse [x radius=0.5,y radius=0.7,rotate=0]}
\def\boundingbox{(-4,-5.1) rectangle (6,2.8)}
\begin{scope}
\clip \boundingbox \laevoNuts; \clip \dextroNuts; \fill[pattern={Lines[angle=45,distance=5pt]},pattern color=black] \interNuts;
\end{scope}
\begin{scope}
\clip \boundingbox \dextroNuts; \clip \laevoNuts; \fill[pattern={Lines[angle=-45,distance=5pt]},pattern color=black] \interNuts;
\end{scope}
\begin{scope}
\clip \boundingbox \interNuts; \clip \laevoNuts; \fill[biColor] \dextroNuts;
\end{scope}
\begin{scope}
\clip \laevoNuts; \clip \dextroNuts; \clip \interNuts; \fill[ambiColor] \boundingbox;
\end{scope}
\draw[thick] \laevoNuts \dextroNuts \interNuts;
\draw[thick,fill=nutColor] \nutGraphs;
\node[draw=none,fill=none] at (-2.5,0.5) {\huge$\mathcal{L}$};
\node[draw=none,fill=none] at (4.5,0.5) {\huge$\mathcal{D}$};
\node[draw=none,fill=none] at (1,1) {\huge$\mathcal{B}$};
\node[draw=none,fill=none] at (1,-3.8) {\huge$\mathcal{I}$};
\node[draw=none,fill=none] at (0.2,-1.2) {\huge$\mathcal{A}$};    
\end{tikzpicture}
\caption{Venn diagram of relationships amongst notions of nut digraphs.  Symbols $\mathcal{D}, \mathcal{L}, \mathcal{B}, \mathcal{A}$ and $\mathcal{I}$ denote dextro-,
laevo-, bi-, ambi- and inter-nut digraphs, respectively. Note that the usual nut graphs occupy the small red oval inside the ambi-nut region,
and that the hatched regions are empty.}
\label{fig:nutUniverse}
\end{figure}

\section{Examples and enumeration}

In this section, we give small examples for the five distinct notions of nut digraphs  introduced above.
To begin, we also enumerate nut digraphs of small orders. 

Let $n$ be the number of vertices of the digraphs under consideration.
Then let  $\mathcal{U}_n$ be the class of undirected graphs,
 $\mathcal{O}_n$ the class of oriented digraphs,
and
 $\mathcal{D}_n, \mathcal{B}_n$ and $\mathcal{A}_n$ be
the classes of dextro-nut, bi-nut and ambi-nut digraphs, respectively, all on $n$ vertices. 
 
Throughout this paper, a digraph will be called regular (resp.~$k$-regular) if its underlying graph is regular (resp.~$k$-regular). 
We will use $\mathcal{O}_n^d$ to denote the class of
$d$-regular digraphs of order $n$.
Moreover, let
$\mathcal{O}(\Gamma)$ denote the class of oriented digraphs whose underlying (undirected) graph is $\Gamma$. We use a similar
notation for the classes of nut digraphs, e.g., $\mathcal{B}(\Gamma)$ is the class of bi-nut digraphs for which the underlying graph
is $\Gamma$.

 Table~\ref{tbl:enumGeneral} lists counts of dextro-, bi- and ambi-nut digraphs among \emph{oriented} graphs on up to $8$ vertices. 
 These numbers were obtained naïvely, i.e.\ by generating all oriented graphs on $n$ vertices and filtering out nut digraphs.
The counts of laevo-nut digraphs are of course
equal to those of the dextro-nut digraphs, order by order. Figure~\ref{fig:smallGeneral} shows 
 small examples of nut digraphs from  Table~\ref{tbl:enumGeneral}.

\begin{table}[!htb]
\centering
\begin{tabular}{rrrrrr}
\hline
$n$ & $|\mathcal{U}_n|$ & $|\mathcal{O}_n|$ 
& $|\mathcal{D}_n|$ & $|\mathcal{B}_n|$ & $|\mathcal{A}_n|$ \\
\hline\hline
3 & 2 & 5 
 & 0 & 0 & 0 \\
4 & 6 & 34 
& 1 & 0 & 0 \\
5 & 21 & 535 
& 4 & 0 & 0 \\
6 & 112 & 20848 
& 153 & 2 & 2 \\
7 & 853 & 2120098 
& 17170 & 21 & 1 \\
8 & 11117 & 572849763  
& 5579793 & 9592 & 104 \\
\hline
\end{tabular}
\caption{Enumeration of nut digraphs among oriented graphs on $n \leq 8$ vertices.}
\label{tbl:enumGeneral}
\end{table}

\begin{figure}[!htb]
\centering
\subcaptionbox{\label{fig:exampleDextro}}{
\begin{tikzpicture}[scale=2]
\tikzstyle{vertex}=[draw,circle,font=\scriptsize,minimum size=5pt,inner sep=1pt,fill=magenta!80!white]
\tikzstyle{edge}=[draw,thick,decoration={markings,mark=at position 0.5 with {\arrow{Stealth}}},postaction=decorate]
\tikzstyle{sedge}=[draw,thick,decoration={markings,mark=at position 0.45 with {\arrow{Stealth}}},postaction=decorate]
\node[vertex,label=90:$1$] (v1) at (0, 0) {};
\node[vertex,label=90:$1$] (v2) at (1, 0) {};
\node[vertex,label=-90:$1$] (v3) at (1, -1) {};
\node[vertex,label=-90:$-1$] (v4) at (0, -1) {};
\path[edge] (v1) -- (v2);
\path[edge] (v2) -- (v3);
\path[sedge] (v3) -- (v1);
\path[edge] (v1) -- (v4);
\path[sedge] (v2) -- (v4);
\path[edge] (v3) -- (v4);
\end{tikzpicture}}
\quad
\subcaptionbox{}{
\begin{tikzpicture}[scale=1.0]
\tikzstyle{vertex}=[draw,circle,font=\scriptsize,minimum size=5pt,inner sep=1pt,fill=magenta!80!white]
\tikzstyle{edge}=[draw,thick,decoration={markings,mark=at position 0.5 with {\arrow{Stealth}}},postaction=decorate]
\tikzstyle{sedge}=[draw,thick,decoration={markings,mark=at position 0.45 with {\arrow{Stealth}}},postaction=decorate]
\tikzstyle{ledge}=[draw,thick,decoration={markings,mark=at position 0.65 with {\arrow{Stealth}}},postaction=decorate]
\node[vertex,label=180:$-1$] (v5) at (-1, 0) {};
\node[vertex,label=90:$-1$] (v1) at (1, 0) {};
\node[vertex,label={[yshift=-2pt]90:$1$}] (v6) at (0, -0.6) {};
\node[vertex,label={[yshift=2pt]-90:$1$}] (v3) at (0, -1.5) {};
\node[vertex,label=-90:$1$] (v0) at (-1, -2.1) {};
\node[vertex,label=-90:$-1$] (v4) at (1, -2.1) {};
\node[vertex,label=0:$-1$] (v2) at (-0.3, 0.7) {};
\path[ledge] (v3) -- (v4);
\path[ledge] (v3) -- (v6);
\path[ledge] (v0) -- (v3);
\path[ledge] (v0) -- (v4);
\path[ledge] (v4) -- (v1);
\path[ledge] (v4) -- (v6);
\path[ledge] (v1) -- (v5);
\path[ledge] (v1) -- (v6);
\path[ledge] (v6) -- (v0);
\path[ledge] (v6) -- (v5);
\path[ledge] (v5) -- (v0);
\path[ledge] (v5) -- (v2);
\end{tikzpicture}}
\quad
\subcaptionbox{$M_1(6)$\label{fig:exampleM1}}{
\begin{tikzpicture}[scale=1.2]
\tikzstyle{vertex}=[draw,circle,font=\scriptsize,minimum size=5pt,inner sep=1pt,fill=magenta!80!white]
\tikzstyle{edge}=[draw,thick,decoration={markings,mark=at position 0.5 with {\arrow{Stealth}}},postaction=decorate]
\tikzstyle{sedge}=[draw,thick,decoration={markings,mark=at position 0.45 with {\arrow{Stealth}}},postaction=decorate]
\tikzstyle{ledge}=[draw,thick,decoration={markings,mark=at position 0.65 with {\arrow{Stealth}}},postaction=decorate]
\node[vertex,label=90:$1$] (v0) at (90:1.2) {};
\node[vertex,label=180:$-1$] (v1) at ({90 + 60}:1.2) {};
\node[vertex,label=180:$1$] (v2) at ({90 + 60 * 2}:1.2) {};
\node[vertex,label=-90:$-1$] (v3) at ({90 + 60 * 3}:1.2) {};
\node[vertex,label=0:$1$] (v4) at ({90 + 60 * 4}:1.2) {};
\node[vertex,label=0:$-1$] (v5) at ({90 + 60 * 5}:1.2) {};
\path[edge] (v1) -- (v0);
\path[edge] (v2) -- (v1);
\path[edge] (v3) -- (v2);
\path[edge] (v4) -- (v3);
\path[edge] (v5) -- (v4);
\path[edge] (v0) -- (v5);
\path[edge] (v2) -- (v0);
\path[edge] (v3) -- (v1);
\path[edge] (v4) -- (v2);
\path[edge] (v5) -- (v3);
\path[edge] (v0) -- (v4);
\path[edge] (v1) -- (v5);
\end{tikzpicture}}
\quad
\subcaptionbox{$M_2(6)$\label{fig:exampleM2}}{
\begin{tikzpicture}[scale=1.2]
\tikzstyle{vertex}=[draw,circle,font=\scriptsize,minimum size=5pt,inner sep=1pt,fill=magenta!80!white]
\tikzstyle{edge}=[draw,thick,decoration={markings,mark=at position 0.5 with {\arrow{Stealth}}},postaction=decorate]
\tikzstyle{sedge}=[draw,thick,decoration={markings,mark=at position 0.45 with {\arrow{Stealth}}},postaction=decorate]
\tikzstyle{ledge}=[draw,thick,decoration={markings,mark=at position 0.65 with {\arrow{Stealth}}},postaction=decorate]
\node[vertex,label=90:$1$] (v0) at (90:1.2) {};
\node[vertex,label=180:$-1$] (v1) at ({90 + 60}:1.2) {};
\node[vertex,label=180:$1$] (v2) at ({90 + 60 * 2}:1.2) {};
\node[vertex,label=-90:$-1$] (v3) at ({90 + 60 * 3}:1.2) {};
\node[vertex,label=0:$1$] (v4) at ({90 + 60 * 4}:1.2) {};
\node[vertex,label=0:$-1$] (v5) at ({90 + 60 * 5}:1.2) {};
\path[edge] (v1) -- (v0);
\path[edge] (v2) -- (v1);
\path[edge] (v3) -- (v2);
\path[edge] (v4) -- (v3);
\path[edge] (v5) -- (v4);
\path[edge] (v0) -- (v5);
\path[edge] (v2) -- (v0);
\path[edge] (v1) -- (v3);
\path[edge] (v4) -- (v2);
\path[edge] (v3) -- (v5);
\path[edge] (v0) -- (v4);
\path[edge] (v5) -- (v1);
\end{tikzpicture}}

\subcaptionbox{\label{fig:exampleAmbi7}}{
\begin{tikzpicture}[scale=1.5]
\tikzstyle{vertex}=[draw,circle,font=\scriptsize,minimum size=5pt,inner sep=1pt,fill=magenta!80!white]
\tikzstyle{edge}=[draw,thick,decoration={markings,mark=at position 0.5 with {\arrow{Stealth}}},postaction=decorate]
\tikzstyle{sedge}=[draw,thick,decoration={markings,mark=at position 0.45 with {\arrow{Stealth}}},postaction=decorate]
\tikzstyle{ssedge}=[draw,thick,decoration={markings,mark=at position 0.4 with {\arrow{Stealth}}},postaction=decorate]
\tikzstyle{ledge}=[draw,thick,decoration={markings,mark=at position 0.65 with {\arrow{Stealth}}},postaction=decorate]
\tikzstyle{lledge}=[draw,thick,decoration={markings,mark=at position 0.8 with {\arrow{Stealth}}},postaction=decorate]
\node[vertex,label=180:$1$] (v0) at (180:1) {};
\node[vertex,label=0:$-2$] (v1) at (0:1) {};
\node[vertex,label=90:$-1$] (v5) at (60:1) {};
\node[vertex,label=90:$-1$] (v4) at (120:1) {};
\node[vertex,label=-90:$1$] (v3) at (-60:1) {};
\node[vertex,label=-90:$1$] (v6) at (-120:1) {};
\node[vertex,label=0:$1$] (v2) at (0, 0) {};
\path[edge] (v4) -- (v6);
\path[lledge] (v4) -- (v2);
\path[sedge] (v4) -- (v1);
\path[edge] (v5) -- (v4);
\path[ledge] (v5) -- (v0);
\path[edge] (v0) -- (v4);
\path[sedge] (v0) -- (v3);
\path[ssedge] (v2) -- (v5);
\path[ssedge] (v2) -- (v6);
\path[edge] (v1) -- (v5);
\path[edge] (v1) -- (v3);
\path[edge] (v6) -- (v0);
\path[edge] (v6) -- (v3);
\path[sedge] (v6) -- (v1);
\path[lledge] (v3) -- (v2);
\path[edge] (v3) -- (v5);
\end{tikzpicture}}
\qquad
\subcaptionbox{\label{fig:exampleBi7}}{
\begin{tikzpicture}[scale=1.4]
\tikzstyle{vertex}=[draw,circle,font=\scriptsize,minimum size=5pt,inner sep=1pt,fill=magenta!80!white]
\tikzstyle{edge}=[draw,thick,decoration={markings,mark=at position 0.5 with {\arrow{Stealth}}},postaction=decorate]
\tikzstyle{sedge}=[draw,thick,decoration={markings,mark=at position 0.45 with {\arrow{Stealth}}},postaction=decorate]
\tikzstyle{ledge}=[draw,thick,decoration={markings,mark=at position 0.65 with {\arrow{Stealth}}},postaction=decorate]
\node[vertex,label=90:$1$] (v2) at (90:1.2) {};
\node[vertex,label=180:$1$] (v0) at ({90 + 72}:1.2) {};
\node[vertex,label=180:$-1$] (v3) at ({90 + 72 * 2}:1.2) {};
\node[vertex,label=0:$-1$] (v1) at ({90 + 72 * 3}:1.2) {};
\node[vertex,label=0:$\mp 1$] (v4) at ({90 + 72 * 4}:1.2) {};
\path[edge] (v0) -- (v3);
\path[edge] (v0) -- (v2);
\path[edge] (v2) -- (v4);
\path[edge] (v4) -- (v1);
\path[edge] (v1) -- (v3);
\node[vertex,label={[yshift=-5]180:$1$}] (v5) at (-0.3, 0) {};
\node[vertex,label={[yshift=-5]0:$-1$}] (v6) at (0.3,0) {};
\path[edge] (v3) -- (v5);
\path[edge] (v3) -- (v6);
\path[edge] (v6) -- (v1);
\path[edge] (v1) -- (v5);
\path[edge] (v5) -- (v0);
\path[edge] (v6) -- (v4);
\path[edge] (v2) -- (v5);
\path[edge] (v6) -- (v2);
\path[ledge] (v5) -- (v6);
\path[ledge] (v6) to [bend right=30] (v0);
\path[sedge] (v4) to [bend right=30] (v5);
\end{tikzpicture}}
\caption{Examples of nut digraphs among \emph{oriented} graphs: (a) the smallest dextro-nut digraph, (b) a dextro-nut digraph with a leaf, (c)~\&~(d) the two smallest ambi-nut digraphs (labeled $M_1(6)$ and $M_2(6)$), (e) the unique ambi-nut digraph on
7 vertices, (f) one of the 20 bi-nut digraphs on 7 vertices that are not ambi-nuts. In all panels, vertex labels show the kernel vector. In case (f), the vectors from the kernel
and co-kernel differ in just one entry (labeled $\mp 1$, with $-1$ corresponding to the kernel vector).}
\label{fig:smallGeneral}
\end{figure}
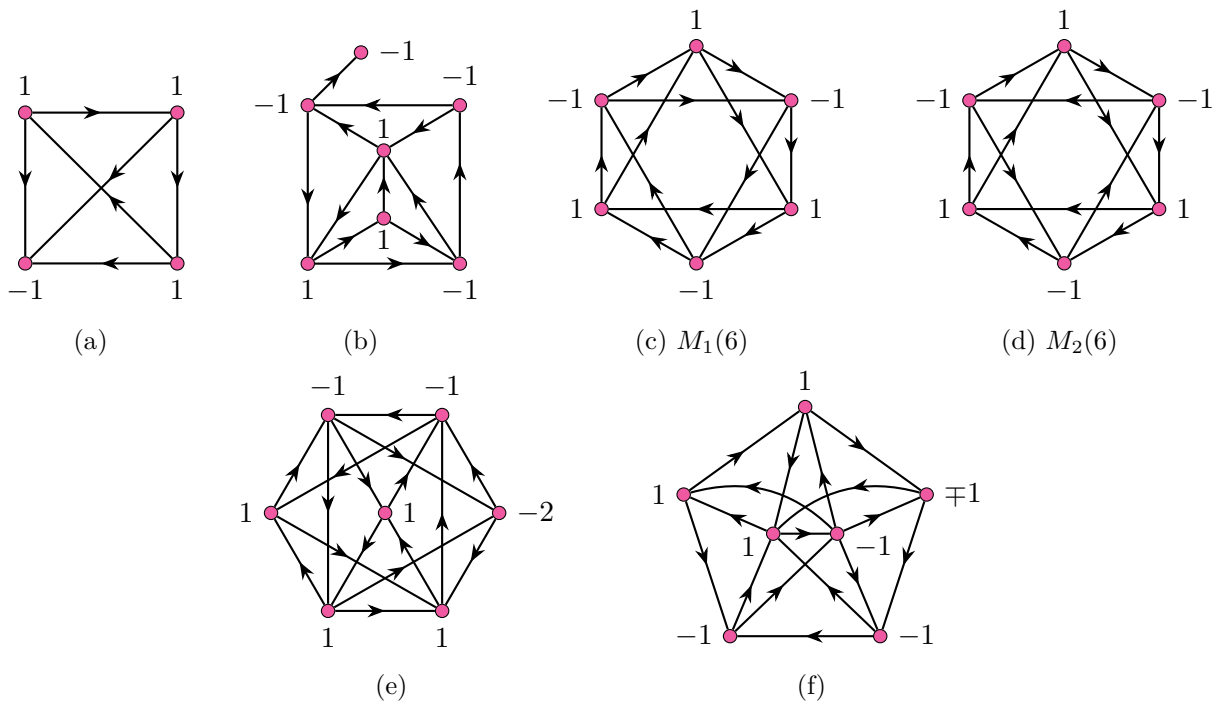

Table~\ref{tbl:enumQuartic} lists counts of dextro-, bi- and ambi-nut digraphs among $4$-regular oriented graphs on up to $11$ vertices.
Some of these digraphs are shown in Figure~\ref{fig:smallQuartic}. Note that the graphs (e) to (g) in Figure~\ref{fig:smallQuartic} all have the circulant $\Circ(8, \{1, 2\})$
as their underlying graph. (For definition, notation and properties of circulants see, e.g.~\cite{DamnjanStevan2022}.)
Circulant underlying graphs are also encountered in Figure~\ref{fig:smallGeneral}. 
This suggests definition of the following families.

\begin{definition}
For $k \in \{1, 2, 3\}$ and $n$ even, let $M_k(n)$ be a directed graph with the vertex set $\{0, 1, \ldots, n - 1\}$ and the arc set 
\begin{align*}
E(M_1(n)) & = \{ (i, i + 1) \mid 0 \leq i < n \} \cup \{ (i, i + 2) \mid 0 \leq i < n \}, \\
E(M_2(n)) & = \{ (i, i + 1) \mid 0 \leq i < n \} \cup \{ (i, i + 2) \colon 0 \leq i < n,\ i \text{ even} \} \cup \{ (i + 2, i) \mid 0 < i < n,\ i \text{ odd} \}, \\
E(M_3(n)) & = \{ (i, i + 1) \mid 0 \leq i < n \} \cup \{ (i + 2, i) \mid 0 \leq i < n \},
\end{align*}
where addition is done modulo $n$. (See Figures~\ref{fig:smallGeneral} and~\ref{fig:smallQuartic} for examples.)
\end{definition}

\noindent Note that all three digraphs $M_1(n), M_2(n)$ and $M_3(n)$ have the circulant $\Circ(n, \{1, 2\})$, the skeleton of the $\frac{n}{2}$-antiprism, as their underlying graph.

\begin{proposition}
\label{prof:maxineFamily}
The digraphs $M_1(n)$ and $M_2(n)$ are ambi-nut digraphs for every even $n \geq 6$. The graph $M_3(n)$ is an ambi-nut digraph for every even $n \geq 6$ that
satisfies $n \not\equiv 0 \pmod{6}$.
\end{proposition}

\begin{proof}
This is a straightforward consequence of local conditions~\eqref{eq:zerosum} and~\eqref{eq:zerosumn}. A vector that spans the (co-)kernel in all three cases is $\x(i) = (-1)^i$.
It is easy to verify that $\dim \Ker M_3(n) = 3$ if $n \equiv 0 \pmod{6}$ and $\dim \Ker M_3(n) = 1$ otherwise, whereas $\dim \Ker M_1(n) = \dim \Ker M_2(n) = 1$ for all even~$n$.
\end{proof}

\begin{table}[!htbp]
\centering
\begin{tabular}{rrrrrr}
\hline
$n$ & $|\mathcal{U}_n^4|$ & $|\mathcal{O}_n^4|$ 
& $|\mathcal{D}_n^4|$ & $|\mathcal{B}_n^4|$ & $|\mathcal{A}_n^4|$ \\
\hline\hline
5 & 1 & 12 
& 0 & 0 & 0 \\
6 & 1 & 112 
& 4 & 2 & 2 \\
7 & 2 & 1602 
& 9 & 0 & 0 \\
8 & 6 & 32263 
& 202 & 27 & 5 \\
9 & 16 & 748576 
& 2255 & 0 & 0 \\
10 & 59 & 19349594 
& 33034 & 2072 & 32 \\
11 & 265 & 548123668 
& 436947 & 0 & 0 \\
\hline
\end{tabular}
\caption{Enumeration of nut digraphs among $4$-regular oriented graphs on $n \leq 11$ vertices.}
\label{tbl:enumQuartic}
\end{table}

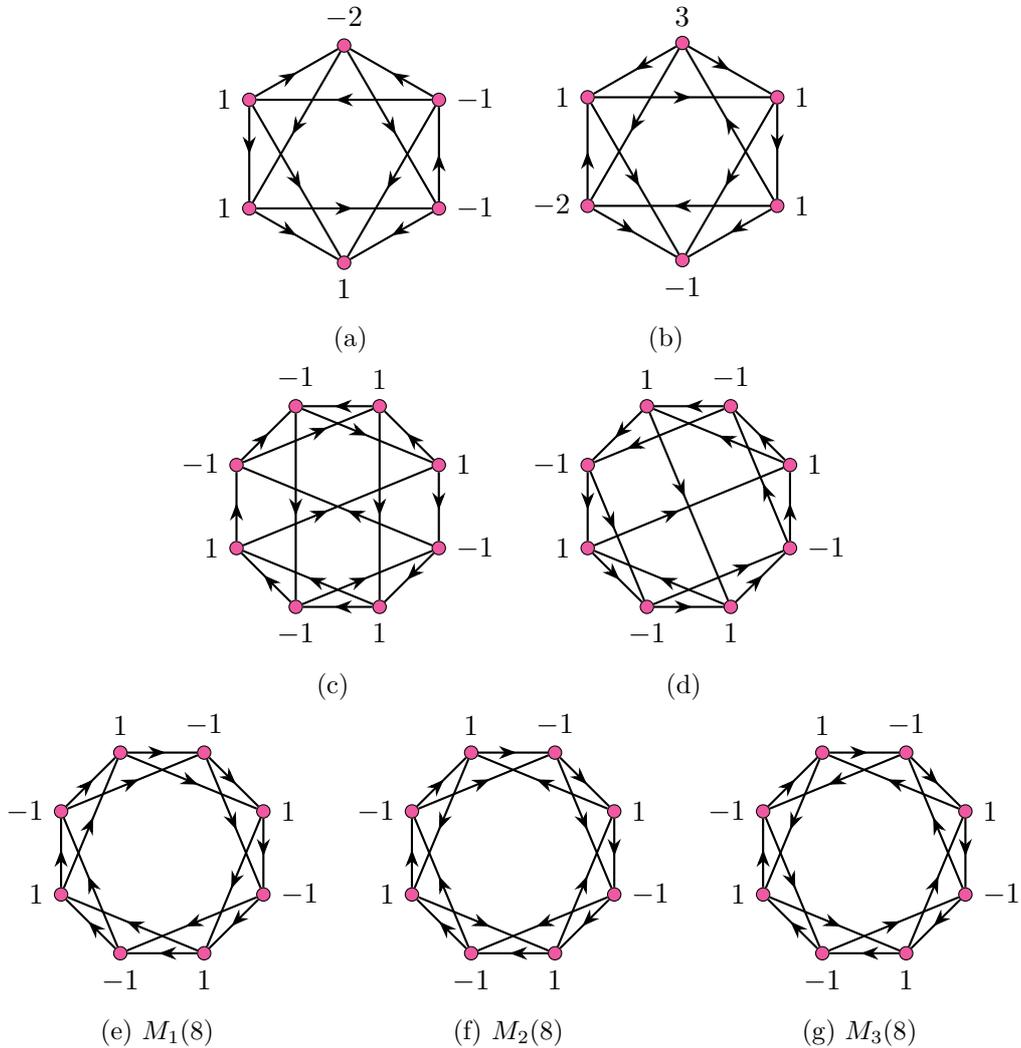
\begin{figure}[!htbp]
\centering
\subcaptionbox{\label{fig:exampleQ1}}{
\begin{tikzpicture}[scale=1.2]
\tikzstyle{vertex}=[draw,circle,font=\scriptsize,minimum size=5pt,inner sep=1pt,fill=magenta!80!white]
\tikzstyle{edge}=[draw,thick,decoration={markings,mark=at position 0.5 with {\arrow{Stealth}}},postaction=decorate]
\tikzstyle{sedge}=[draw,thick,decoration={markings,mark=at position 0.45 with {\arrow{Stealth}}},postaction=decorate]
\tikzstyle{ledge}=[draw,thick,decoration={markings,mark=at position 0.55 with {\arrow{Stealth}}},postaction=decorate]
\node[vertex,label=90:$-2$] (v0) at (90:1.2) {};
\node[vertex,label=180:$1$] (v1) at ({90 + 60}:1.2) {};
\node[vertex,label=180:$1$] (v2) at ({90 + 60 * 2}:1.2) {};
\node[vertex,label=-90:$1$] (v3) at ({90 + 60 * 3}:1.2) {};
\node[vertex,label=0:$-1$] (v4) at ({90 + 60 * 4}:1.2) {};
\node[vertex,label=0:$-1$] (v5) at ({90 + 60 * 5}:1.2) {};
\path[edge] (v1) -- (v0);
\path[edge] (v5) -- (v0);
\path[edge] (v1) -- (v2);
\path[edge] (v2) -- (v3);
\path[edge] (v4) -- (v3);
\path[edge] (v4) -- (v5);
\path[ledge] (v0) -- (v2);
\path[ledge] (v0) -- (v4);
\path[ledge] (v2) -- (v4);
\path[ledge] (v5) -- (v1);
\path[ledge] (v5) -- (v3);
\path[ledge] (v1) -- (v3);
\end{tikzpicture}}
\subcaptionbox{\label{fig:exampleQ2}}{
\begin{tikzpicture}[scale=1.2]
\tikzstyle{vertex}=[draw,circle,font=\scriptsize,minimum size=5pt,inner sep=1pt,fill=magenta!80!white]
\tikzstyle{edge}=[draw,thick,decoration={markings,mark=at position 0.5 with {\arrow{Stealth}}},postaction=decorate]
\tikzstyle{sedge}=[draw,thick,decoration={markings,mark=at position 0.45 with {\arrow{Stealth}}},postaction=decorate]
\tikzstyle{ledge}=[draw,thick,decoration={markings,mark=at position 0.55 with {\arrow{Stealth}}},postaction=decorate]
\node[vertex,label=90:$3$] (v0) at (90:1.2) {};
\node[vertex,label=180:$1$] (v1) at ({90 + 60}:1.2) {};
\node[vertex,label=180:$-2$] (v2) at ({90 + 60 * 2}:1.2) {};
\node[vertex,label=-90:$-1$] (v3) at ({90 + 60 * 3}:1.2) {};
\node[vertex,label=0:$1$] (v4) at ({90 + 60 * 4}:1.2) {};
\node[vertex,label=0:$1$] (v5) at ({90 + 60 * 5}:1.2) {};
\path[edge] (v0) -- (v1);
\path[edge] (v0) -- (v5);
\path[edge] (v2) -- (v1);
\path[edge] (v2) -- (v3);
\path[edge] (v4) -- (v3);
\path[edge] (v5) -- (v4);
\path[ledge] (v0) -- (v2);
\path[ledge] (v4) -- (v0);
\path[ledge] (v4) -- (v2);
\path[ledge] (v1) -- (v5);
\path[ledge] (v5) -- (v3);
\path[ledge] (v1) -- (v3);
\end{tikzpicture}}

\subcaptionbox{\label{fig:exampleQ3}}{
\begin{tikzpicture}[scale=1.2]
\tikzstyle{vertex}=[draw,circle,font=\scriptsize,minimum size=5pt,inner sep=1pt,fill=magenta!80!white]
\tikzstyle{edge}=[draw,thick,decoration={markings,mark=at position 0.5 with {\arrow{Stealth}}},postaction=decorate]
\tikzstyle{sedge}=[draw,thick,decoration={markings,mark=at position 0.45 with {\arrow{Stealth}}},postaction=decorate]
\tikzstyle{ledge}=[draw,thick,decoration={markings,mark=at position 0.55 with {\arrow{Stealth}}},postaction=decorate]
\tikzstyle{lledge}=[draw,thick,decoration={markings,mark=at position 0.65 with {\arrow{Stealth}}},postaction=decorate]
\node[vertex,label=0:$1$] (v6) at (22.5:1.2) {};
\node[vertex,label=90:$1$] (v0) at ({22.5 + 45}:1.2) {};
\node[vertex,label=90:$-1$] (v3) at ({22.5 + 45 * 2}:1.2) {};
\node[vertex,label=180:$-1$] (v5) at ({22.5 + 45 * 3}:1.2) {};
\node[vertex,label=180:$1$] (v2) at ({22.5 + 45 * 4}:1.2) {};
\node[vertex,label=-90:$-1$] (v7) at ({22.5 + 45 * 5}:1.2) {};
\node[vertex,label=-90:$1$] (v4) at ({22.5 + 45 * 6}:1.2) {};
\node[vertex,label=0:$-1$] (v1) at ({22.5 + 45 * 7}:1.2) {};
\path[ledge] (v0) -- (v3);
\path[ledge] (v0) -- (v4);
\path[ledge] (v1) -- (v4);
\path[sedge] (v1) -- (v5);
\path[ledge] (v2) -- (v5);
\path[sedge] (v2) -- (v6);
\path[ledge] (v3) -- (v7);
\path[edge] (v3) -- (v6);
\path[ledge] (v4) -- (v7);
\path[edge] (v4) -- (v2);
\path[lledge] (v5) -- (v0);
\path[ledge] (v5) -- (v3);
\path[ledge] (v6) -- (v0);
\path[ledge] (v6) -- (v1);
\path[edge] (v7) -- (v1);
\path[ledge] (v7) -- (v2);
\end{tikzpicture}}
\subcaptionbox{\label{fig:exampleQ4}}{
\begin{tikzpicture}[scale=1.2]
\tikzstyle{vertex}=[draw,circle,font=\scriptsize,minimum size=5pt,inner sep=1pt,fill=magenta!80!white]
\tikzstyle{edge}=[draw,thick,decoration={markings,mark=at position 0.5 with {\arrow{Stealth}}},postaction=decorate]
\tikzstyle{sedge}=[draw,thick,decoration={markings,mark=at position 0.45 with {\arrow{Stealth}}},postaction=decorate]
\tikzstyle{ledge}=[draw,thick,decoration={markings,mark=at position 0.55 with {\arrow{Stealth}}},postaction=decorate]
\tikzstyle{lledge}=[draw,thick,decoration={markings,mark=at position 0.75 with {\arrow{Stealth}}},postaction=decorate]
\node[vertex,label=0:$1$] (v5) at (22.5:1.2) {};
\node[vertex,label=90:$-1$] (v7) at ({22.5 + 45}:1.2) {};
\node[vertex,label=90:$1$] (v2) at ({22.5 + 45 * 2}:1.2) {};
\node[vertex,label=180:$-1$] (v4) at ({22.5 + 45 * 3}:1.2) {};
\node[vertex,label=180:$1$] (v0) at ({22.5 + 45 * 4}:1.2) {};
\node[vertex,label=-90:$-1$] (v3) at ({22.5 + 45 * 5}:1.2) {};
\node[vertex,label=-90:$1$] (v6) at ({22.5 + 45 * 6}:1.2) {};
\node[vertex,label=0:$-1$] (v1) at ({22.5 + 45 * 7}:1.2) {};
\path[ledge] (v0) -- (v3);
\path[sedge] (v0) -- (v5);
\path[ledge] (v1) -- (v5);
\path[sedge] (v1) -- (v7);
\path[ledge] (v2) -- (v4);
\path[sedge] (v2) -- (v6);
\path[lledge] (v3) -- (v1);
\path[ledge] (v3) -- (v6);
\path[ledge] (v4) -- (v0);
\path[sedge] (v4) -- (v3);
\path[edge] (v5) -- (v2);
\path[ledge] (v5) -- (v7);
\path[edge] (v6) -- (v0);
\path[ledge] (v6) -- (v1);
\path[ledge] (v7) -- (v2);
\path[lledge] (v7) -- (v4);
\end{tikzpicture}}

\subcaptionbox{$M_1(8)$\label{fig:exampleQ5}}{
\begin{tikzpicture}[scale=1.2]
\tikzstyle{vertex}=[draw,circle,font=\scriptsize,minimum size=5pt,inner sep=1pt,fill=magenta!80!white]
\tikzstyle{edge}=[draw,thick,decoration={markings,mark=at position 0.5 with {\arrow{Stealth}}},postaction=decorate]
\tikzstyle{sedge}=[draw,thick,decoration={markings,mark=at position 0.45 with {\arrow{Stealth}}},postaction=decorate]
\tikzstyle{ledge}=[draw,thick,decoration={markings,mark=at position 0.55 with {\arrow{Stealth}}},postaction=decorate]
\tikzstyle{lledge}=[draw,thick,decoration={markings,mark=at position 0.75 with {\arrow{Stealth}}},postaction=decorate]
\node[vertex,label=0:$1$] (v0) at (22.5:1.2) {};
\node[vertex,label=90:$-1$] (v4) at ({22.5 + 45}:1.2) {};
\node[vertex,label=90:$1$] (v6) at ({22.5 + 45 * 2}:1.2) {};
\node[vertex,label=180:$-1$] (v1) at ({22.5 + 45 * 3}:1.2) {};
\node[vertex,label=180:$1$] (v3) at ({22.5 + 45 * 4}:1.2) {};
\node[vertex,label=-90:$-1$] (v7) at ({22.5 + 45 * 5}:1.2) {};
\node[vertex,label=-90:$1$] (v5) at ({22.5 + 45 * 6}:1.2) {};
\node[vertex,label=0:$-1$] (v2) at ({22.5 + 45 * 7}:1.2) {};
\path[ledge] (v0) -- (v2);
\path[ledge] (v0) -- (v5);
\path[ledge] (v1) -- (v4);
\path[ledge] (v1) -- (v6);
\path[ledge] (v2) -- (v5);
\path[ledge] (v2) -- (v7);
\path[ledge] (v3) -- (v1);
\path[ledge] (v3) -- (v6);
\path[ledge] (v4) -- (v0);
\path[ledge] (v4) -- (v2);
\path[ledge] (v5) -- (v3);
\path[ledge] (v5) -- (v7);
\path[ledge] (v6) -- (v4);
\path[ledge] (v6) -- (v0);
\path[ledge] (v7) -- (v1);
\path[ledge] (v7) -- (v3);
\end{tikzpicture}}
\subcaptionbox{$M_2(8)$\label{fig:exampleQ6}}{
\begin{tikzpicture}[scale=1.2]
\tikzstyle{vertex}=[draw,circle,font=\scriptsize,minimum size=5pt,inner sep=1pt,fill=magenta!80!white]
\tikzstyle{edge}=[draw,thick,decoration={markings,mark=at position 0.5 with {\arrow{Stealth}}},postaction=decorate]
\tikzstyle{sedge}=[draw,thick,decoration={markings,mark=at position 0.45 with {\arrow{Stealth}}},postaction=decorate]
\tikzstyle{ledge}=[draw,thick,decoration={markings,mark=at position 0.55 with {\arrow{Stealth}}},postaction=decorate]
\tikzstyle{lledge}=[draw,thick,decoration={markings,mark=at position 0.75 with {\arrow{Stealth}}},postaction=decorate]
\node[vertex,label=0:$1$] (v0) at (22.5:1.2) {};
\node[vertex,label=90:$-1$] (v4) at ({22.5 + 45}:1.2) {};
\node[vertex,label=90:$1$] (v6) at ({22.5 + 45 * 2}:1.2) {};
\node[vertex,label=180:$-1$] (v1) at ({22.5 + 45 * 3}:1.2) {};
\node[vertex,label=180:$1$] (v3) at ({22.5 + 45 * 4}:1.2) {};
\node[vertex,label=-90:$-1$] (v7) at ({22.5 + 45 * 5}:1.2) {};
\node[vertex,label=-90:$1$] (v5) at ({22.5 + 45 * 6}:1.2) {};
\node[vertex,label=0:$-1$] (v2) at ({22.5 + 45 * 7}:1.2) {};
\path[ledge] (v0) -- (v2);
\path[ledge] (v1) -- (v4);
\path[ledge] (v1) -- (v6);
\path[ledge] (v2) -- (v5);
\path[ledge] (v2) -- (v7);
\path[ledge] (v3) -- (v1);
\path[ledge] (v4) -- (v0);
\path[ledge] (v4) -- (v2);
\path[ledge] (v5) -- (v7);
\path[ledge] (v6) -- (v4);
\path[ledge] (v7) -- (v1);
\path[ledge] (v7) -- (v3);
\path[ledge] (v0) -- (v6);
\path[ledge] (v6) -- (v3);
\path[ledge] (v3) -- (v5);
\path[ledge] (v5) -- (v0);
\end{tikzpicture}}
\subcaptionbox{$M_3(8)$\label{fig:exampleQ7}}{
\begin{tikzpicture}[scale=1.2]
\tikzstyle{vertex}=[draw,circle,font=\scriptsize,minimum size=5pt,inner sep=1pt,fill=magenta!80!white]
\tikzstyle{edge}=[draw,thick,decoration={markings,mark=at position 0.5 with {\arrow{Stealth}}},postaction=decorate]
\tikzstyle{sedge}=[draw,thick,decoration={markings,mark=at position 0.45 with {\arrow{Stealth}}},postaction=decorate]
\tikzstyle{ledge}=[draw,thick,decoration={markings,mark=at position 0.55 with {\arrow{Stealth}}},postaction=decorate]
\tikzstyle{lledge}=[draw,thick,decoration={markings,mark=at position 0.75 with {\arrow{Stealth}}},postaction=decorate]
\node[vertex,label=0:$1$] (v0) at (22.5:1.2) {};
\node[vertex,label=90:$-1$] (v4) at ({22.5 + 45}:1.2) {};
\node[vertex,label=90:$1$] (v6) at ({22.5 + 45 * 2}:1.2) {};
\node[vertex,label=180:$-1$] (v1) at ({22.5 + 45 * 3}:1.2) {};
\node[vertex,label=180:$1$] (v3) at ({22.5 + 45 * 4}:1.2) {};
\node[vertex,label=-90:$-1$] (v7) at ({22.5 + 45 * 5}:1.2) {};
\node[vertex,label=-90:$1$] (v5) at ({22.5 + 45 * 6}:1.2) {};
\node[vertex,label=0:$-1$] (v2) at ({22.5 + 45 * 7}:1.2) {};
\path[ledge] (v0) -- (v2);
\path[ledge] (v1) -- (v6);
\path[ledge] (v2) -- (v5);
\path[ledge] (v3) -- (v1);
\path[ledge] (v4) -- (v0);
\path[ledge] (v5) -- (v7);
\path[ledge] (v6) -- (v4);
\path[ledge] (v7) -- (v3);
\path[ledge] (v0) -- (v6);
\path[ledge] (v6) -- (v3);
\path[ledge] (v3) -- (v5);
\path[ledge] (v5) -- (v0);
\path[ledge] (v4) -- (v1);
\path[ledge] (v1) -- (v7);
\path[ledge] (v7) -- (v2);
\path[ledge] (v2) -- (v4);
\end{tikzpicture}}
\caption{Small $4$-regular nut digraphs. Panels (a) \& (b) show the two $4$-regular dextro-nut digraphs on $6$ vertices that are not bi-nut digraphs.
Both have the same underlying graph as $M_1(6)$, and since each of them
contains a sink, they are not isomorphic to $M_1(6)$ or $M_2(6)$. Panels (c) to (g) show 
the full set of  five $4$-regular ambi-nut digraphs on $8$ vertices.}
\label{fig:smallQuartic}
\end{figure}

Note that $\Circ(n, \{1, 2\})$ is itself a nut graph for every even $n \geq 6$ that satisfies $n \not\equiv 0 \pmod{6}$.
The graph $M_1(8)$ is an ambi-nut digraph and in fact its underlying graph is a nut graph. 

The following definition introduces three families that can be considered as directed analogues of Rose Window graphs~\cite{Stitch2008}.

\begin{definition}
For $k \in \{1, 2, 3\}$ and $n \geq 5$, let $D_k(n)$ be a directed graph with the vertex set $\{v_0, v_1, \ldots, v_{n - 1}\} \cup \{u_0, u_1, \ldots, u_{n - 1}\}$ and the arc set 
\begin{align*}
E(D_1(n)) & = \{ (v_i, v_{i + 1}), (u_i, u_{i + 2}), (v_i, u_i), (u_i, v_{i + 1}) \mid 0 \leq i < n \}, \\
E(D_2(n)) & = \{ (v_i, v_{i + 1}), (u_i, u_{i + 2}), (u_i, v_i), (v_{i + 1}, u_i) \mid 0 \leq i < n \},\\
E(D_3(n)) & = \{ (v_i, v_{i + 1}), (u_{i + 2}, u_i), (v_i, u_i), (u_i, v_{i + 1}) \mid 0 \leq i < n \}.
\end{align*}
where addition is done modulo $n$. See Figure~\ref{fig:daemons} for examples.
\end{definition}

\begin{figure}[!htbp]
\centering
\subcaptionbox{$D_1(5)$\label{fig:exampleSSS}}{
\begin{tikzpicture}[scale=1.3]
\tikzstyle{vertex}=[draw,circle,font=\scriptsize,minimum size=5pt,inner sep=1pt,fill=magenta!80!white]
\tikzstyle{edge}=[draw,thick,decoration={markings,mark=at position 0.5 with {\arrow{Stealth}}},postaction=decorate]
\tikzstyle{sedge}=[draw,thick,decoration={markings,mark=at position 0.45 with {\arrow{Stealth}}},postaction=decorate]
\tikzstyle{ledge}=[draw,thick,decoration={markings,mark=at position 0.90 with {\arrow{Stealth}}},postaction=decorate]
\node[vertex,label={[xshift=4pt]90:$1$}] (v0) at ({90 - 36}:0.7) {};
\node[vertex,label={[xshift=-4pt]90:$1$}] (v1) at ({90 + 72 - 36}:0.7) {};
\node[vertex,label=90:$1$] (v2) at ({90 + 2 * 72 - 36}:0.7) {};
\node[vertex,label=-90:$1$] (v3) at ({90 + 3 * 72 - 36}:0.7) {};
\node[vertex,label=90:$1$] (v4) at ({90 + 4 * 72 - 36}:0.7) {};
\node[vertex,label=90:$-1$] (u0) at (90:1.7) {};
\node[vertex,label={[xshift=-2pt]90:$-1$}] (u1) at ({90 + 72}:1.7) {};
\node[vertex,label=180:$-1$] (u2) at ({90 + 2 * 72}:1.7) {};
\node[vertex,label=0:$-1$] (u3) at ({90 + 3 * 72}:1.7) {};
\node[vertex,label=90:$-1$] (u4) at ({90 + 4 * 72}:1.7) {};
\path[edge] (u1) -- (u2);
\path[edge] (u2) -- (u3);
\path[edge] (u3) -- (u4);
\path[edge] (u4) -- (u0);
\path[edge] (u0) -- (u1);
\path[edge] (v0) -- (u0);
\path[edge] (u4) -- (v0) ;
\path[edge] (v4) -- (u4) ;
\path[edge] (u3) -- (v4);
\path[edge] (v3) -- (u3);
\path[edge] (u2) -- (v3) ;
\path[edge] (v2) -- (u2);
\path[edge] (u1) -- (v2) ;
\path[edge] (v1) -- (u1) ;
\path[edge] (u0) -- (v1);
\path[ledge] (v1) -- (v3);
\path[ledge] (v2) -- (v4);
\path[ledge] (v3) -- (v0);
\path[ledge] (v4) -- (v1);
\path[ledge] (v0) -- (v2);
\end{tikzpicture}
}
\subcaptionbox{$D_2(5)$\label{fig:exampleSDS}}{
\begin{tikzpicture}[scale=1.3]
\tikzstyle{vertex}=[draw,circle,font=\scriptsize,minimum size=5pt,inner sep=1pt,fill=magenta!80!white]
\tikzstyle{edge}=[draw,thick,decoration={markings,mark=at position 0.5 with {\arrow{Stealth}}},postaction=decorate]
\tikzstyle{sedge}=[draw,thick,decoration={markings,mark=at position 0.45 with {\arrow{Stealth}}},postaction=decorate]
\tikzstyle{ledge}=[draw,thick,decoration={markings,mark=at position 0.90 with {\arrow{Stealth}}},postaction=decorate]
\node[vertex,label={[xshift=4pt]90:$1$}] (v0) at ({90 - 36}:0.7) {};
\node[vertex,label={[xshift=-4pt]90:$1$}] (v1) at ({90 + 72 - 36}:0.7) {};
\node[vertex,label=90:$1$] (v2) at ({90 + 2 * 72 - 36}:0.7) {};
\node[vertex,label=-90:$1$] (v3) at ({90 + 3 * 72 - 36}:0.7) {};
\node[vertex,label=90:$1$] (v4) at ({90 + 4 * 72 - 36}:0.7) {};
\node[vertex,label=90:$-1$] (u0) at (90:1.7) {};
\node[vertex,label={[xshift=-2pt]90:$-1$}] (u1) at ({90 + 72}:1.7) {};
\node[vertex,label=180:$-1$] (u2) at ({90 + 2 * 72}:1.7) {};
\node[vertex,label=0:$-1$] (u3) at ({90 + 3 * 72}:1.7) {};
\node[vertex,label=90:$-1$] (u4) at ({90 + 4 * 72}:1.7) {};
\path[edge] (u1) -- (u2);
\path[edge] (u2) -- (u3);
\path[edge] (u3) -- (u4);
\path[edge] (u4) -- (u0);
\path[edge] (u0) -- (u1);
\path[edge] (u0) -- (v0);
\path[edge] (v0) -- (u4);
\path[edge] (u4) -- (v4);
\path[edge] (v4) -- (u3);
\path[edge] (u3) -- (v3);
\path[edge] (v3) -- (u2);
\path[edge] (u2) -- (v2);
\path[edge] (v2) -- (u1);
\path[edge] (u1) -- (v1);
\path[edge] (v1) -- (u0);
\path[ledge] (v1) -- (v3);
\path[ledge] (v2) -- (v4);
\path[ledge] (v3) -- (v0);
\path[ledge] (v4) -- (v1);
\path[ledge] (v0) -- (v2);
\end{tikzpicture}
}
\subcaptionbox{$D_3(5)$\label{fig:exampleSSD}}{
\begin{tikzpicture}[scale=1.3]
\tikzstyle{vertex}=[draw,circle,font=\scriptsize,minimum size=5pt,inner sep=1pt,fill=magenta!80!white]
\tikzstyle{edge}=[draw,thick,decoration={markings,mark=at position 0.5 with {\arrow{Stealth}}},postaction=decorate]
\tikzstyle{sedge}=[draw,thick,decoration={markings,mark=at position 0.45 with {\arrow{Stealth}}},postaction=decorate]
\tikzstyle{ledge}=[draw,thick,decoration={markings,mark=at position 0.90 with {\arrow{Stealth}}},postaction=decorate]
\node[vertex,label={[xshift=4pt]90:$1$}] (v0) at ({90 - 36}:0.7) {};
\node[vertex,label={[xshift=-4pt]90:$1$}] (v1) at ({90 + 72 - 36}:0.7) {};
\node[vertex,label=90:$1$] (v2) at ({90 + 2 * 72 - 36}:0.7) {};
\node[vertex,label=-90:$1$] (v3) at ({90 + 3 * 72 - 36}:0.7) {};
\node[vertex,label=90:$1$] (v4) at ({90 + 4 * 72 - 36}:0.7) {};
\node[vertex,label=90:$-1$] (u0) at (90:1.7) {};
\node[vertex,label={[xshift=-2pt]90:$-1$}] (u1) at ({90 + 72}:1.7) {};
\node[vertex,label=180:$-1$] (u2) at ({90 + 2 * 72}:1.7) {};
\node[vertex,label=0:$-1$] (u3) at ({90 + 3 * 72}:1.7) {};
\node[vertex,label=90:$-1$] (u4) at ({90 + 4 * 72}:1.7) {};
\path[edge] (u1) -- (u2);
\path[edge] (u2) -- (u3);
\path[edge] (u3) -- (u4);
\path[edge] (u4) -- (u0);
\path[edge] (u0) -- (u1);
\path[edge] (v0) -- (u0);
\path[edge] (u4) -- (v0) ;
\path[edge] (v4) -- (u4) ;
\path[edge] (u3) -- (v4);
\path[edge] (v3) -- (u3);
\path[edge] (u2) -- (v3) ;
\path[edge] (v2) -- (u2);
\path[edge] (u1) -- (v2) ;
\path[edge] (v1) -- (u1) ;
\path[edge] (u0) -- (v1);
\path[ledge] (v3) -- (v1);
\path[ledge] (v4) -- (v2);
\path[ledge] (v0) -- (v3);
\path[ledge] (v1) -- (v4);
\path[ledge] (v2) -- (v0);
\end{tikzpicture}
}
\caption{Small $4$-regular ambi-nut digraphs with a Rose Window graph as underlying graph.}
\label{fig:daemons}
\end{figure}
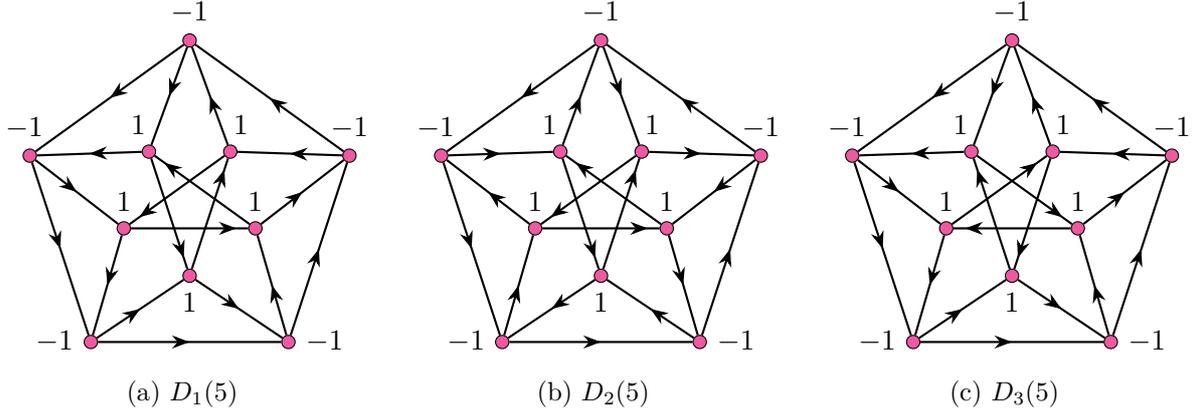

\begin{proposition}
The digraphs $D_1(n), D_2(n)$ and $D_3(n)$ are ambi-nut digraphs for every odd $n \geq 5$.
\end{proposition}

\begin{proof}
The proof is similar to that of Proposition~\ref{prof:maxineFamily}, i.e. local conditions~\eqref{eq:zerosum} and~\eqref{eq:zerosumn} can be applied. 
A vector that spans the (co-)kernel in all three cases is $\x(v_i) = 1$ and $\x(u_i) = -1$ for $0 \leq i < n$.
It is easy to verify that $\dim \Ker D_1(n) = \dim \Ker D_2(n) = \dim \Ker D_3(n) = 1$ if $n$ is odd. Moreover, $\dim \Ker D_1(n) = \dim \Ker D_3(n) = 2$ if $n$ is even.
In the case of $D_2$ and even $n$ we have that $\dim \Ker D_2(n) = 4$ if $n \equiv 0 \pmod{4}$ and $\dim \Ker D_2(n) = 2$ if $n \not\equiv 0 \pmod{4}$.
\end{proof}

Note that $D_1(n), D_2(n)$ and $D_3(n)$ share the same underlying graph, the Rose Window graph $R_n(1, 2)$ \cite{Stitch2008}. 
In \cite{NutOrbitPaper} it was proved that $R_n(1, 2)$, $n \geq 5$, is a nut graph if and only if $n \not\equiv 0 \pmod{3}$. 

Ambi-nut digraphs can also be obtained as cartesian products of certain directed graphs which themselves are not
ambi-nut digraphs. Let $\vec{C}_n$ denote the directed cycle on $n$ vertices. Then we have:

\begin{proposition}
Let $m, n \geq 3$ and let $G = \vec{C}_n \cart \vec{C}_m$. The following statements are equivalent:
\begin{enumerate}[label=(\roman*)]
\item
$G$ is a dextro-nut digraph;
\item
$G$ is a laevo-nut digraph;
\item
$G$ is an ambi-nut digraph;
\item
 $mn \equiv 0 \pmod{2}$ and $\gcd(m, n) = 1$.
\end{enumerate}
\end{proposition}

\begin{proof}
Recall that $V(G) = \ZZ_n \times \ZZ_m$ with arcs of the form $(i, j) \to (i + 1, j)$ and $(i, j) \to (i, j + 1)$ where addition
in the first coordinate is done modulo $n$ and in the second modulo $m$. 

Since $G \cong G^R$, statements (i) and (ii) are equivalent. Moreover, (iii) implies (i) and (ii). Hence, it suffices to show that (i) implies (iv) and that (iv) implies (iii).

Let $\x \in \ker G$.  From the local condition \eqref{eq:zerosum} we get that 
\begin{equation}
\label{eq:groupStuff}
\x(i, j) = -\x(i + 1, j - 1).
\end{equation}
Let $C$ be the subgroup of $\ZZ_n \times \ZZ_m$ generated by $(1, -1)$ and let $D$ be the subgroup generated by $(2, -2)$. 
Let $a = \x(0, 0)$. From \eqref{eq:groupStuff} it follows that $\x(v) = a$ for $v \in D$ and $\x(v) = -a$ for $v \in C \setminus D$.

Suppose now that (i) holds. If $D = C$ then $a = 0$, which implies that $G$ is not a dextro-nut digraph, a contradiction. Hence, $D$ is a subgroup of $C$ of index $2$.
Let $\y \colon V(G) \to \RR$ such that
$\y(v) = 1$ for $v \in D$, $\y(v) = - 1$ for $v \in C \setminus D$ and $\y(v) = 0$ for $v \in V(G) \setminus C$. Then $\y \in \ker G$. Since $G$ is a dextro-nut digraph it follows
that $V(G) = C$. In particular, $\ZZ_n \times \ZZ_m$ is a cyclic group generated by the element $(1, -1)$. But then $\gcd(m, n) = 1$, and since $V(G)$ contains a subgroup
of index $2$, $mn$ is even. This proves that (i) implies (iv).

Suppose that (iv) holds. Then $C = V(G)$ and $D$ is a subgroup of $V(G)$ of index $2$ and thus $\y$, as defined above, is a full vector such that 
$\y \in \Ker G$ and $\y \in \CoKer G$. Since $(1, -1)$ generates the whole $V(G)$, \eqref{eq:groupStuff} implies $\y$ is a unique vector from $\ker G$ up to scalar multiplication.
Since $\eta(G) = 1$,  $G$ is an ambi-nut digraph, as required.
\end{proof}

We have found several families of ambi-nut digraphs whose underlying graphs are nut graphs. However, we have also encounted ambi-nut digraphs,
whose underlying graphs are not nut graphs. As the next proposition will show, in
this case the underlying graphs  are necessarily \emph{core graphs}.
In fact, this holds more generally, for the inter-nut digraphs.

\begin{proposition}
\label{prop:ambiCore}
If $G$ is an inter-nut digraph then its underlying graph $\Gamma$ is a core graph.
\end{proposition}

\begin{proof}
Let $G$ be an inter-nut digraph. Then there exists a full vector $\x \in \Ker G \cap \CoKer G$. Then \eqref{eq:zerosum}
and \eqref{eq:zerosumn} hold. These imply that $\sum_{u\in G^+(v)}\x(u) + \sum_{u\in G^-(v)}\x(u) = 0$. The latter is precisely the local condition for the underlying
graph $\Gamma$. Therefore $\x \in \Ker \Gamma$, which implies that $\Gamma$ is a core graph.
\end{proof}

In this survey of computational results, we also enumerated nut tournaments, i.e. nut digraphs whose underlying graph is $K_n$; see Table~\ref{tbl:enumTournaments}.  
As one can see, there are no ambi-nut tournaments listed in the table. This is a direct consequence of Proposition~\ref{prop:ambiCore}, as $K_n$ is not a core graph,
and therefore $\mathcal{A}(K_n) = \emptyset$ for all $n$. 

\begin{table}[!htbp]
\centering
\begin{tabular}{rrrrrrr}
\hline
$n$ & $|\mathcal{O}(K_n)|$ & $|\mathcal{D}(K_n)|$ & $|\mathcal{B}(K_n)|$ & $|\mathcal{A}(K_n)|$ \\
\hline\hline
4 & 4 & 1 & 0 & 0 \\
5 & 12 & 0 & 0 & 0 \\
6 & 56 & 3 & 0 & 0 \\
7 & 456 & 9 & 0 & 0 \\
8 & 6880 & 119 & 0 & 0 \\
9 & 191536 & 2373 & 10 & 0 \\
10 & 9733056 & 90782 & 567 & 0 \\
11 & 903753248 & 5918592 & 26629 & 0 \\
\hline
\end{tabular}
\caption{Enumeration of nut digraphs among tournaments on $n \leq 11$ vertices.}
\label{tbl:enumTournaments}
\end{table}

Proposition~\ref{prop:ambiCore} also suggests an improved strategy of enumerating ambi-nut digraphs. For a given order $n$, we first find all core
graphs on order $n$. Each core graph $\Gamma$ gives rise to a collection of oriented digraphs (for which $\Gamma$ is the underlying graph).
We can further restrict consideration to digraphs $G$ with $\delta^-(G) \ge 2$ and $\delta^+(G) \ge 2$ (see Lemma~\ref{lem:degreeBounds} in Section~\ref{sec:basic}).
Consequently, it is enough to generate core graphs $\Gamma$ with $\delta(\Gamma) \geq 4$.
Tables~\ref{tbl:ambiNutsImproved} and~\ref{tbl:quarticAmbiNutsImproved} show the results of such a search for ambi-nut digraphs based on general oriented cores and
$4$-regular oriented cores, respectively. The search on $4$-regular cores of odd order was performed only to illustrate the fact that $4$-regular cores of odd order may exist, 
even though ambi-nut digraphs of odd order do not.

Note that many core graphs
do not produce ambi-nut digraphs. See, for example, the graph $\Gamma^\dagger$ in Figure~\ref{fig:badAmbiCore}. As $\Gamma^\dagger$ is
a nut graph, the null space eigenvector is uniquely determined (up to scalar multiplication). The orientation of
edges in any ambi-nut digraph that has $\Gamma^\dagger$ as
underlying graph must  be consistent with this null space eigenvector.
This observation could be used to filter core graphs, and also to carry out substantial pruning of branches in the process of
generation of directed graphs.

\begin{table}[!htb]
\centering
\begin{tabular}{rrrrrrr}
\hline
$n$ & Core & Oriented & $|\mathcal{A}_n|$ & Good \\
\hline\hline
6 & 1 & 4 & 2 & 1 \\
7 & 1 & 26 & 1 & 1 \\
8 & 13 & 20958 & 104 & 10 \\
9 & 117 & 16677343 & 3371 & 68 \\
10 & 5299 & 65740041126 & 1404682 & 2544 \\
\hline
\end{tabular}
\caption{Enumeration of ambi-nut digraphs on $n \leq 10$ vertices by the method based on Proposition~\ref{prop:ambiCore}.
The column labelled `Core' gives the number of core graphs $\Gamma$ with $\delta(\Gamma) \geq 4$. The
`Oriented' column gives the number of oriented graphs $G$ that were obtained from the subset  satisfying the condition $\delta^-(G) \ge 2$ and $\delta^+(G) \ge 2$.
The column `Good' counts the graphs from the `Core' set that 
produce at least one ambi-nut digraph of the set $\mathcal{A}_n$.}
\label{tbl:ambiNutsImproved}
\end{table}

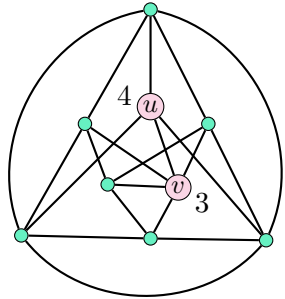
\begin{figure}[!htb]
\centering
\begin{tikzpicture}[scale=0.4]
\definecolor{aquamarine}{RGB}{102,241,194}
\tikzstyle{vertex}=[draw,circle,font=\small,minimum size=5pt,inner sep=1pt,fill=magenta!20!white]
\tikzstyle{edge}=[draw,thick]
\node[vertex,fill=aquamarine] (v0) at (-2.01, -1.01) {};
\node[vertex,label={[xshift=-2pt,yshift=4pt]-30:$3$}] (v1) at (0.31, -1.11) {$v$};
\node[vertex,fill=aquamarine] (v2) at (-2.757142857142857, 1) {};
\node[vertex,fill=aquamarine] (v3) at (1.30, 1) {};
\node[vertex,fill=aquamarine] (v4) at (-0.60, -2.80) {};
\node[vertex,fill=aquamarine] (v5) at (3.20, -2.86) {};
\node[vertex,label={[xshift=2pt,yshift=4pt]180:$4$}] (v6) at (-0.60, 1.57) {$u$};
\node[vertex,fill=aquamarine] (v7) at (-4.85, -2.70) {};
\node[vertex,fill=aquamarine] (v8) at (-0.60, 4.79) {};
\path[edge] (v0) -- (v1);
\path[edge] (v0) -- (v2);
\path[edge] (v0) -- (v3);
\path[edge] (v0) -- (v4);
\path[edge] (v1) -- (v4);
\path[edge] (v1) -- (v3);
\path[edge] (v1) -- (v2);
\path[edge] (v1) -- (v6);
\path[edge] (v6) -- (v5);
\path[edge] (v6) -- (v7);
\path[edge] (v6) -- (v8);
\path[edge] (v2) -- (v7);
\path[edge] (v2) -- (v8);
\path[edge] (v4) -- (v7);
\path[edge] (v4) -- (v5);
\path[edge] (v3) -- (v8);
\path[edge] (v3) -- (v5);
\path[edge] (v5) to[bend right=50] (v8);
\path[edge] (v8) to[bend right=50] (v7);
\path[edge] (v7) to[bend right=50] (v5);
\end{tikzpicture}
\caption{A `Bad' core graph, i.e., a nut graph $\Gamma^\dagger$ that is not the underlying graph of any ambi-nut digraph. 
The vertices labeled $u$ and $v$ carry respective entries $4$ and $3$ in the kernel eigenvector, while the rest carry entry $-1$. 
In a putative ambi-nut digraph obtained by orienting edges of $\Gamma^\dagger$,
vertex $u$ has $d^- (u) = d^+ (u) = 2$. From the entries of the kernel eigenvector, it is clear that this condition cannot be satisfied.}
\label{fig:badAmbiCore}
\end{figure}

\begin{table}[!htb]
\centering
\begin{tabular}{rrrrrrr}
\hline
$n$ & Core & Oriented & $|\mathcal{A}_n|$ & Good \\
\hline\hline
5 & 0 & 0 & 0 & 0 \\ 
6 & 1 & 4 & 2 & 1 \\
7 & 0 & 0 & 0 & 0 \\
8 & 5 & 47 & 5 & 3 \\
9 & 0 & 0 & 0 & 0 \\
10 & 21 & 1645 & 32 & 16 \\
11 & 0 & 0 & 0 & 0 \\
12 & 446 & 146371 & 860 & 225 \\
13 & 0 & 0 & 0 & 0 \\
14 & 20794 & ? & ? & ? \\
15 & 4 & 4945 & 0 & 0 \\
\hline
\end{tabular}
\caption{Enumeration of $4$-regular ambi-nut digraphs on $n \leq 15$ vertices by the method based on Proposition~\ref{prop:ambiCore}. The column `Core' gives the number of $4$-regular core graphs. The
column `Oriented' gives the number of oriented graphs $G$ that were obtained from the subset  that satisfy $\delta^-(G) = 2$ and $\delta^+(G) = 2$.
The column `Good' counts the graphs from `Core' that 
produce at least one ambi-nut digraph of the set $\mathcal{A}_n$.}
\label{tbl:quarticAmbiNutsImproved}
\end{table}

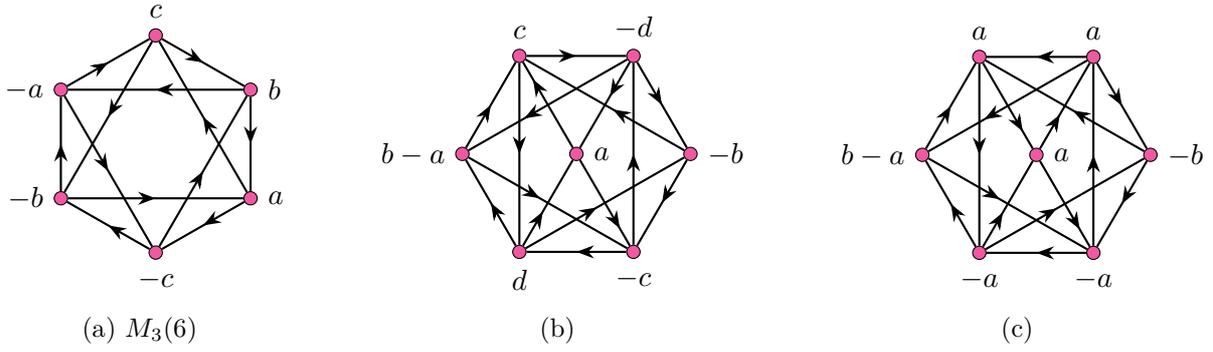
\begin{figure}[!htb]
\centering
\subcaptionbox{$M_3(6)$\label{fig:exampleM3_6}}{
\begin{tikzpicture}[scale=1.2]
\tikzstyle{vertex}=[draw,circle,font=\scriptsize,minimum size=5pt,inner sep=1pt,fill=magenta!80!white]
\tikzstyle{edge}=[draw,thick,decoration={markings,mark=at position 0.5 with {\arrow{Stealth}}},postaction=decorate]
\tikzstyle{sedge}=[draw,thick,decoration={markings,mark=at position 0.45 with {\arrow{Stealth}}},postaction=decorate]
\tikzstyle{ledge}=[draw,thick,decoration={markings,mark=at position 0.65 with {\arrow{Stealth}}},postaction=decorate]
\node[vertex,label=90:$c$] (v0) at (90:1.2) {};
\node[vertex,label=180:$-a$] (v1) at ({90 + 60}:1.2) {};
\node[vertex,label=180:$-b$] (v2) at ({90 + 60 * 2}:1.2) {};
\node[vertex,label=-90:$-c$] (v3) at ({90 + 60 * 3}:1.2) {};
\node[vertex,label=0:$a$] (v4) at ({90 + 60 * 4}:1.2) {};
\node[vertex,label=0:$b$] (v5) at ({90 + 60 * 5}:1.2) {};
\path[edge] (v1) -- (v0);
\path[edge] (v2) -- (v1);
\path[edge] (v3) -- (v2);
\path[edge] (v4) -- (v3);
\path[edge] (v5) -- (v4);
\path[edge] (v0) -- (v5);
\path[edge] (v0) -- (v2);
\path[edge] (v1) -- (v3);
\path[edge] (v2) -- (v4);
\path[edge] (v3) -- (v5);
\path[edge] (v4) -- (v0);
\path[edge] (v5) -- (v1);
\end{tikzpicture}}
\qquad
\subcaptionbox{\label{fig:example_core4}}{
\begin{tikzpicture}[scale=1.5]
\tikzstyle{vertex}=[draw,circle,font=\scriptsize,minimum size=5pt,inner sep=1pt,fill=magenta!80!white]
\tikzstyle{edge}=[draw,thick,decoration={markings,mark=at position 0.5 with {\arrow{Stealth}}},postaction=decorate]
\tikzstyle{sedge}=[draw,thick,decoration={markings,mark=at position 0.45 with {\arrow{Stealth}}},postaction=decorate]
\tikzstyle{ssedge}=[draw,thick,decoration={markings,mark=at position 0.4 with {\arrow{Stealth}}},postaction=decorate]
\tikzstyle{ledge}=[draw,thick,decoration={markings,mark=at position 0.65 with {\arrow{Stealth}}},postaction=decorate]
\tikzstyle{lledge}=[draw,thick,decoration={markings,mark=at position 0.8 with {\arrow{Stealth}}},postaction=decorate]
\node[vertex,label=180:$b-a$] (v0) at (180:1) {};
\node[vertex,label=0:$-b$] (v1) at (0:1) {};
\node[vertex,label=90:$-d$] (v5) at (60:1) {};
\node[vertex,label=90:$c$] (v4) at (120:1) {};
\node[vertex,label=-90:$-c$] (v3) at (-60:1) {};
\node[vertex,label=-90:$d$] (v6) at (-120:1) {};
\node[vertex,label=0:$a$] (v2) at (0, 0) {};
\path[edge] (v4) -- (v6);
\path[lledge] (v2) -- (v4);
\path[sedge] (v1) -- (v4);
\path[edge] (v4) -- (v5);
\path[ledge] (v5) -- (v0);
\path[edge] (v0) -- (v4);
\path[sedge] (v0) -- (v3);
\path[ssedge] (v5) -- (v2);
\path[ssedge] (v6) -- (v2);
\path[edge] (v5) -- (v1);
\path[edge] (v1) -- (v3);
\path[edge] (v6) -- (v0);
\path[edge] (v3) -- (v6);
\path[sedge] (v6) -- (v1);
\path[lledge] (v2) -- (v3);
\path[edge] (v3) -- (v5);
\end{tikzpicture}}
\qquad
\subcaptionbox{\label{fig:example_core2}}{
\begin{tikzpicture}[scale=1.5]
\tikzstyle{vertex}=[draw,circle,font=\scriptsize,minimum size=5pt,inner sep=1pt,fill=magenta!80!white]
\tikzstyle{edge}=[draw,thick,decoration={markings,mark=at position 0.5 with {\arrow{Stealth}}},postaction=decorate]
\tikzstyle{sedge}=[draw,thick,decoration={markings,mark=at position 0.45 with {\arrow{Stealth}}},postaction=decorate]
\tikzstyle{ssedge}=[draw,thick,decoration={markings,mark=at position 0.4 with {\arrow{Stealth}}},postaction=decorate]
\tikzstyle{ledge}=[draw,thick,decoration={markings,mark=at position 0.65 with {\arrow{Stealth}}},postaction=decorate]
\tikzstyle{lledge}=[draw,thick,decoration={markings,mark=at position 0.8 with {\arrow{Stealth}}},postaction=decorate]
\node[vertex,label=180:$b-a$] (v0) at (180:1) {};
\node[vertex,label=0:$-b$] (v1) at (0:1) {};
\node[vertex,label=90:$a$] (v5) at (60:1) {};
\node[vertex,label=90:$a$] (v4) at (120:1) {};
\node[vertex,label=-90:$-a$] (v3) at (-60:1) {};
\node[vertex,label=-90:$-a$] (v6) at (-120:1) {};
\node[vertex,label=0:$a$] (v2) at (0, 0) {};
\path[edge] (v4) -- (v6);
\path[lledge] (v4) -- (v2);
\path[sedge] (v1) -- (v4);
\path[edge] (v5) -- (v4);
\path[ledge] (v5) -- (v0);
\path[edge] (v0) -- (v4);
\path[sedge] (v0) -- (v3);
\path[ssedge] (v2) -- (v5);
\path[ssedge] (v6) -- (v2);
\path[edge] (v5) -- (v1);
\path[edge] (v1) -- (v3);
\path[edge] (v6) -- (v0);
\path[edge] (v3) -- (v6);
\path[sedge] (v6) -- (v1);
\path[lledge] (v2) -- (v3);
\path[edge] (v3) -- (v5);
\end{tikzpicture}}
\caption{Small examples of ambi-core digraphs with nullity strictly greater than $1$. 
Letters $a, b, c, \ldots$ are used to indicate a choice of independent parameters spanning the nullspace of each digraph.
The digraph $M_3(6)$ has nullity $3$ and
is the unique such digraph on $6$ vertices. There are four such digraphs on $7$ vertices. They all have the same underlying
graph. One (b) has nullity $4$, and the three others, one of which is shown in (c), have nullity $2$. Note that the ambi-nut digraph in
Figure~\ref{fig:smallGeneral}(d) has the same underlying graph.}
\label{fig:ambiCores}
\end{figure}

Furthermore, the notion of a core graph extends naturally to digraphs as follows. A digraph $G$ is dextro-core (resp.\ laevo-core) if $\Ker G$ (resp.\ $\CoKer G$) contains a full vector. 
A digraph $G$ is bi-core if it is both dextro-core and laevo-core. A digraph $G$ is ambi-core if $\Ker G = \CoKer G$ and $\Ker G$ contains a full vector. A digraph $G$ is 
inter-core if $\Ker G \cap \CoKer G$ contains a full vector. In Figure~\ref{fig:ambiCores} we show some small examples
of ambi-core digraphs with nullity greater than $1$.

The searches carried out to produce Tables~\ref{tbl:enumGeneral}--\ref{tbl:quarticAmbiNutsImproved} were deliberately limited to oriented graphs. 
If this restriction is lifted, ambi-nut digraphs that contain pairs of oppositely oriented arcs appear from small order.
See Figure~\ref{fig:nonOrientedAmbis} for examples. On $4$ vertices, there exists one ambi-nut digraph that is not an
oriented graph. On $6$ vertices, there are 14 ambi-nut digraphs that are not oriented graphs. Three of them are shown in the figure.
Of course, every (undirected) graph is in fact a directed graph where each edge is now viewed as a pair of oppositely oriented arcs.
Thus, every (undirected) nut graph is in fact an ambi-nut digraph whose arcs all appear in pairs.

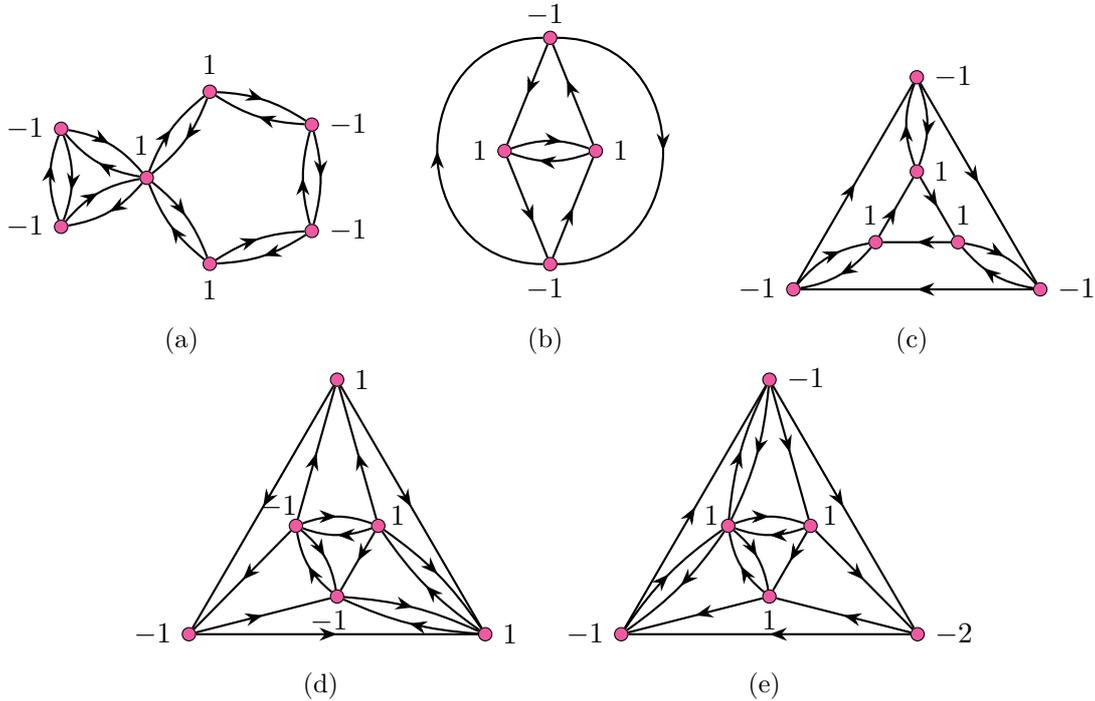
\begin{figure}[!htbp]
\centering
\subcaptionbox{}{
\begin{tikzpicture}[scale=1.0]
\tikzstyle{vertex}=[draw,circle,font=\scriptsize,minimum size=5pt,inner sep=1pt,fill=magenta!80!white]
\tikzstyle{edge}=[draw,thick,decoration={markings,mark=at position 0.5 with {\arrow{Stealth}}},postaction=decorate]
\tikzstyle{sedge}=[draw,thick,decoration={markings,mark=at position 0.45 with {\arrow{Stealth}}},postaction=decorate]
\tikzstyle{ledge}=[draw,thick,decoration={markings,mark=at position 0.65 with {\arrow{Stealth}}},postaction=decorate]
\node[vertex,label={[xshift=-2pt,yshift=2pt]90:$1$}] (v6) at ({180}:1.2) {};
\node[vertex,label=-90:$1$] (v3) at ({180 + 72}:1.2) {};
\node[vertex,label=0:$-1$] (v0) at ({180 + 72 * 2}:1.2) {};
\node[vertex,label=0:$-1$] (v5) at ({180 + 72 * 3}:1.2) {};
\node[vertex,label=90:$1$] (v2) at ({180 + 72 * 4}:1.2) {};
\node[vertex,label=180:$-1$] (v1) at ($ (v6) + (150:1.3) $) {};
\node[vertex,label=180:$-1$] (v4) at ($ (v6) + (210:1.3) $) {};
\path[ledge] (v1) to[bend left=20] (v4);
\path[ledge] (v4) to[bend left=20] (v1);
\path[edge] (v1) to[bend left=20] (v6);
\path[edge] (v6) to[bend left=20] (v1);
\path[edge] (v6) to[bend left=20] (v4);
\path[edge] (v4) to[bend left=20] (v6);
\path[edge] (v6) to[bend left=15] (v3);
\path[edge] (v3) to[bend left=15] (v6);
\path[edge] (v2) to[bend left=15] (v6);
\path[edge] (v6) to[bend left=15] (v2);
\path[edge] (v2) to[bend left=15] (v5);
\path[edge] (v5) to[bend left=15] (v2);
\path[edge] (v0) to[bend left=15] (v3);
\path[edge] (v3) to[bend left=15] (v0);
\path[edge] (v0) to[bend left=15] (v5);
\path[edge] (v5) to[bend left=15] (v0);
\end{tikzpicture}}
\subcaptionbox{}{
\begin{tikzpicture}[scale=1.5]
\tikzstyle{vertex}=[draw,circle,font=\scriptsize,minimum size=5pt,inner sep=1pt,fill=magenta!80!white]
\tikzstyle{edge}=[draw,thick,decoration={markings,mark=at position 0.5 with {\arrow{Stealth}}},postaction=decorate]
\tikzstyle{ledge}=[draw,thick,decoration={markings,mark=at position 0.65 with {\arrow{Stealth}}},postaction=decorate]
\node[vertex,label={[xshift=-2pt,yshift=-2pt]90:$-1$}] (v0) at (0, 1) {};
\node[vertex,label={[xshift=-2pt,yshift=1pt]-90:$-1$}] (v1) at (0, -1) {};
\node[vertex,label=180:$1$] (v2) at (-0.4, 0) {};
\node[vertex,label=0:$1$] (v3) at (0.4, 0) {};
\path[edge] (v0) -- (v2);
\path[ledge] (v2) -- (v1);
\path[ledge] (v3) -- (v0);
\path[edge] (v1) -- (v3);
\path[ledge] (v2) to [bend left=20] (v3);
\path[ledge] (v3) to [bend left=20] (v2);
\path[edge] (v0) .. controls ($ (v0) + (1.3, 0) $) and ($ (v1)  + (1.3, 0) $) .. (v1);
\path[edge] (v1) .. controls ($ (v1) + (-1.3, 0) $) and ($ (v0) + (-1.3, 0) $) .. (v0);
\end{tikzpicture}
}
\subcaptionbox{}{
\begin{tikzpicture}[scale=1.25] 
\tikzstyle{vertex}=[draw,circle,font=\scriptsize,minimum size=5pt,inner sep=1pt,fill=magenta!80!white]
\tikzstyle{edge}=[draw,thick,decoration={markings,mark=at position 0.5 with {\arrow{Stealth}}},postaction=decorate]
\tikzstyle{ledge}=[draw,thick,decoration={markings,mark=at position 0.65 with {\arrow{Stealth}}},postaction=decorate]
\node[vertex,label={0:$1$}] (v0) at (90:0.5) {};
\node[vertex,label={[xshift=-2pt,yshift=0pt]90:$1$}] (v1) at (210:0.5) {};
\node[vertex,label={[xshift=2pt,yshift=0pt]90:$1$}] (v2) at (-30:0.5) {};
\node[vertex,label=0:$-1$] (v3) at (90:1.5) {};
\node[vertex,label=180:$-1$] (v4) at (210:1.5) {};
\node[vertex,label=0:$-1$] (v5) at (-30:1.5) {};
\path[edge] (v1) -- (v0);
\path[edge] (v0) -- (v2);
\path[edge] (v2) -- (v1);
\path[edge] (v4) -- (v3);
\path[edge] (v3) -- (v5);
\path[edge] (v5) -- (v4);
\path[edge] (v0) to [bend left=20] (v3);
\path[ledge] (v3) to [bend left=20] (v0);
\path[edge] (v1) to [bend left=20] (v4);
\path[ledge] (v4) to [bend left=20] (v1);
\path[edge] (v2) to [bend left=20] (v5);
\path[ledge] (v5) to [bend left=20] (v2);
\end{tikzpicture}
}

\subcaptionbox{}{
\begin{tikzpicture}[scale=1.25]
\tikzstyle{vertex}=[draw,circle,font=\scriptsize,minimum size=5pt,inner sep=1pt,fill=magenta!80!white]
\tikzstyle{edge}=[draw,thick,decoration={markings,mark=at position 0.5 with {\arrow{Stealth}}},postaction=decorate]
\tikzstyle{sedge}=[draw,thick,decoration={markings,mark=at position 0.45 with {\arrow{Stealth}}},postaction=decorate]
\tikzstyle{ledge}=[draw,thick,decoration={markings,mark=at position 0.65 with {\arrow{Stealth}}},postaction=decorate]
\node[vertex,label=0:$1$] (v5) at (90:1.8) {};
\node[vertex,label={[xshift=7pt,yshift=7pt]180:$-1$}] (v3) at ({90 + 60}:0.5) {};
\node[vertex,label=180:$-1$] (v1) at ({90 + 60 * 2}:1.8) {};
\node[vertex,label={[xshift=-3pt,yshift=1pt]-90:$-1$}] (v4) at ({90 + 60 * 3}:0.5) {};
\node[vertex,label=0:$1$] (v2) at ({90 + 60 * 4}:1.8) {};
\node[vertex,label={[xshift=-2pt,yshift=4pt]0:$1$}] (v0) at ({90 + 60 * 5}:0.5) {};
\path[edge] (v3) -- (v5);
\path[edge] (v0) -- (v5);
\path[edge] (v5) -- (v1);
\path[edge] (v5) -- (v2);
\path[edge] (v1) -- (v2);
\path[edge] (v1) -- (v4);
\path[edge] (v3) -- (v1);
\path[edge] (v0) -- (v4);
\path[edge] (v3) to [bend left=20] (v0);
\path[edge] (v0) to [bend left=20] (v3);
\path[edge] (v3) to [bend left=20] (v4);
\path[edge] (v4) to [bend left=20] (v3);
\path[edge] (v0) to [bend left=10] (v2);
\path[edge] (v2) to [bend left=10] (v0);
\path[edge] (v4) to [bend left=10] (v2);
\path[edge] (v2) to [bend left=10] (v4);
\end{tikzpicture}
}
\subcaptionbox{}{
\begin{tikzpicture}[scale=1.25]
\tikzstyle{vertex}=[draw,circle,font=\scriptsize,minimum size=5pt,inner sep=1pt,fill=magenta!80!white]
\tikzstyle{edge}=[draw,thick,decoration={markings,mark=at position 0.5 with {\arrow{Stealth}}},postaction=decorate]
\tikzstyle{sedge}=[draw,thick,decoration={markings,mark=at position 0.45 with {\arrow{Stealth}}},postaction=decorate]
\tikzstyle{ledge}=[draw,thick,decoration={markings,mark=at position 0.65 with {\arrow{Stealth}}},postaction=decorate]
\node[vertex,label=0:$-1$] (v3) at (90:1.8) {};
\node[vertex,label={[xshift=3pt,yshift=4pt]180:$1$}] (v0) at ({90 + 60}:0.5) {};
\node[vertex,label=180:$-1$] (v4) at ({90 + 60 * 2}:1.8) {};
\node[vertex,label={[xshift=0pt,yshift=1pt]-90:$1$}] (v2) at ({90 + 60 * 3}:0.5) {};
\node[vertex,label=0:$-2$] (v1) at ({90 + 60 * 4}:1.8) {};
\node[vertex,label={[xshift=-2pt,yshift=4pt]0:$1$}] (v5) at ({90 + 60 * 5}:0.5) {};
\path[edge] (v3) -- (v5);
\path[edge] (v3) -- (v1);
\path[edge] (v5) -- (v2);
\path[edge] (v5) -- (v1);
\path[edge] (v1) -- (v4);
\path[edge] (v1) -- (v2);
\path[edge] (v4) -- (v3);
\path[edge] (v2) -- (v4);
\path[edge] (v5) to [bend left=20] (v0);
\path[edge] (v0) to [bend left=20] (v5);
\path[edge] (v2) to [bend left=20] (v0);
\path[edge] (v0) to [bend left=20] (v2);
\path[edge] (v0) to [bend left=10] (v3);
\path[edge] (v3) to [bend left=10] (v0);
\path[edge] (v0) to [bend left=10] (v4);
\path[edge] (v4) to [bend left=10] (v0);
\end{tikzpicture}
}
\caption{Small ambi-nut digraphs with pairs of oppositely oriented arcs. 
Panel (a) shows the Sciriha graph $S_1$ as we view it in the theory of directed graphs.
Panel (b) shows the unique example on $4$ vertices, while (c)--(e) show three out of
$14$ examples on $6$ vertices.}
\label{fig:nonOrientedAmbis}
\end{figure} 

Finally, in this survey of examples, we consider inter-nut digraphs $G$ that are not ambi-nut digraphs, i.e. $\dim \Ker (G) > 1$.
There are $2$ such inter-nut digraphs on $6$ vertices, and $27$ on $7$ vertices.  These $27$ inter-nut digraphs can all be obtained from just $6$
different underlying graphs by choosing a suitable orientation. It happens that those $6$ underlying graphs include the three nut graphs on $7$ 
vertices, i.e. the three Sciriha graphs~\cite{NutOrbitPaper}.

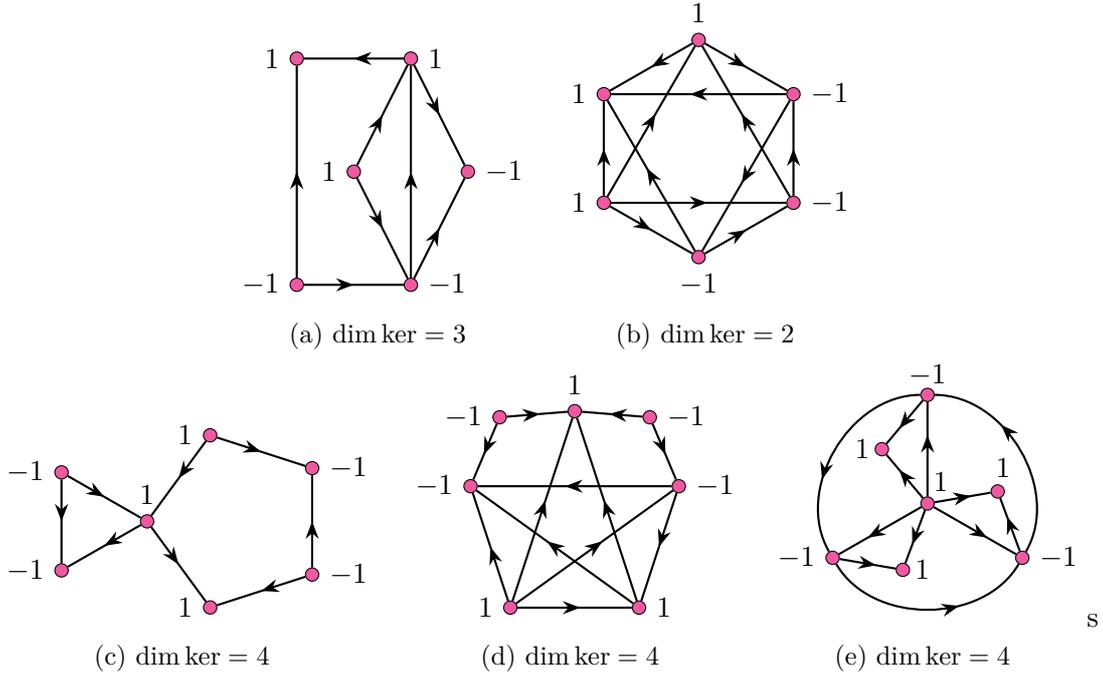
\begin{figure}[!htbp]
\centering
\subcaptionbox{$\dim\ker = 3$}{
\begin{tikzpicture}[scale=1.5]
\tikzstyle{vertex}=[draw,circle,font=\scriptsize,minimum size=5pt,inner sep=1pt,fill=magenta!80!white]
\tikzstyle{edge}=[draw,thick,decoration={markings,mark=at position 0.5 with {\arrow{Stealth}}},postaction=decorate]
\tikzstyle{sedge}=[draw,thick,decoration={markings,mark=at position 0.45 with {\arrow{Stealth}}},postaction=decorate]
\tikzstyle{ledge}=[draw,thick,decoration={markings,mark=at position 0.65 with {\arrow{Stealth}}},postaction=decorate]
\node[vertex,label=0:$1$] (v5) at (0, 2) {};
\node[vertex,label=0:$-1$] (v4) at (0, 0) {};
\node[vertex,label=0:$-1$] (v2) at (0.5, 1) {};
\node[vertex,label=-180:$1$] (v1) at (-0.5, 1) {};
\node[vertex,label=180:$1$] (v3) at (-1, 2) {};
\node[vertex,label=180:$-1$] (v0) at (-1, 0) {};
\path[edge] (v0) -- (v3);
\path[edge] (v0) -- (v4);
\path[edge] (v4) -- (v5);
\path[edge] (v4) -- (v2);
\path[edge] (v5) -- (v3);
\path[edge] (v5) -- (v2);
\path[edge] (v1) -- (v4);
\path[edge] (v1) -- (v5);
\end{tikzpicture}
}
\subcaptionbox{$\dim\ker = 2$}{
\begin{tikzpicture}[scale=1.2]
\tikzstyle{vertex}=[draw,circle,font=\scriptsize,minimum size=5pt,inner sep=1pt,fill=magenta!80!white]
\tikzstyle{edge}=[draw,thick,decoration={markings,mark=at position 0.5 with {\arrow{Stealth}}},postaction=decorate]
\tikzstyle{sedge}=[draw,thick,decoration={markings,mark=at position 0.45 with {\arrow{Stealth}}},postaction=decorate]
\tikzstyle{ledge}=[draw,thick,decoration={markings,mark=at position 0.55 with {\arrow{Stealth}}},postaction=decorate]
\node[vertex,label=90:$1$] (v3) at (90:1.2) {};
\node[vertex,label=180:$1$] (v4) at ({90 + 60}:1.2) {};
\node[vertex,label=180:$1$] (v0) at ({90 + 60 * 2}:1.2) {};
\node[vertex,label=-90:$-1$] (v2) at ({90 + 60 * 3}:1.2) {};
\node[vertex,label=0:$-1$] (v5) at ({90 + 60 * 4}:1.2) {};
\node[vertex,label=0:$-1$] (v1) at ({90 + 60 * 5}:1.2) {};
\path[edge] (v0) -- (v4);
\path[ledge] (v0) -- (v3);
\path[edge] (v0) -- (v2);
\path[ledge] (v0) -- (v5);
\path[ledge] (v2) -- (v4);
\path[edge] (v2) -- (v5);
\path[edge] (v5) -- (v1);
\path[ledge] (v5) -- (v3);
\path[ledge] (v1) -- (v2);
\path[ledge] (v1) -- (v4);
\path[edge] (v3) -- (v1);
\path[edge] (v3) -- (v4);
\end{tikzpicture}
}

\subcaptionbox{$\dim\ker = 4$}{
\begin{tikzpicture}[scale=1.0]
\tikzstyle{vertex}=[draw,circle,font=\scriptsize,minimum size=5pt,inner sep=1pt,fill=magenta!80!white]
\tikzstyle{edge}=[draw,thick,decoration={markings,mark=at position 0.5 with {\arrow{Stealth}}},postaction=decorate]
\tikzstyle{sedge}=[draw,thick,decoration={markings,mark=at position 0.45 with {\arrow{Stealth}}},postaction=decorate]
\tikzstyle{ledge}=[draw,thick,decoration={markings,mark=at position 0.65 with {\arrow{Stealth}}},postaction=decorate]
\node[vertex,label=90:$1$] (v6) at ({180}:1.2) {};
\node[vertex,label=180:$1$] (v3) at ({180 + 72}:1.2) {};
\node[vertex,label=0:$-1$] (v0) at ({180 + 72 * 2}:1.2) {};
\node[vertex,label=0:$-1$] (v5) at ({180 + 72 * 3}:1.2) {};
\node[vertex,label=180:$1$] (v2) at ({180 + 72 * 4}:1.2) {};
\node[vertex,label=180:$-1$] (v1) at ($ (v6) + (150:1.3) $) {};
\node[vertex,label=180:$-1$] (v4) at ($ (v6) + (210:1.3) $) {};
\path[edge] (v1) -- (v4);
\path[edge] (v1) -- (v6);
\path[edge] (v6) -- (v4);
\path[edge] (v6) -- (v3);
\path[edge] (v2) -- (v6);
\path[edge] (v2) -- (v5);
\path[edge] (v0) -- (v3);
\path[edge] (v0) -- (v5);
\end{tikzpicture}
}
\subcaptionbox{$\dim\ker = 4$}{
\begin{tikzpicture}[scale=1.2]
\tikzstyle{vertex}=[draw,circle,font=\scriptsize,minimum size=5pt,inner sep=1pt,fill=magenta!80!white]
\tikzstyle{edge}=[draw,thick,decoration={markings,mark=at position 0.5 with {\arrow{Stealth}}},postaction=decorate]
\tikzstyle{sedge}=[draw,thick,decoration={markings,mark=at position 0.45 with {\arrow{Stealth}}},postaction=decorate]
\tikzstyle{ledge}=[draw,thick,decoration={markings,mark=at position 0.55 with {\arrow{Stealth}}},postaction=decorate]
\node[vertex,label=90:$1$] (v5) at ({90}:1.2) {};
\node[vertex,label=180:$-1$] (v4) at ({90 + 72}:1.2) {};
\node[vertex,label=180:$1$] (v0) at ({90 + 72 * 2}:1.2) {};
\node[vertex,label=0:$1$] (v6) at ({90 + 72 * 3}:1.2) {};
\node[vertex,label=0:$-1$] (v3) at ({90 + 72 * 4}:1.2) {};
\node[vertex,label=180:$-1$] (v2) at ({90 + 36}:1.4) {};
\node[vertex,label=0:$-1$] (v1) at ({90 -36}:1.4) {};
\path[edge] (v0) -- (v4);
\path[ledge] (v0) -- (v5);
\path[ledge] (v0) -- (v3);
\path[ledge] (v0) -- (v6);
\path[ledge] (v6) -- (v4);
\path[ledge] (v6) -- (v5);
\path[edge] (v3) -- (v6);
\path[ledge] (v3) -- (v4);
\path[ledge] (v2) -- (v4);
\path[ledge] (v2) -- (v5);
\path[ledge] (v1) -- (v3);
\path[ledge] (v1) -- (v5);
\end{tikzpicture}
}
\subcaptionbox{$\dim\ker = 4$}{
\begin{tikzpicture}[scale=1.2]
\tikzstyle{vertex}=[draw,circle,font=\scriptsize,minimum size=5pt,inner sep=1pt,fill=magenta!80!white]
\tikzstyle{edge}=[draw,thick,decoration={markings,mark=at position 0.5 with {\arrow{Stealth}}},postaction=decorate]
\tikzstyle{sedge}=[draw,thick,decoration={markings,mark=at position 0.45 with {\arrow{Stealth}}},postaction=decorate]
\tikzstyle{ledge}=[draw,thick,decoration={markings,mark=at position 0.65 with {\arrow{Stealth}}},postaction=decorate]
\node[vertex,label={[xshift=5pt,yshift=-2pt]90:$1$}] (v6) at (0, 0) {};
\node[vertex,label={[yshift=-3pt]90:$-1$}] (v5) at ({90}:1.2) {};
\node[vertex,label=180:$-1$] (v0) at ({90 +120}:1.2) {};
\node[vertex,label=0:$-1$] (v4) at ({90 - 120}:1.2) {};
\node[vertex,label={[xshift=2pt]180:$1$}] (v2) at ($ (90:0.6) + (180:0.5) $) {};
\node[vertex,label={[xshift=-2pt]0:$1$}] (v3) at ($ (210:0.6) + (300:0.5) $) {};
\node[vertex,label={[yshift=-1pt,xshift=2pt]90:$1$}] (v1) at ($ (-30:0.6) + (60:0.5) $) {};
\path[ledge] (v6) -- (v1);
\path[ledge] (v6) -- (v2);
\path[ledge] (v6) -- (v3);
\path[ledge] (v6) -- (v4);
\path[ledge] (v6) -- (v5);
\path[ledge] (v6) -- (v0);
\path[ledge] (v0) -- (v3);
\path[ledge] (v4) -- (v1);
\path[ledge] (v5) -- (v2);
\path[ledge] (v5) to [bend right=60] (v0);
\path[ledge] (v0) to [bend right=60] (v4);
\path[ledge] (v4) to [bend right=60] (v5);
\end{tikzpicture}}s
\caption{Examples of inter-nut digraphs that are not ambi-nut digraphs. Panels (a) and (b) show both such inter-nut digraphs on $6$ vertices. Panels (c) to (e) show three
of $27$ such inter-nut digraphs on $7$ vertices. Note that their underlying graphs are the Sciriha graphs $S_1$, $S_2$ and $S_3$, respectively~\cite{NutOrbitPaper}.}
\label{fig:examplesInternuts}
\end{figure}

\section{Basic properties of nut digraphs}
\label{sec:basic}

In the undirected universe, nut graphs are connected, non-bipartite and leafless. All these properties have their counterparts in the directed universe.
We establish some useful machinery which will be used in the present section and also invoked for some constructions in Section~\ref{sec:constructions}.
For a digraph $G$, a function $\x \colon V(G) \to \RR$ and a vertex $v$ of $G$, let
\begin{equation}
\label{eq:inAndOutLocal}
\Su_\x^+(v) = \sum_{u\in G^+(v)}\x(u)\quad \hbox{ and } \quad
\Su_\x^-(v) = \sum_{u\in G^-(v)}\x(u)
\end{equation}
and recall from \eqref{eq:zerosum} and \eqref{eq:zerosumn} that $\x\in\Ker G$ provided $\Su_\x^+(v) = 0$ for every $v\in V(G)$ and that $\x\in\CoKer G$ provided $\Su_\x^-(v) = 0$ for every $v\in V(G)$.

In the undirected universe, vertices of singular graphs may be
partitioned into \emph{core} and \emph{core-forbidden} vertices. Namely,  a vertex $v$ of a singular graph $G$ is a core vertex if there exists some $\x \in \Ker G$ such that
$\x(u) \neq 0$, otherwise $v$ is  a core-forbidden vertex (see \cite[Definition 1]{Sciriha2021}). 
In a further refinement, core-forbidden vertices are partitioned into \emph{middle} and \emph{upper}, accordingly as $\eta(G - v) = \eta(G)$ or $\eta(G - v) = \eta(G) + 1$.
As $\Ker G$ and $\CoKer G$ are not necessarily the same in the digraph universe, we
have to take this fact into consideration in framing corresponding definitions for digraphs.

\begin{definition}
Let $G$ be a singular digraph. A vertex $u \in V(G)$  is a \emph{dextro-core} vertex if there exists some $\x \in \Ker G$ such that
$\x(u) \neq 0$, otherwise $v$ is a \emph{dextro-core-forbidden} vertex. A vertex $u \in V(G)$  is a \emph{laevo-core} vertex if there exists 
some $\y \in \CoKer G$ such that $\y(u) \neq 0$, otherwise $v$ is a \emph{laevo-core-forbidden} vertex.
\end{definition}

The following lemma shows that if $G$ is a laevo-nut digraph,  then a function $\y$ happens to be an element of $\Ker G$ under slightly weaker assumptions.

\begin{lemma}
\label{lem:ext}
Let $G$ be a digraph and let $w \in V(G)$.
If a function $\y \in \RR^V$ satisfies the condition $\Su_\y^+(u) = 0$ for all $u\in V(G) \setminus \{w\}$ and if 
$w$ is a laevo-core vertex
then $\Su_\y^+(w) = 0$, and thus $\y \in \Ker G$. 
Similarly, if a function $\y\in \RR^V$ satisfies the condition $\Su_\y^-(u) = 0$ for all $u\in V(G) \setminus \{w\}$
and 
$w$ is a dextro-core vertex
then $\Su_\y^-(w) = 0$, and thus $\y\in \CoKer G$.
\end{lemma}

\begin{proof}
For two functions $\w,\z \in \RR^V$, let
\begin{equation}
\label{eq:<>}
\langle \w,\z \rangle = \sum_{u,v \in V(G) \atop v\to u} \w(v)\z(u),
\end{equation}
and observe that
\begin{equation}
 \langle \x, \y\rangle = \sum_{u \in V(G)} \sum_{v\in G^-(u)} \x(v)\y(u) = \sum_{u \in V(G)} \y(u) \Su_\x^-(u) = 0.
\end{equation}
On the other hand,
\begin{equation}
 \langle \x, \y\rangle = \sum_{v \in V(G)} \sum_{u\in G^+(v)} \x(v)\y(u) = \sum_{v \in V(G) \setminus\{w\}}  \x(v) \Su_\y^+(v)  + \x(w) \Su_\y^+(w) = \x(w) \Su_\y^+(w).
\end{equation}
Therefore, $\x(w) \Su_\y^+(w)=0$, and since $\x(w) \not = 0$, we see that $\Su_\y^+(w) = 0$, as claimed.
The second claim of the lemma can be proved by applying the first claim to
 the reverse digraph $G^R$.
\end{proof}

\subsection{Connectedness and vertex deletion}

A digraph $G$ is said to be {\em connected} if its underlying graph is connected. 

\begin{lemma}
Every dextro-nut digraph, every laevo-nut digraph and every inter-nut digraph is connected.
\end{lemma}

\begin{proof}
If $G$ is a dextro-nut digraph then there exists a full vector $\x \in \Ker G$. Let $V(G) = V_1 \cup V_2$, $V_1 \neq \emptyset$ and $V_2 \neq \emptyset$, such that
$G[V_1]$ is a connected component. (Then $G[V_2]$ is a disjoint union of one or more connected components.) Let us define a function $\y \colon V(G) \to \mathbb{R}$ by 
$\y(u) = \x(u)$ for $u \in V_1$ and $\y(u) = 0$ for $u \in V_2$. As~\eqref{eq:zerosum} holds for $\y$, we have that $\y \in \Ker G$. This contradicts the antecedent
that $G$ is a dextro-nut digraph. An analogous argument applies to laevo-nut digraphs. 

For inter-nut digraphs, the above arguments also apply with minor adjustments. Namely, one has to take $\x \in \Ker G \cap \CoKer G$ and observe that then the function 
$\y \in \Ker G \cap \CoKer G$ as it satisfies both~\eqref{eq:zerosum} and~\eqref{eq:zerosumn}.
\end{proof}

All examples of ambi-nut digraphs which we have encountered happen to be \emph{strongly connected} (i.e.\ there exists a directed path between every pair of vertices).
We introduce notions needed to prove that this is a general property of ambi-nut digraphs. First, we prove an elementary fact from linear algebra~\cite{Strang2023}.

\begin{lemma}
\label{thm:linAlgThm}
Let $A \in \mathbb{C}^{n \times m}$ (i.e.\ $A$ is an $n \times m$ matrix over the field $\mathbb{C}$). Let $B$ be a submatrix obtained from $A$ by deleting any row and any column.
Then
\begin{equation}
\eta(A) - 1 \leq \eta(B) \leq \eta(A) + 1 \quad \text{and} \quad \eta(A^\intercal) - 1 \leq \eta(B^\intercal) \leq \eta(A^\intercal) + 1.
\end{equation}
\end{lemma}

\begin{proof}
Without loss of generality, assume that we delete the first row and the first column. We can write
\begin{equation}
A = \begin{bmatrix}
 z & \mathbf{c}^\intercal \\
 \mathbf{d} & B 
\end{bmatrix},
\end{equation}
where $z \in \mathbb{C}$ and $\mathbf{c}, \mathbf{d} \in  \mathbb{C}^{(n - 1) \times 1}$. Let $D \coloneqq \begin{bsmallmatrix} \mathbf{c}^\intercal \\ B \end{bsmallmatrix}$.
When $\mathbf{c}^\intercal$ is added to the row-space of $B$, the rank is either unchanged or increases by $1$, i.e.\ $\rank(B) \leq \rank(D) \leq \rank(B) + 1$. When the vector
$\begin{bsmallmatrix} z \\ \mathbf{d} \end{bsmallmatrix}$ is added to the column-space of $D$, again, the rank is either unchanged or increases by $1$. Therefore, 
\begin{equation}
\rank(B) \leq \rank(A) \leq \rank(B) + 2.
\label{eq:rankStuff}
\end{equation}
Recall that $\rank(A) + \eta(A) = m$ and $\rank(A^\intercal) + \eta(A^\intercal) = n$. 
Similarly, $\rank(B) + \eta(B) = m - 1$ and $\rank(B^\intercal) + \eta(B^\intercal) = n - 1$.
Moreover, $\rank(A) = \rank(A)^\intercal$ and $\rank(B) = \rank(B)^\intercal$. From \eqref{eq:rankStuff} we obtain $\eta(A) - 1 \leq \eta(B) \leq \eta(A) + 1$ and, 
similarly, $\eta(A^\intercal) - 1 \leq \eta(B^\intercal) \leq \eta(A^\intercal) + 1$.
\end{proof}

Lemma~\ref{thm:linAlgThm} has an immediate consequence.

\begin{corollary}
\label{cor:nullityCorollary}
Let $G$ be a digraph of order $n \geq 2$ and $v \in V(G)$ an arbitrary vertex. Then 
\begin{equation}
\eta(G) - 1 \leq \eta(G - v) \leq \eta(G) + 1.
\end{equation}
\end{corollary}

It is known that the above corollary holds for undirected graphs (see, e.g., \cite[Theorem 1.2]{Ali2016}), where the usual proof relies on the Cauchy Interlacing Theorem for symmetric matrices,
which is not applicable in the directed universe.

In the undirected universe, if a core-vertex $v$ is deleted from a graph $G$,
then $\eta(G - v) = \eta(G) - 1$ (see, e.g., \cite[Corollary 13]{ScirihaCoalescence}), whereas if a core-forbidden vertex $v$ is deleted, then $\eta(G - v) = \eta(G)$ or $\eta(G - v) = \eta(G) + 1$
(see, e.g., \cite{Omniconductors}). We will now generalise this to digraphs.

\begin{lemma}
\label{lem:atLeastOneCore}
Let $G$ be a singular digraph with at least $2$ vertices and let $v \in V(G)$ be any vertex.
If $v$ is laevo-core then $\eta(G - v) \leq \eta(G)$. If, in addition, $v$ is dextro-core then $\eta(G - v)  = \eta(G) - 1$.
\end{lemma}

\begin{proof}
Let $\eta = \eta(G)$ and $\eta' = \eta(G - v)$. Let 
$\{ \x_1, \x_2, \ldots, \x_{\eta'} \}$ be a basis for $\ker(G - v)$, and let $\widetilde{\x}_i \colon V(G) \to \mathbb{R}$, $1 \leq i \leq \eta'$, be defined by $\widetilde{\x}_i(u) = \x_i(u)$ for $u \neq v$
 and $\widetilde{\x}_i(v) = 0$. It is easy to see that $\widetilde{\x}_1, \widetilde{\x}_2, \ldots, \widetilde{\x}_{\eta'}$ are linearly independent and each of them satisfies the assumptions of Lemma~\ref{lem:ext}. Since $v$ is a laevo-core vertex there exists $\widetilde{\y} \in \CoKer G$ such that $\widetilde{\y}(v) \neq 0$. This allows us to use Lemma~\ref{lem:ext} to conclude that $\langle \widetilde{\x}_1, \ldots, \widetilde{\x}_{\eta'} \rangle \subseteq \Ker G$ and hence $\eta' \leq \eta$. 
 
 If $v$ is also dextro-core, then there exists $\widetilde{\x}_0 \in \Ker G$ such that $\widetilde{\x}_0 (v) \neq 0$. Then
 $\widetilde{\x}_0 \notin \langle \widetilde{\x}_1, \ldots, \widetilde{\x}_{\eta'} \rangle$, implying that $\eta'  < \eta$.
 From Corollary~\ref{cor:nullityCorollary} it follows that $\eta(G - v)  = \eta(G) - 1$, as claimed.
 \end{proof}
 
 \begin{lemma}
\label{lem:atLeastOneForbidden}
Let $G$ be a singular digraph with at least $2$ vertices and let $v \in V(G)$ be any vertex.
If $v$ is laevo-core-forbidden or dextro-core-forbidden then $\eta(G - v) \in \{ \eta(G), \eta(G) + 1 \}$.
\end{lemma}

\begin{proof}
Let $\eta = \eta(G)$. 
Suppose first that $v$ is dextro-core-forbidden.
Let $\widetilde{\z}_1, \widetilde{\z}_2, \ldots, \widetilde{\z}_\eta$ be a basis of $\ker G$ and let  $\z_i \colon V(G - v) \to \mathbb{R}$, $1 \leq i \leq \eta$, be defined by $\z_i(u) = \widetilde{\z}_i(u)$ 
for $u \in V(G - v)$.
Since $v$ is dextro-core-forbidden, $\widetilde{\z}_1(v) = \cdots = \widetilde{\z}_\eta(v) = 0$, implying that $\z_1, \ldots, \z_\eta$ are linearly independent and
 $\langle  \z_1, \ldots, \z_\eta \rangle \subseteq \ker (G - v)$.  Therefore, $\eta(G - v) \geq \eta$. By Corollary~\ref{cor:nullityCorollary} it follows that $\eta(G - v) = \eta$ or $\eta(G - v) = \eta + 1$.
 
 If $v$ is laevo-core-forbidden, then $v$ is dextro-core-forbidden in the reverse digraph $G^R$. Since $\eta(G) = \eta(G^R)$, the results follows by the previous paragraph.
\end{proof}

\begin{theorem}
\label{thm:niceThm}
Let $G$ be a singular digraph of order $n \geq 2$ and let $v \in V(G)$ be any vertex.
\begin{enumerate}[label=(\roman*)]
\item 
If $v$ is dextro-core and laevo-core then $\eta(G - v) = \eta(G) - 1$.
\item 
If $v$ is dextro-core-forbidden and laevo-core then $\eta(G - v) = \eta(G)$.
\item 
If $v$ is dextro-core and laevo-core-forbidden then $\eta(G - v) = \eta(G)$.
\item 
If $v$ is dextro-core-forbidden and laevo-core-forbidden then $\eta(G - v) \in \{ \eta(G), \eta(G) + 1 \}$. Moreover, if $G$ is bipartite then $\eta(G - v) = \eta(G) + 1$.
\end{enumerate}
\end{theorem}

\begin{proof}
Statement (i) follows directly from Lemma~\ref{lem:atLeastOneCore} and the first part of statement (iv) follows directly from Lemma~\ref{lem:atLeastOneForbidden}.
Statement (ii) is obtained by combining Lemmas \ref{lem:atLeastOneCore} and \ref{lem:atLeastOneForbidden}. Finally, statement (iii) follows by application of (ii) to the reverse digraph $G^R$. 

Now we prove the second part of statement (iv).
Without loss of generality, for $G$ bipartite the adjacency matrices of $G$ and $G - v$ can be written as
\begin{align}
A(G) & = \begin{bmatrix}
0 & \mathbf{0}^\intercal & \mathbf{u}^\intercal \\
\mathbf{0} & O_{p - 1} & B \\
\mathbf{w} & C & O_{q} \\
\end{bmatrix}, &
A(G - v) & = \begin{bmatrix}
 O_{p - 1} & B \\
 C & O_{q} \\
\end{bmatrix},
\end{align}
where $O_t$ denotes the all-zero $t \times t$ matrix. 
If there exists a column vector $\mathbf{a}$ such that $\mathbf{w} = C \mathbf{a}$, then $\begin{bsmallmatrix}-1 & \mathbf{a} & 0 \end{bsmallmatrix}^\intercal \in \Ker A(G)$,
contradicting the assumption that $v$ is a dextro-core-forbidden vertex. 
Therefore, $\mathbf{w} \notin \im C$.
Recall that $\im C$ denotes the image of $C$, i.e.\ $\im C = \{ C\x \mid \x \in \mathbb{R}^{p - 1} \}$.
Similarly, since $v$ is also laevo-core-forbidden, it is dextro-core-forbidden in $G^R$, implying that $\mathbf{u} \notin \im B^\intercal$.
In particular, this implies that $\mathbf{w}$ is not in the column-space of $C$ and $\mathbf{u}^\intercal$ is not in the row-space of $B$. 
But then $\rank A(G) = \rank A(G - v) + 2$. As in the proof of Lemma~\ref{thm:linAlgThm}, by the Rank-Nullity Theorem it then follows 
that $\eta(G - v) = \eta(G) + 1$.
\end{proof}

\noindent
Theorem~\ref{thm:niceThm} gives us a way of retrieving the middle/upper stratification of core-forbidden vertices. 

\begin{definition}
Let $G$ be a singular digraph.
A vertex $v \in V(G)$ that is both dextro-core-forbidden and laevo-core-forbidden is called \emph{upper} if $\eta(G - v) = \eta(G) + 1$,
 and \emph{middle} if $\eta(G - v) = \eta(G)$.
\end{definition}

For a bipartite (undirected) graph $G$  the parity of $\eta(G)$ is the same as the parity of $|V(G)|$ by the Pairing Theorem, and hence middle vertices are absent. 
However, for bipartite digraphs, as Theorem~\ref{thm:niceThm} shows, it is also possible that deletion of a vertex leaves the nullity unchanged.
 
\begin{example} 
The digraphs in Figure~\ref{fig:niceTheoremAll} show that all scenarios in Theorem~\ref{thm:niceThm} may occur. 

Consider the digraph $G_1$ in Figure~\ref{fig:niceTheoremAll}(a) with $\eta(G_1) = 2$. Vertices $1$ and $2$ belong to case (i), i.e.\ they are dextro-core and laevo-core, and thus
$\eta(G_1 - 1) = \eta(G_1 - 2) = 1$. Vertices $3$ and $4$ belong to cases (ii) and (iii), respectively, and thus $\eta(G_1 - 3) = \eta(G_1 - 4) = 2$. Vertices $5$ and $6$ belong to case (iv),
where vertex $5$ is upper and vertex $6$ is middle, i.e.\ $\eta(G_1 - 5) = 3$ and $\eta(G_1 - 6) = 2$.

The digraph $G_2$ in Figure~\ref{fig:niceTheoremAll}(b) is bipartite with $\eta(G_2) = 2$. Vertices $1$ and $2$ belong to case (i) and 
thus $\eta(G_2 - 1) = \eta(G_2 - 2) = 1$. Vertices $3$ and $4$ belong to cases (ii) and (iii), 
respectively, and thus $\eta(G_2 - 3) = \eta(G_2 - 4) = 2$. The vertex $5$ belongs to case (iv), but since $G_2$ is bipartite, it can only be upper, with $\eta(G_2 - 5) = 3$.

\begin{figure}[!htbp]
\centering
\subcaptionbox{$G_1$}{
\begin{tikzpicture}[scale=1.0]
\tikzstyle{vertex}=[draw,circle,font=\scriptsize,minimum size=5pt,inner sep=1pt,fill=magenta!80!white]
\tikzstyle{edge}=[draw,thick,decoration={markings,mark=at position 0.6 with {\arrow{Stealth}}},postaction=decorate]
\node[vertex,label=180:$1$] (v1) at (0, 1) {};
\node[vertex,label=180:$3$] (v4) at (-1, 0) {};
\node[vertex,label=0:$5$] (v5) at (1, 0) {};
\node[vertex,label=90:$2$] (v2) at (0, 0) {};
\node[vertex,label=180:$4$] (v0) at (0, -1) {};
\node[vertex,label=0:$6$] (v3) at (1.5, -1) {};
\path[edge] (v1) -- (v4);
\path[edge] (v5) -- (v1);
\path[edge] (v2) -- (v4);
\path[edge] (v0) -- (v4);
\path[edge] (v5) -- (v2);
\path[edge] (v0) -- (v5);
\path[edge] (v0) -- (v3);
\path[edge] (v3) -- (v5);
\end{tikzpicture}
}
\qquad
\subcaptionbox{$G_2$}{
\begin{tikzpicture}[scale=1.0]
\tikzstyle{vertex}=[draw,circle,font=\scriptsize,minimum size=5pt,inner sep=1pt,fill=magenta!80!white]
\tikzstyle{edge}=[draw,thick,decoration={markings,mark=at position 0.6 with {\arrow{Stealth}}},postaction=decorate]
\node[vertex,label=180:$1$] (v1) at (0, 1) {};
\node[vertex,label=180:$3$] (v4) at (-1, 0) {};
\node[vertex,label=0:$5$] (v5) at (1, 0) {};
\node[vertex,label=90:$2$] (v2) at (0, 0) {};
\node[vertex,label=180:$4$] (v0) at (0, -1) {};
\path[edge] (v1) -- (v4);
\path[edge] (v5) -- (v1);
\path[edge] (v2) -- (v4);
\path[edge] (v0) -- (v4);
\path[edge] (v5) -- (v2);
\path[edge] (v0) -- (v5);
\end{tikzpicture}
}
\caption{Examples of digraphs covering all scenarios in Theorem~\ref{thm:niceThm}.}
\label{fig:niceTheoremAll}
\end{figure}
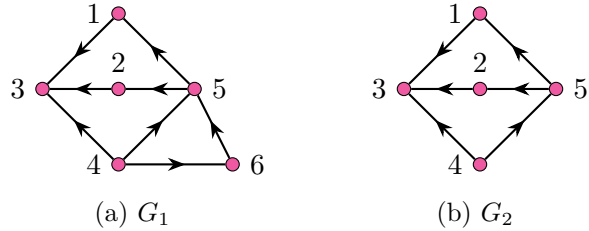
\end{example}

Since every vertex of a bi-nut digraph is dextro-core and laevo-core, Theorem~\ref{thm:niceThm} implies the following corollary, which
generalises \cite[Corollary 15]{ScirihaCoalescence} from nut graphs to bi-nut digraphs. 

\begin{corollary}
\label{lem:extendedSciriha}
Let $G$ be a bi-nut digraph and let $w \in V(G)$ be an arbitrary vertex. Then $G - w$ is a non-singular digraph.
\end{corollary}

We will now make use of another concept from linear algebra.
An $n \times n$ matrix $A$ is \emph{reducible} (in the sense of \cite[Definition 6.2.21]{Horn2013}) if it is permutation-equivalent to a matrix 
\begin{equation}
\begin{bmatrix}
B & C \\
O & D
\end{bmatrix},
\end{equation}
where neither $B$ nor $D$ is an empty matrix (i.e. a matrix with no rows or no columns). A matrix that is not reducible is called \emph{irreducible}. A digraph $G$ is strongly 
connected if and only if $A(G)$ is irreducible (see, e.g.,~\cite[Theorem 6.2.24]{Horn2013} or \cite[Theorem 3.2.1]{BrualdiBook}). Now, we have all the tools required to prove the next theorem.

\begin{theorem}
\label{thm:stronglyConnectedBi}
Every bi-nut digraph is strongly connected.
\end{theorem}

\begin{proof}
For contradiction, suppose that a bi-nut digraph $G$ is not strongly connected. Then its adjacency matrix is reducible and can be written, without loss of generality, as
\begin{equation}
A(G) = \begin{bmatrix}
B & C \\
O & D
\end{bmatrix},
\end{equation}
where $O$ is an all-zero matrix.
Since $G$ is a bi-nut digraph, it is singular and thus
\begin{equation}
\label{eq:detProductBD}
\det A(G) = (\det B) (\det D) = 0.
\end{equation}

Now, let $G'$ be obtained from $G$ by deleting the first vertex, i.e.\ the vertex that corresponds to the first row and first column in $A(G)$. The adjacency matrix of $G'$ is of the form
\begin{equation}
A(G') = \begin{bmatrix}
B' & C' \\ O' & D
\end{bmatrix},
\end{equation}
where $O'$ is an all-zero matrix.
By Corollary~\ref{lem:extendedSciriha}, $G'$ is a non-singular digraph, and thus $\det A(G') = (\det B') (\det D) \neq 0$. In particular, $\det D \neq 0$.
Let $G''$ be obtained from $G$ by deleting the last vertex. By similar reasoning we get that $\det B \neq 0$. But this contradicts \eqref{eq:detProductBD}.
\end{proof}

Note that every ambi-nut digraph is a bi-nut digraph, and hence ambi-nut digraphs are strongly connected.

\subsection{Bipartiteness}

Nut graphs are non-bipartite. The same is true in the universe of directed graphs, as we now show.
Note that a digraph $G$ is bipartite if and only if its underlying graph is bipartite.

\begin{lemma}
Every dextro-nut digraph, every laevo-nut digraph and every inter-nut digraph is non-bipartite.
\end{lemma}

\begin{proof}
 Suppose that $G$ is a bipartite digraph and that $U$ and $W$ are the parts of the bipartition of its underlying graph.
 If $\x \in \Ker G$, then the function $\x' \colon V(G) \to \RR$ coinciding with $\x$ on $U$ and mapping each vertex of $W$ to $0$
 is also in $\Ker G$, implying that $G$ is not a dextro-nut digraph. Since the reverse digraph of a bipartite digraph is also bipartite, 
 the same argument applied to the reverse digraph $G^R$ shows that it is not a dextro-nut digraph, and thus $G$ is not a laevo-digraph nut.
 
 For inter-nut digraphs, the above arguments also apply with minor adjustments. Namely, one has to take $\x \in \Ker G \cap \CoKer G$ and observe that then the function 
$\x' \in \Ker G \cap \CoKer G$, implying that $G$ is not an inter-nut digraph.
\end{proof}

Clearly, the above lemma applies to ambi-nut digraphs and bi-nut digraphs, as they are dextro-nut digraphs.

\subsection{Vertex degrees and leaflessness}

Nut graphs have no vertices of degree $1$. In the universe of directed graphs, the situation is more complex.
Leaflessness generalises as follows. 

\begin{lemma}
\label{lem:degreeBounds}
If $G$ is a dextro-nut digraph, then $\delta^-(G) \ge 1$ and $\delta^+(G) \not =1$.
Similarly, if $G$ is a laevo-nut digraph, then then $\delta^+(G) \ge 1$ and $\delta^-(G) \not =1$.
Consequently, if $G$ is a bi-nut digraph, then $\delta^-(G) \ge 2$ and $\delta^+(G) \ge 2$.
\end{lemma}

\begin{proof}
Let $G$ be a dextro-nut digraph with its kernel spanned by $\x$. If there exists a vertex $v\in V(G)$ with no in-neighbours,
then a function $\x' \colon V(G) \to \RR$ coinciding with $\x$ in each vertex except in $v$ while mapping $v$ to $0$ is
also in $\Ker G$, contradicting the fact that $\Ker G$ is one-dimensional. Hence, $\delta^-(G) \ge 1$.
Furthermore, if there is vertex $v\in V(G)$ with exactly one out-neighbour, say $w$, then $\sum_{u \in G^+(v)} \x(u) = \x(w) \not = 0$,
contradicting the assumption that $\x \in \Ker G$.
The claim about laevo-nut digraphs follows, by applying the  preceding reasoning to the reverse digraph $G^R$.
\end{proof}

In other words, Lemma~\ref{lem:degreeBounds} says that dextro-nut digraphs may not possess sources. However, dextro-nut digraphs can have sinks; see Figure~\ref{fig:smallGeneral}(a).
Similarly, by Lemma~\ref{lem:degreeBounds}, laevo-nut digraphs may not possess sinks, but can have sources. Note that the underlying graph of a dextro- or laevo-nut digraph may have leaves; see, e.g., Figure~\ref{fig:smallGeneral}(b). Also note that a given inter-nut digraph $G$ may have both sinks and sources, as seen from examples in Figure~\ref{fig:examplesInternuts}.

Lemma~\ref{lem:degreeBounds} implies that if $G$ is a bi-nut digraph then the minimum degree of its underlying graph is at least $4$.
Table~\ref{tbl:enumQuartic} gives us some information about the simplest case, where the underlying graph is quartic.
From Table~\ref{tbl:enumQuartic} it appears that bi-nut digraphs (and therefore ambi-nut digraphs) do not exist if the underlying graph is quartic
and of odd order. The next proposition generalises this observation.

\begin{proposition}
Let $G$ be a bi-nut digraph, such that its underlying graph is $4$-regular. Then $G$ is of even order $2n$. Moreover,  $\Ker G$ is spanned
by a vector $\x$ that consists of entries from $\{+1, -1\}$, $n$ of which are $+1$, and $\CoKer G$ is spanned by a vector $\y$ that consists of entries from $\{+1, -1\}$, $n$ of which are $+1$.
\end{proposition}

\begin{proof}
By Lemma~\ref{lem:degreeBounds}, $d^+(v) = d^-(v) = 2$ for every $v \in V(G)$.
Let $\ell\colon V(G) \to \{+a, -a, \star\}$ denote a mapping that assigns to each $v$ vertex of $G$ value $+a$ or value $-a$ or leaves 
 $\ell(v)$ undefined, which is indicated by the $\star$ symbol.
 We describe an algorithm which finds a non-trivial vector from $\ker G$.

Initially, set $\ell(u) = \star$ for every $u \in V(G)$. Choose an arbitrary vertex $v_0 \in V(G)$ and set $\ell(v_0) = +a$. While there still exists a vertex $v \in V(G)$ with out-neighbours 
$v'$ and $v''$, such that $\ell(v') \neq \star$ and $\ell(v'') = \star$, update $\ell(v'')$ by setting $\ell(v'') = -\ell(v')$. When this algorithm stops, for each vertex $v \in V(G)$ with out-neighbours
$v'$ and $v''$ either (i) $\ell(v') = \ell(v'') = \star$ or (ii) $\ell(v') \neq \star$ and $\ell(v'') \neq \star$ holds.

If there exists a vertex $v \in V(G)$ with out-neighbours 
$v'$ and $v''$ such that $\ell(v') = \ell(v'') = +a$ or $\ell(v') = \ell(v'') = -a$, then it follows that $a = 0$ and thus $\ker G$ is trivial or contains a non-full vector. This contradicts the assumption
that $G$ is a bi-nut digraph. Therefore, if (ii) holds it must be that $\ell(v') \neq \ell(v'')$.

Let us now define the vector $\x\colon V(G) \to \mathbb{R}$ by
\begin{equation}
\x(v) = \begin{cases}
\phantom{-}0 & \text{if } \ell(v) = \star, \\
+1 & \text{if } \ell(v) = +a, \\
-1 & \text{if } \ell(v) = -a.
\end{cases}
\end{equation}
The vector $\x$ satisfies \eqref{eq:zerosum}, so $\x \in \ker G$. The existence of a vertex $v \in V(G)$ such that $\x(v) = 0$ would contradict the asuumption that $G$ is a bi-nut digraph. It follows that
$\x(v) \in \{ +1, -1\}$ for every $v \in V(G)$ and vector $\x$ spans $\ker G$.
The fact that half of the vertices carry entry  $+1$ and the other half $-1$ is straightforward to prove by standard combinatorial  `double counting'.

The same approach is used to prove the existence of the vector $\y$ from the proposition. The only difference is that we need to consider in-neighbours instead of out-neighbours.
\end{proof}

Note that the algorithm in the above proof emulates the standard pencil-and-paper approach~\cite{PencilPaper,Zivkovic1972} to determining 
the null space of a chemical graph.

\section{Constructions}
\label{sec:constructions}

In the realm of undirected graphs,
several constructions have been described for producing larger nut graphs from smaller.
Examples include the \emph{bridge construction} (insertion of two vertices on a bridge) \cite{SciGut1998}, the \emph{subdivision construction} (insertion of four vertices
on an edge) \cite{SciGut1998}, and the \emph{coalescence construction} (identification of a vertex in one nut graph with a vertex of another) \cite{ScirihaCoalescence}.
Multiplier constructions, so called because they produce a nut graph  from a $(2t)$-regular parent  graph,
in which the order of the new graph is a fixed multiple of the order of the parent, were also introduced previously \cite{NutOrbitPaper}.
Analogues can be defined for nut digraphs.

\subsection{Subdivision construction}

One possible construction that in some sense mimics subdivision is illustrated in Figure~\ref{fig:subdivisionGeneralisations}.

\begin{theorem}
\label{thm:subdivInter}
Let $G$ be a digraph and let $u \to v$ be an arc in $G$. Let $\widetilde{G}$ be the digraph obtained from $G$ by removing the arc $u \to v$, and adding four new vertices
$u', u'', v', v''$ and five new arcs $u \to v'', u' \to v'', u' \to v', u'' \to v', u'' \to v$. 
If $G$ is an inter-nut digraph, then $\widetilde{G}$ is an inter-nut digraph that is not an ambi-nut digraph.
\end{theorem}

\begin{figure}[!htbp]
\centering
\subcaptionbox{$G$}{
\begin{tikzpicture}[scale=1.5]
\clip (-2.3, 1.4) rectangle (0.3, -0.2);
\tikzstyle{vertex}=[draw,circle,font=\scriptsize,minimum size=5pt,inner sep=1pt,fill=white]
\tikzstyle{edge}=[draw,thick,decoration={markings,mark=at position 0.55 with {\arrow{Stealth}}},postaction=decorate]
\draw[thick,fill=cyan!30!white] (0, 1) .. controls ($ (0, 1) + (-50:2) $) and ($ (-2, 1) + (-130:2) $) .. (-2, 1) .. controls ($ (-2, 1) + (-50:0.5) $) and ($ (0, 1) + (-130:0.5) $) .. (0, 1);
\node[vertex,label=90:$v$] (v2) at (0, 1) {};
\node[vertex,label=90:$u$] (v3) at (-2, 1) {};
\path[edge] (v3) -- (v2) {};
\end{tikzpicture}
}
\qquad
\subcaptionbox{$\widetilde{G}$}{
\begin{tikzpicture}[scale=1.5]
\clip (-2.3, 1.4) rectangle (0.3, -0.2);
\tikzstyle{vertex}=[draw,circle,font=\scriptsize,minimum size=5pt,inner sep=1pt,fill=white]
\tikzstyle{edge}=[draw,thick,decoration={markings,mark=at position 0.8 with {\arrow{Stealth}}},postaction=decorate]
\draw[thick,fill=cyan!30!white] (0, 1) .. controls ($ (0, 1) + (-50:2) $) and ($ (-2, 1) + (-130:2) $) .. (-2, 1) .. controls ($ (-2, 1) + (-50:0.5) $) and ($ (0, 1) + (-130:0.5) $) .. (0, 1);
\node[vertex,label=90:$v$] (v2) at (0, 1) {};
\node[vertex,label=90:$u$] (v3) at (-2, 1) {};
\node[vertex,label=90:$v''$,fill=gray] (a1) at ({-2 + 0.4}, 1) {};
\node[vertex,label=90:$u'$,fill=gray] (a2) at ({-2 + 2*0.4}, 1) {};
\node[vertex,label=90:$v'$,fill=gray] (a3) at ({-2*0.4}, 1) {};
\node[vertex,label=90:$u''$,fill=gray] (a4) at ({-0.4}, 1) {};
\path[edge] (v3) -- (a1) {};
\path[edge] (a2) -- (a1) {};
\path[edge] (a2) -- (a3) {};
\path[edge] (a4) -- (a3) {};
\path[edge] (a4) -- (v2) {};
\end{tikzpicture}
}
\caption{A possible subdivision construction for digraphs.}
\label{fig:subdivisionGeneralisations}
\end{figure}
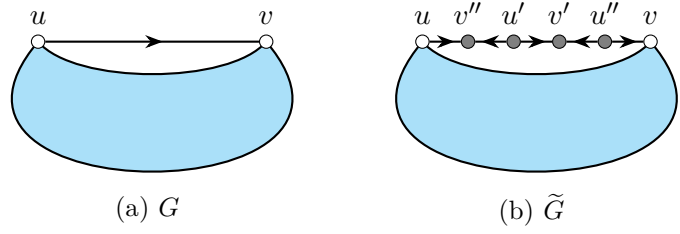

\begin{proof}
If $G$ is an inter-nut digraph, then there exists a full vector $\x \in \Ker G \cap \CoKer G$. The vector $\x$
can be extended to $\bx \in \mathbb{R}^{V(\widetilde{G})}$ by defining $\bx(w) = \x(w)$ for $w \in V(G)$ and $\bx(v'') = -\bx(v') = \x(v)$ and $\bx(u'') = -\bx(u') = \x(u)$. Clearly, $\bx$ is full
and $\bx \in \Ker\widetilde{G} \cap \CoKer \widetilde{G}$. This implies that $\widetilde{G}$ is an inter-core digraph. 

Take any non-trivial vector $\bq \in \Ker\widetilde{G} \cap \CoKer \widetilde{G}$. Equations \eqref{eq:zerosum} and \eqref{eq:zerosumn} imply that 
$\bq(v) = -\bq(v') = \bq(v'')$ and $\bq(u) = -\bq(u') = \bq(u'')$.
Let $\q$ be the restriction of $\bq$ to $G$. More precisely, $\q(w) = \bq(w)$ for $w \in V(G)$. It is easy to see that $\q \in \Ker G \cap \CoKer G$. As $G$ is an inter-nut digraph,
$\q = \lambda \x$ for some $\lambda \neq 0$. But then $\bq = \lambda \bx$. This means that $\bx$ spans $\Ker\widetilde{G} \cap \CoKer \widetilde{G}$ and $\widetilde{G}$ is thus an inter-nut digraph.

Moreover, let $\by, \bz \in \mathbb{R}^{V(\widetilde{G})}$ be vectors
defined as
\begin{equation}
\by(w) = \begin{cases}
 1 & w = u' \text{ or } w = u'', \\
0 & \text{otherwise};
\end{cases}
\quad\text{and}\quad
\bz(w) = \begin{cases}
 1 & w = v' \text{ or } w = v'', \\
0 & \text{otherwise}.
\end{cases}
\end{equation}
It is clear that $\by \in \Ker \widetilde{G}$ and $\by \notin \CoKer \widetilde{G}$, while $\bz \in \CoKer \widetilde{G}$ and $\bz \notin \Ker \widetilde{G}$.
Hence, $\widetilde{G}$ is not an ambi-nut digraph.
\end{proof}

Note that even if $G$ in Theorem~\ref{thm:subdivInter} is an ambi-nut digraph, $\widetilde{G}$ is not an ambi-nut digraph. The described
construction can only be used to produce larger inter-nut digraphs from existing inter-nut digraphs.

As we prove next, inter-nut and bi-nut digraphs have no cut-arcs (i.e. analogues of bridges).

\begin{proposition}
No inter-nut digraph contains a cut-arc.
\end{proposition}

\begin{proof}
Suppose that there exists an inter-nut digraph $\widehat{G}$ with a cut-arc $u \to w$.
When the cut-arc is removed, we obtain two connected components, $G_1$ and $G_2$, such that
$u \in V(G_1)$ and $w \in V(G_2)$. Since $\widehat{G}$ is an inter-nut digraph, there exists a full vector
$\widehat{\x} \in \Ker \widehat{G} \cap \CoKer \widehat{G}$. Let $\x$ denote the restriction of $\widehat{\x}$
to $G_2$, i.e. $\x(v) = \widehat{\x}(v)$ for all $v \in V(G_2)$.
Observe that $G_2$ is a dextro-nut digraph and that $\x$ is full and $\x \in \Ker G_2$. Further observe that
$\Su_\x^-(v) = 0$ for all $v \in V(G_2) \setminus \{w\}$. By Lemma~\ref{lem:ext}, $\Su_\x^-(w) = 0$ and 
$\x \in \CoKer G_2$. Observe that $G^-(w) = G_2^-(w) \cup \{u\}$. It follows that $\widehat{\x}(u) = 0$.
But this contradicts the fact that $\widehat{\x}$ is full. Therefore $\widehat{G}$ cannot be an inter-nut digraph.
\end{proof}

The following is a direct consequence of Theorem~\ref{thm:stronglyConnectedBi}.

\begin{corollary}
No bi-nut digraph contains a cut-arc.
\end{corollary}

Note that every ambi-nut digraph is a bi-nut digraph, and hence ambi-nut digraphs do not contain cut-arcs.
Dextro-nut digraphs and thus laevo-nut digraphs may, however, possess cut-arcs; see Figure~\ref{fig:dextrocoal}(a).

\subsection{Coalescence construction} 

Let $G$ and $H$ be two digraphs and $v$ a vertex of $G$ and $u$ a vertex of $H$. The \emph{coalescence} $(G, v) \odot (H, u)$ is defined as the digraph obtained from the disjoint union 
of $G$ and $H$  by identifying the vertices $v$ and $u$, where the `merged' vertex retains all the out-neighbours and all the in-neighbours of both $v$ and $u$.
Note that when the adjacency relations of $G$ and $H$ are both symmetric and irreflexive (i.e. when $G$ and $H$ are simple, undirected graphs), the coalescence as defined here
coincides with the coalescence of graphs defined in~\cite{ScirihaCoalescence}.

\begin{theorem}
\label{the:coal}
Let $G$ and $H$ be bi-nut digraphs, where $v^*\in V(G)$ and $u^*\in V(H)$. Then $(G, v^*) \odot (H, u^*)$ is a bi-nut digraph. Moreover, if both $G$ and $H$ are ambi-nut digraphs, then so is
$(G, v^*) \odot (H, u^*)$.
\end{theorem}

\begin{proof}
Let $\x,\y,\bx$, and $\by$ be such that $\Ker G = \langle \x \rangle$, $\CoKer G = \langle \bx \rangle$, $\Ker H = \langle \y \rangle$, and $\CoKer H = \langle \by \rangle$.
Without loss of generality we may assume that $\x(v^*) = \bx(v^*) = \y(u^*) = \by(u^*)$. Further, let $V$ denote the vertex-set of $C \coloneqq (G, v^*) \odot (H, u^*)$. We will abuse the notation slightly
 and denote the vertex of $C$ obtained by identifying $u^*$ and $v^*$  by $u^*$ or $v^*$, or, when we want to avoid ambiguity, by $\overline{uv}$, whichever is most convenient.

Define the functions $\z, \bz \in \RR^V$  by letting them coincide with the functions $\x$ and $\bx$ on $V(G)\setminus \{v^*\}$ and with $\y$ and $\by$ on $V(H) \setminus \{ u^*\}$, respectively. 
Further, let $\z(\overline{uv}) = \x(v^*) = \y(u^*)$ and
$\bz(\overline{uv}) = \bx(v^*)=\by(u^*)$. Observe that $\z \in \Ker C$ and $\bz \in \CoKer C$, showing that both $\Ker C$ and $\CoKer C$ contain a full vector. We now need to show that
they are both one-dimensional.

Suppose that $\w \in \Ker C$. 
Since the out-neighbourhood in $C$ of each vertex in $V(G) \setminus \{v^*\}$ coincides with its out-neighbourhood in $G$, it follows that 
$$
 \sum_{u\in C^+(v)} \w(u)\> =\>  \sum_{u\in G^+(v)} \w(u)\> =\>  0\>\> \hbox{ for all }\> v\in V(G) \setminus \{v^*\}.
$$
Now, since $\y \in \CoKer G$ and $\y(v^*) \not = 0$,  Lemma~\ref{lem:ext} implies that the restriction $\w_G$ of $\w$ to $V(G)$ lies in $\Ker G$. Since $G$ is a dextro-nut digraph, it follows that
$ \w_G  = \lambda  \x$ for some $\lambda \in \RR$, $\lambda \not = 0$. By an analogous argument, one can show that the restriction $\w_H$ of $\w$ to $V(H)$ satisfies $\w_H = \mu \y$.
Since $\x(v^*) = \y(u^*) = \w(\overline{uv})$, it follows that $\lambda = \mu$ and thus $\w = \lambda \z$. This shows that $\Ker C$ is indeed one-dimensional and spanned by $\z$.
The proof that $\CoKer C$ is also one-dimensional can be obtained by applying the above to the reverse digraphs $G^R$ and $H^R$. This shows that $C$ is indeed a bi-nut digraph,
the cokernel of which is spanned by the function $\bz\colon V \to \RR$, coinciding with $\bx$ on $V(G)$ and with $\by$ on $V(H)$.

If $G$ and $H$ are both ambi-nut digraphs, then $\bx = \x$ and $\by = \y$, and thus $\bz =\z$, showing that $C$ is an ambi-nut digraph.
\end{proof}

Note that  the coalescence of two dextro-nut digraphs is not necessarily a dextro-nut digraph. 
Figure~\ref{fig:dextrocoal}(a) shows an example where coalescence of dextro-nut digraphs from Figure~\ref{fig:smallGeneral}(a) 
and~\ref{fig:smallGeneral}(b) produces a dextro-nut digraph (which contains a cut-arc but no leaves).
On the other hand, 
Figure~\ref{fig:dextrocoal}(b) shows an example where a coalescence of two copies of the dextro-nut 
digraph from Figure~\ref{fig:smallGeneral}(a) is not a dextro-nut digraph,
as confirmed by the existence of a non-full kernel eigenvector of the resulting digraph.

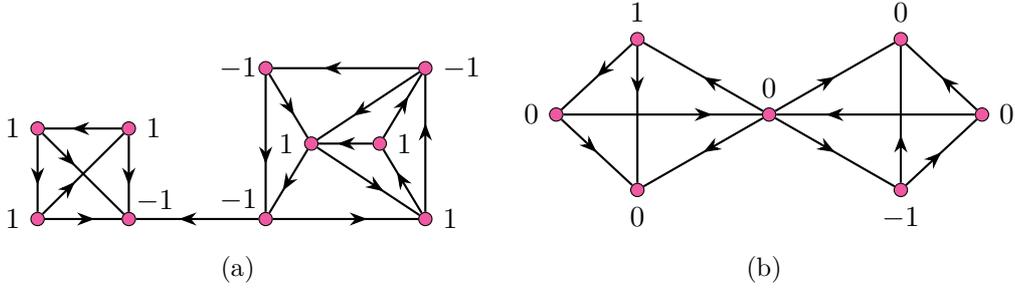
\begin{figure}[!htbp]
\centering
\subcaptionbox{}{
\begin{tikzpicture}[scale=1.0,rotate=90]
\tikzstyle{vertex}=[draw,circle,font=\scriptsize,minimum size=5pt,inner sep=1pt,fill=magenta!80!white]
\tikzstyle{edge}=[draw,thick,decoration={markings,mark=at position 0.5 with {\arrow{Stealth}}},postaction=decorate]
\tikzstyle{sedge}=[draw,thick,decoration={markings,mark=at position 0.4 with {\arrow{Stealth}}},postaction=decorate]
\tikzstyle{ledge}=[draw,thick,decoration={markings,mark=at position 0.65 with {\arrow{Stealth}}},postaction=decorate]
\node[vertex,label={[yshift=-3pt,xshift=2pt]120:$-1$}] (v5) at (-1, 0) {};
\node[vertex,label={[xshift=3.2pt]180:$-1$}] (v1) at (1, 0) {};
\node[vertex,label={[yshift=0pt]180:$1$}] (v6) at (0, -0.6) {};
\node[vertex,label={[yshift=0pt]0:$1$}] (v3) at (0, -1.5) {};
\node[vertex,label=0:$1$] (v0) at (-1, -2.1) {};
\node[vertex,label=0:$-1$] (v4) at (1, -2.1) {};
\node[vertex,label={[yshift=-3pt,xshift=-3pt]45:$-1$}] (v2) at (-1, 1.8) {};
\node[vertex,label=0:$1$] (v7) at ($ (v2) + (1.2, 0) $) {};
\node[vertex,label=180:$1$] (v8) at ($ (v2) + (1.2, 1.2) $) {};
\node[vertex,label=180:$1$] (v9) at ($ (v2) + (0, 1.2) $) {};
\path[ledge] (v3) -- (v4);
\path[ledge] (v3) -- (v6);
\path[ledge] (v0) -- (v3);
\path[ledge] (v0) -- (v4);
\path[ledge] (v4) -- (v1);
\path[ledge] (v4) -- (v6);
\path[ledge] (v1) -- (v5);
\path[ledge] (v1) -- (v6);
\path[ledge] (v6) -- (v0);
\path[ledge] (v6) -- (v5);
\path[ledge] (v5) -- (v0);
\path[ledge] (v5) -- (v2);
\path[ledge] (v7) -- (v2);
\path[sedge] (v8) -- (v2);
\path[ledge] (v9) -- (v2);
\path[ledge] (v7) -- (v8);
\path[ledge] (v8) -- (v9);
\path[sedge] (v9) -- (v7);
\end{tikzpicture}
}
\subcaptionbox{}{
\begin{tikzpicture}[scale=2]
\tikzstyle{vertex}=[draw,circle,font=\scriptsize,minimum size=5pt,inner sep=1pt,fill=magenta!80!white]
\tikzstyle{edge}=[draw,thick,decoration={markings,mark=at position 0.5 with {\arrow{Stealth}}},postaction=decorate]
\tikzstyle{sedge}=[draw,thick,decoration={markings,mark=at position 0.35 with {\arrow{Stealth}}},postaction=decorate]
\tikzstyle{ledge}=[draw,thick,decoration={markings,mark=at position 0.75 with {\arrow{Stealth}}},postaction=decorate]
\node[vertex,label=90:$0$] (v1) at (0, 0) {};
\node[vertex,label=90:$1$] (v2) at (150:1) {};
\node[vertex,label=180:$0$] (v3) at (180:1.4) {};
\node[vertex,label=-90:$0$] (v4) at (210:1) {};
\node[vertex,label=90:$0$] (u2) at (30:1) {};
\node[vertex,label=0:$0$] (u3) at (0:1.4) {};
\node[vertex,label=-90:$-1$] (u4) at (-30:1) {};
\path[edge] (v1) -- (v2);
\path[edge] (v1) -- (v4);
\path[ledge] (v3) -- (v1);
\path[edge] (v1) -- (u2);
\path[edge] (v1) -- (u4);
\path[ledge] (u3) -- (v1);
\path[edge] (v2) -- (v3);
\path[edge] (v3) -- (v4);
\path[sedge] (v2) -- (v4);
\path[edge] (u4) -- (u3);
\path[sedge] (u4) -- (u2);
\path[edge] (u3) -- (u2);
\end{tikzpicture}}
\caption{Coalescence of two dextro-nut digraphs may or may not result in a dextro-nut digraph. (a) Dextro-nut digraphs from Figure~\ref{fig:smallGeneral}(a) and~\ref{fig:smallGeneral}(b) yield a dextro-nut digraph; but (b) two copies of the dextro-nut digraph from Figure~\ref{fig:smallGeneral}(a) (at one of the vertices bearing $1$ in the displayed kernel eigenvector) do not.}
\label{fig:dextrocoal}
\end{figure}

\subsection{Cross-over construction}

Let $u \to v$ be an arc of a digraph $G$ and let $s \to t$ be an arc of a digraph $H$. The {\em cross-over} of $G$ and $H$ with respect to $u \to v$ and $s \to t$, denoted by $(G, u,v) \bowtie (H, s,t)$,
is defined as the digraph obtained from the disjoint union of $G$ and $H$ by deleting the arcs $u \to v$ and $s \to t$ and adding the arcs $u \to t$ and $s \to v$; see Figure~\ref{fig:schemcross}.

\begin{figure}[!htbp]
\centering
\subcaptionbox{\label{fig:dextrocross1}}{
\begin{tikzpicture}[scale=2]
\clip (-0.2, 0.3) rectangle (2.7, -1.3);
\tikzstyle{vertex}=[draw,circle,font=\scriptsize,minimum size=5pt,inner sep=1pt,fill=white]
\tikzstyle{edge}=[draw,thick,decoration={markings,mark=at position 0.5 with {\arrow{Stealth}}},postaction=decorate]
\tikzstyle{sedge}=[draw,thick,decoration={markings,mark=at position 0.35 with {\arrow{Stealth}}},postaction=decorate]
\tikzstyle{ledge}=[draw,thick,decoration={markings,mark=at position 0.75 with {\arrow{Stealth}}},postaction=decorate]
\draw[thick,fill=cyan!30!white] (1, 0) .. controls ($ (1, 0) + (140:2) $) and ($ (1, -1) + (-140:2) $) .. (1, -1) .. controls ($ (1, -1) + (140:0.5) $) and ($ (1, 0) + (-140:0.5) $) .. (1, 0);
\draw[thick,fill=cyan!30!white] (1.5, 0) .. controls ($ (1.5, 0) + (40:2) $) and ($ (1.5, -1) + (-40:2) $) .. (1.5, -1)  .. controls ($ (1.5, -1) + (40:0.5) $) and ($ (1.5, 0) + (-40:0.5) $) .. (1.5, 0);
\node[vertex,label=90:$v$] (v2) at (1, 0) {};
\node[vertex,label=-90:$u$] (v3) at (1, -1) {};
\node[vertex,label=90:$t$] (u2) at (1.5, 0) {};
\node[vertex,label=-90:$s$] (u3) at (1.5, -1) {};
\path[edge] (u3) -- (u2) {};
\path[edge] (v3) -- (v2) {};
\node[fill=none,draw=none] at (0.3, -0.5) {$G$};
\node[fill=none,draw=none] at (2.2, -0.5) {$H$};
\end{tikzpicture}}
\qquad
\subcaptionbox{\label{fig:dextrocross2}}{
\begin{tikzpicture}[scale=2]
\clip (-0.2, 0.3) rectangle (2.7, -1.3);
\tikzstyle{vertex}=[draw,circle,font=\scriptsize,minimum size=5pt,inner sep=1pt,fill=white]
\tikzstyle{edge}=[draw,thick,decoration={markings,mark=at position 0.5 with {\arrow{Stealth}}},postaction=decorate]
\tikzstyle{sedge}=[draw,thick,decoration={markings,mark=at position 0.35 with {\arrow{Stealth}}},postaction=decorate]
\tikzstyle{ledge}=[draw,thick,decoration={markings,mark=at position 0.75 with {\arrow{Stealth}}},postaction=decorate]
\draw[thick,fill=cyan!30!white] (1, 0) .. controls ($ (1, 0) + (140:2) $) and ($ (1, -1) + (-140:2) $) .. (1, -1) .. controls ($ (1, -1) + (140:0.5) $) and ($ (1, 0) + (-140:0.5) $) .. (1, 0);
\draw[thick,fill=cyan!30!white] (1.5, 0) .. controls ($ (1.5, 0) + (40:2) $) and ($ (1.5, -1) + (-40:2) $) .. (1.5, -1)  .. controls ($ (1.5, -1) + (40:0.5) $) and ($ (1.5, 0) + (-40:0.5) $) .. (1.5, 0);
\node[vertex,label=90:$v$] (v2) at (1, 0) {};
\node[vertex,label=-90:$u$] (v3) at (1, -1) {};
\node[vertex,label=90:$t$] (u2) at (1.5, 0) {};
\node[vertex,label=-90:$s$] (u3) at (1.5, -1) {};
\path[sedge] (u3) .. controls ($ (u3) + (0, 0.5) $) and ($ (v2) + (0, -0.5) $) .. (v2);
\path[sedge] (v3) .. controls ($ (v3) + (0, 0.5) $) and ($ (u2) + (0, -0.5) $) .. (u2);
\end{tikzpicture}}
\caption{Schematic of the cross-over construction. Panel (a) shows the disjoint union of digraph $G$ and $H$ that contain
arcs $u \to v$ and $s \to t$, respectively. Panel (b) shows the digraph $(G, u,v) \bowtie (H, s,t)$.}
\label{fig:schemcross}
\end{figure}
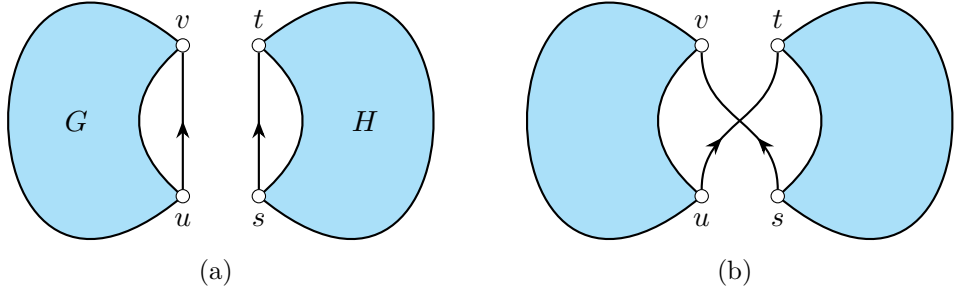

\begin{theorem}
Let $G$ and $H$ be two ambi-nut digraphs with 
$\Ker G = \langle \x \rangle=\CoKer G$ and 
$\Ker H = \langle \y \rangle=\CoKer H$.
If $u^* \to v^*$ is an arc of $G$ and $s^* \to t^*$ an arc of $H$ such that
\begin{equation}
\label{eq:cross}
\x(u^*) = \y(s^*),\quad
\x(v^*) = \y(t^*),
\end{equation}
then the digraph $(G, u^*,v^*) \bowtie (H, s^*,t^*)$ is an ambi-nut digraph.
\end{theorem}

\begin{proof}
Let $C \coloneqq (G, u^*,v^*) \bowtie (H, s^*,t^*)$. Without loss of generality, we shall assume that $V(G)$ and $V(H)$ are disjoint, so that $V(C)$ is simply the union of $V(G)$ and $V(H)$.
Further, let $\z\colon V(C) \to \RR$ be defined by letting $\z(u)$ to be equal to $\x(u)$ for all $u\in V(G)$ and equal to $\y(u)$ for all $u\in V(H)$. By the definition of
the cross-over and by the assumption (\ref{eq:cross}), it follows that $\z$ is a full vector contained both in $\Ker C$ and in $\CoKer C$.
What remains to show is that $\Ker C$ and $\CoKer C$ are both $1$-dimensional.

Let $\w\colon V(C) \to \RR$ be an element of $\Ker C$ such that $\w(u^*) = \x(u^*)$ and let $\w_G$ be its restriction to $V(G)$.
Since the out-neighbourhood $C^+(v)$ is contained in $V(G)$ for every $v\in V(G) \setminus \{u^*\}$ (and is in fact equal to $G^+(v)$,
we see that $\w_G$ satisfies the local condition $\Su_{\w_G}^+(v) = 0$ for all $v\in V(G)$ except possibly for $v=u^*$. However, the existence of
a full vector in $\CoKer G$, together with Lemma~\ref{lem:ext},  implies that this condition is satisfied also for $v=u^*$, and thus that $\w_G \in \Ker G$.
Since $\Ker G$ is a one-dimensional space spanned by $\x$ and since $\w_G$ coincides with $\x$ in one vertex, we see that $\w_G = \x$.
Using an analogous argument, one can also show that the restriction of $\w$ to $V(H)$ equals $\y$, implying that $\w=\z$, as required.

Moreover, applying the above to the reverse digraphs $G^R$ and $H^R$, one can show that $\CoKer C$ is one-dimensional (and thus spanned by $\z$),
proving that $C$ is an ambi-nut.
\end{proof}

Note that a minor refinement of the above proof (by separately investigating the kernel and co-kernel) yields the slightly stronger claim that the cross-over of two bi-nut digraphs $G$ and $H$ with their kernel and co-kernel 
vectors being compatible on the arcs $u^* \to v^*$ and $s^* \to t^*$ is also a bi-nut. However, as the example in Figure~\ref{fig:dextrocross} shows, the cross-over of dextro-nut digraphs is not necessarily a dextro-nut digraph.

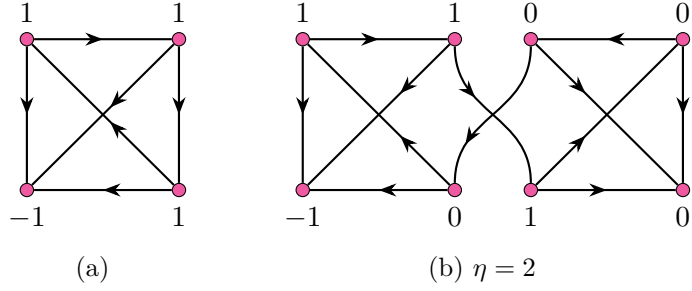
\begin{figure}[!htbp]
\centering
\subcaptionbox{\label{fig:dextrocross1}}{
\begin{tikzpicture}[scale=2]
\tikzstyle{vertex}=[draw,circle,font=\scriptsize,minimum size=5pt,inner sep=1pt,fill=magenta!80!white]
\tikzstyle{edge}=[draw,thick,decoration={markings,mark=at position 0.5 with {\arrow{Stealth}}},postaction=decorate]
\tikzstyle{sedge}=[draw,thick,decoration={markings,mark=at position 0.45 with {\arrow{Stealth}}},postaction=decorate]
\node[vertex,label=90:$1$] (v1) at (0, 0) {};
\node[vertex,label=90:$1$] (v2) at (1, 0) {};
\node[vertex,label=-90:$1$] (v3) at (1, -1) {};
\node[vertex,label=-90:$-1$] (v4) at (0, -1) {};
\path[edge] (v1) -- (v2);
\path[edge] (v2) -- (v3);
\path[sedge] (v3) -- (v1);
\path[edge] (v1) -- (v4);
\path[sedge] (v2) -- (v4);
\path[edge] (v3) -- (v4);
\end{tikzpicture}}
\qquad
\subcaptionbox{$\eta = 2$\label{fig:dextrocross2}}{
\begin{tikzpicture}[scale=2]
\tikzstyle{vertex}=[draw,circle,font=\scriptsize,minimum size=5pt,inner sep=1pt,fill=magenta!80!white]
\tikzstyle{edge}=[draw,thick,decoration={markings,mark=at position 0.5 with {\arrow{Stealth}}},postaction=decorate]
\tikzstyle{sedge}=[draw,thick,decoration={markings,mark=at position 0.35 with {\arrow{Stealth}}},postaction=decorate]
\tikzstyle{ledge}=[draw,thick,decoration={markings,mark=at position 0.75 with {\arrow{Stealth}}},postaction=decorate]
\node[vertex,label=90:$1$] (v1) at (0, 0) {};
\node[vertex,label=90:$1$] (v2) at (1, 0) {};
\node[vertex,label=-90:$0$] (v3) at (1, -1) {};
\node[vertex,label=-90:$-1$] (v4) at (0, -1) {};
\node[vertex,label=90:$0$] (u1) at (2.5, 0) {};
\node[vertex,label=90:$0$] (u2) at (1.5, 0) {};
\node[vertex,label=-90:$1$] (u3) at (1.5, -1) {};
\node[vertex,label=-90:$0$] (u4) at (2.5, -1) {};
\path[edge] (v1) -- (v2);
\path[sedge] (v3) -- (v1);
\path[edge] (v1) -- (v4);
\path[sedge] (v2) -- (v4);
\path[edge] (v3) -- (v4);
\path[sedge] (v2) .. controls ($ (v2) + (0, -0.5) $) and ($ (u3) + (0, 0.5) $) .. (u3);
\path[ledge] (u2) .. controls ($ (u2) + (0, -0.5) $) and ($ (v3) + (0, 0.5) $) .. (v3);
\path[edge] (u1) -- (u2);
\path[sedge] (u3) -- (u1);
\path[edge] (u1) -- (u4);
\path[sedge] (u2) -- (u4);
\path[edge] (u3) -- (u4);
\end{tikzpicture}}
\caption{Crossover of two copies of the dextro-nut digraph in (a) results in the digraph in (b), which is not a dextro-nut
as witnessed by the presence of zero entries in the displayed non-full kernel vector.}
\label{fig:dextrocross}
\end{figure}

\subsection{Multiplier constructions}

In \cite[Section 4.2]{NutOrbitPaper}  triangle- and pentagon-multiplier constructions for nut graphs were studied.
These constructions proved useful for obtaining nut graphs from any regular graph of even degree,
and were crucial for proving that any finite group can be realised by a nut graph~\cite{GroupNutPaper}.
Here, we generalise this theory to ambi-nut digraphs.

We begin by defining the notion of a \emph{gadget}. Let $M_{(i)}$ denote the matrix obtained from $M$ by deleting the $i$-th row.
For the present purpose, then:

\begin{definition}
\label{def:theMagicVector}
A gadget is a pair $(G, r)$, where $G$ is a digraph and $r \in V(G)$, with the property that there exists a full vector $\x$ that spans the respective
($1$-dimensional) kernels of both $A(G)_{(r)}$ and $A(G^R)_{(r)}$. The vertex $r$ will be called the \emph{root} of the gadget.
\end{definition}

Note that the vector $\x$ in the above definition satisfies $\Su_\x^+(v) = \Su_\x^-(v) = 0$ for all $v \in V(G) \setminus \{r\}$. Recall that $\Su_\x^+(v)$ and $\Su_\x^-(v)$ are defined in \eqref{eq:inAndOutLocal}.

\begin{proposition}
\label{prop:inDemandOutDemand}
Let $(G, r)$ be a gadget. Then $\Su_\x^+(r) = \Su_\x^-(r)$, where $\x$ is the full vector from Definition~\ref{def:theMagicVector}.
\end{proposition}

\begin{proof}
We shall use a double-counting approach, as in the proof of Lemma~\ref{lem:ext}.
Let
\begin{equation}
\label{eq:normx}
M :=  \sum_{u,v \in V(G) \atop v\to u} \x(v)\x(u).
\end{equation}
On one hand, we see that
\begin{equation}
M = \sum_{u \in V(G)} \sum_{v\in G^-(u)} \x(v)\x(u) = \sum_{u \in V(G)} \x(u) \Su_\x^-(u) = \x(r) \Su_\x^-(r).
\end{equation}
On the other hand
\begin{equation}
M = \sum_{v \in V(G)} \sum_{u\in G^+(v)} \x(v)\x(u) = \sum_{v \in V(G)}  \x(v) \Su_\x^+(v)  = \x(r) \Su_\x^+(r).
\end{equation}
Since, by assumption, $\x$ is a full vector, it follows that
$\Su_\x^-(r) = \Su_\x^+(r)$, as claimed.
\end{proof}

\begin{definition}
The \emph{demand} of a gadget $(G, r)$, denoted $\dem (G, r)$, is the value $-\Su_\x^+(r) / \x(r)$, where $\x$ is the full vector from Definition~\ref{def:theMagicVector}.
\end{definition}

The demand is well defined, since $\x(r) \neq 0$ and $\x$ is unique up to scalar multiplication.
By Proposition~\ref{prop:inDemandOutDemand}, $\dem (G, r) = -\Su_\x^-(r) / \x(r)$. Note that $(G, r)$ is a gadget with $\dem(G, r) = 0$ if $G$ is an 
ambi-nut digraph and $r$ is an arbitrary vertex of $G$. Moreover, the following holds:

\begin{proposition}
\label{prop:gadgetsAreAmbi}
Let $(G, r)$ be a gadget with $\dem(G, r) = 0$. Then $G$ is an ambi-nut digraph. 
\end{proposition}

\begin{proof}
Let $\x$ be the full vector from Definition~\ref{def:theMagicVector}. We already noted that  $\Su_\x^+(v) = \Su_\x^-(v) = 0$ for all $v \in V(G) \setminus \{r\}$.
From the definition of demand and Proposition~\ref{prop:inDemandOutDemand} it follows that $\Su_\x^+(r) = \Su_\x^-(r) = 0$. Hence, \eqref{eq:zerosum} and 
\eqref{eq:zerosumn} hold, so $\x \in \Ker G$ and $\x \in \CoKer G$. Suppose that $\eta(G) > 1$. Then there exists $\y \in \Ker G$ which is linearly independent from $\x$.
But then $\y$ is also in the kernel of $A(G)_{(r)}$, contradicting the requirement of Definition~\ref{def:theMagicVector} that $\Ker A(G)_{(r)}$ is $1$-dimensional.
Thus, $\eta(G) = 1$ and $G$ is therefore an ambi-nut digraph.
\end{proof}

Figure~\ref{fig:bunchOfGadgets} shows examples of gadgets with demands from the set $\{1, -1/2, -1/3, -1, 2/3\}$.
Note that the demand of a gadget is always a rational number.

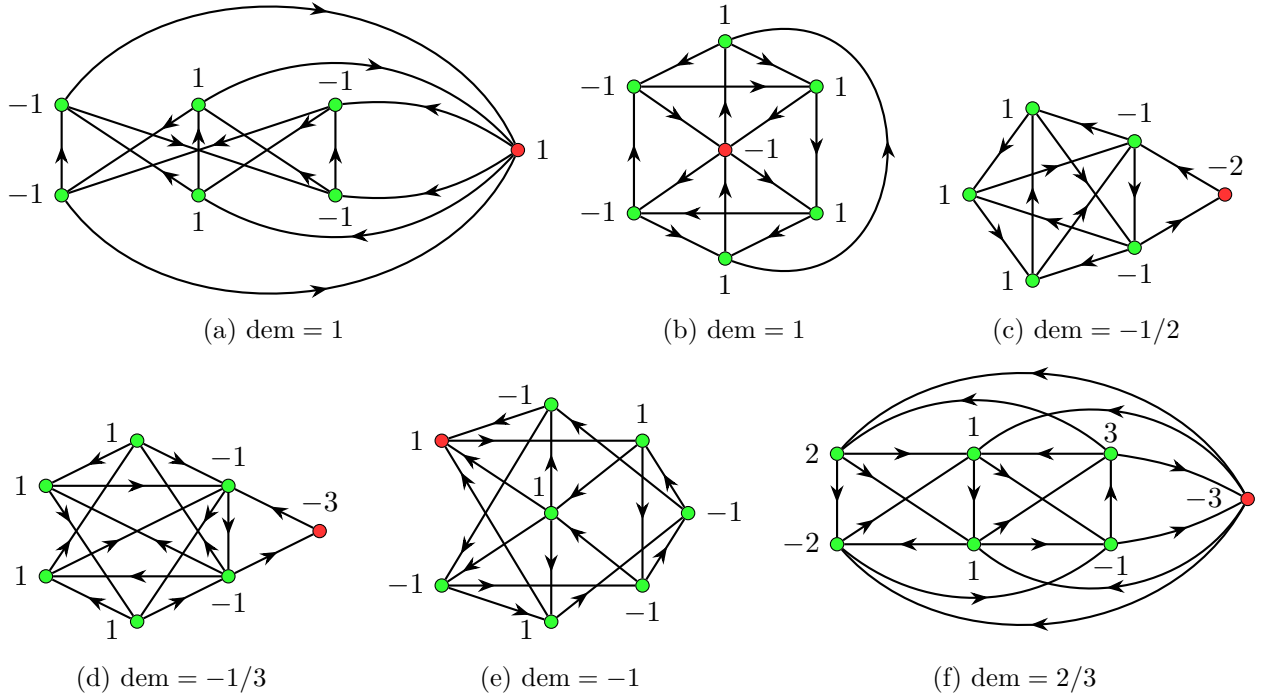
\begin{figure}[!h]
\centering
\subcaptionbox{$\dem = 1$}{
\begin{tikzpicture}[scale=1.2]
\clip (-0.6, 1.15) rectangle (5.35, -2.15);
\tikzstyle{vertex}=[draw,circle,font=\scriptsize,minimum size=5pt,inner sep=1pt,fill=green!80!white]
\tikzstyle{edge}=[draw,thick,decoration={markings,mark=at position 0.55 with {\arrow{Stealth}}},postaction=decorate]
\tikzstyle{sedge}=[draw,thick,decoration={markings,mark=at position 0.25 with {\arrow{Stealth}}},postaction=decorate]
\tikzstyle{medge}=[draw,thick,decoration={markings,mark=at position 0.45 with {\arrow{Stealth}}},postaction=decorate]
\tikzstyle{ledge}=[draw,thick,decoration={markings,mark=at position 0.8 with {\arrow{Stealth}}},postaction=decorate]
\node[vertex,label=180:$-1$] (v4) at (0, 0) {};
\node[vertex,label=90:$1$] (v3) at (1.5, 0) {};
\node[vertex,label={[yshift=-2pt]90:$-1$}] (v5) at (3, 0) {};
\node[vertex,label=180:$-1$] (v2) at (0, -1) {};
\node[vertex,label=-90:$1$] (v0) at (1.5, -1) {};
\node[vertex,label={[yshift=2pt]-90:$-1$}] (v1) at (3, -1) {};

\node[vertex,label=0:$1$,fill=red!80!white] (v6) at (5, -0.5) {};
\path[medge] (v4) -- (v1) {};
\path[sedge] (v0) -- (v4) {};
\path[edge] (v2) -- (v4) {};

\path[sedge] (v3) -- (v2) {};
\path[ledge] (v0) -- (v3) {};
\path[sedge] (v1) -- (v3) {};

\path[edge] (v1) -- (v5) {};
\path[medge] (v5) -- (v2) {};
\path[sedge] (v5) -- (v0) {};
\path[edge] (v6) to[bend left=20] (v1) {};
\path[edge] (v6) to[bend right=20] (v5) {};
\path[edge] (v6) to[bend left=40] (v0) {};
\path[edge] (v4) to[bend left=60] (v6) {};
\path[edge] (v3) to[bend left=40] (v6) {};
\path[edge] (v2) to[bend right=60] (v6) {};
\end{tikzpicture}
}
\subcaptionbox{$\dem = 1$}{
\begin{tikzpicture}[scale=1.2]
\clip (-0.6, 0.95) rectangle (2.9, -2.35);
\tikzstyle{vertex}=[draw,circle,font=\scriptsize,minimum size=5pt,inner sep=1pt,fill=green!80!white]
\tikzstyle{cedge}=[draw,thick,decoration={markings,mark=at position 0.52 with {\arrow{Stealth}}},postaction=decorate]
\tikzstyle{edge}=[draw,thick,decoration={markings,mark=at position 0.55 with {\arrow{Stealth}}},postaction=decorate]
\tikzstyle{ledge}=[draw,thick,decoration={markings,mark=at position 0.75 with {\arrow{Stealth}}},postaction=decorate]
\tikzstyle{medge}=[draw,thick,decoration={markings,mark=at position 0.45 with {\arrow{Stealth}}},postaction=decorate]
\node[vertex,label=180:$-1$] (v3) at (0, 0) {};
\node[vertex,label=0:$1$] (v1) at (2, 0) {};
\node[vertex,label=90:$1$] (v5) at (1, 0.5) {};
\node[vertex,label=180:$-1$] (v0) at (0, -1.4) {};
\node[vertex,label=0:$1$] (v4) at (2, -1.4) {};
\node[vertex,label=-90:$1$] (v2) at (1, {-1.4-0.5}) {};
\node[vertex,label=0:$-1$,fill=red!80!white] (v6) at (1, -0.7) {};
\path[ledge] (v3) -- (v1) {};
\path[edge] (v5) -- (v1) {};
\path[edge] (v5) -- (v3) {};
\path[ledge] (v4) -- (v0) {};
\path[edge] (v4) -- (v2) {};
\path[edge] (v0) -- (v2) {};
\path[edge] (v0) -- (v3) {};
\path[edge] (v1) -- (v4) {};
\path[cedge] (v2) .. controls ($ (v2) + (-20:2.5) $) and ($ (v5) + (20:2.5) $) ..  (v5) {};
\path[edge] (v6) -- (v0) {};
\path[edge] (v6) -- (v4) {};
\path[medge] (v6) -- (v5) {};
\path[edge] (v1) -- (v6) {};
\path[ledge] (v2) -- (v6) {};
\path[edge] (v3) -- (v6) {};
\end{tikzpicture}
}
\subcaptionbox{$\dem = -1/2$}{
\begin{tikzpicture}[scale=1.2]
\tikzstyle{vertex}=[draw,circle,font=\scriptsize,minimum size=5pt,inner sep=1pt,fill=green!80!white]
\tikzstyle{edge}=[draw,thick,decoration={markings,mark=at position 0.55 with {\arrow{Stealth}}},postaction=decorate]
\node[vertex,label=180:$1$] (v0) at (180:1) {};
\node[vertex,label=180:$1$] (v2) at ({180 + 72}:1) {};
\node[vertex,label=-90:$-1$] (v5) at ({180 + 2*72}:1) {};
\node[vertex,label=90:$-1$] (v4) at ({180 + 3*72}:1) {};
\node[vertex,label=180:$1$] (v3) at ({180 + 4*72}:1) {};
\node[vertex,label=90:$-2$,fill=red!80!white] (v1) at (0:1.8) {};
\path[edge] (v0) -- (v2) {};
\path[edge] (v5) -- (v2) {};
\path[edge] (v4) -- (v5) {};
\path[edge] (v4) -- (v3) {};
\path[edge] (v3) -- (v0) {};
\path[edge] (v5) -- (v0) {};
\path[edge] (v0) -- (v4) {};
\path[edge] (v2) -- (v4) {};
\path[edge] (v2) -- (v3) {};
\path[edge] (v3) -- (v5) {};
\path[edge] (v1) -- (v4) {};
\path[edge] (v5) -- (v1) {};
\end{tikzpicture}
}

\subcaptionbox{$\dem = -1/3$}{
\begin{tikzpicture}[scale=1.2]
\tikzstyle{vertex}=[draw,circle,font=\scriptsize,minimum size=5pt,inner sep=1pt,fill=green!80!white]
\tikzstyle{edge}=[draw,thick,decoration={markings,mark=at position 0.55 with {\arrow{Stealth}}},postaction=decorate]
\tikzstyle{sedge}=[draw,thick,decoration={markings,mark=at position 0.25 with {\arrow{Stealth}}},postaction=decorate]
\tikzstyle{medge}=[draw,thick,decoration={markings,mark=at position 0.40 with {\arrow{Stealth}}},postaction=decorate]
\node[vertex,label=90:$-1$] (v5) at (0, 1) {};
\node[vertex,label=-90:$-1$] (v6) at (0, 0) {};
\node[vertex,label={[yshift=3pt]180:$1$}] (v3) at (-1, 1.5) {};
\node[vertex,label={[yshift=-3pt]180:$1$}] (v4) at (-1, -0.5) {};
\node[vertex,label=180:$1$] (v1) at (-2, 1) {};
\node[vertex,label=180:$1$] (v0) at (-2, 0) {};
\node[vertex,label=90:$-3$,fill=red!80!white] (v2) at (1, 0.5) {};
\path[edge] (v3) -- (v1) {};
\path[sedge] (v1) -- (v4) {};
\path[edge] (v4) -- (v0) {};
\path[sedge] (v0) -- (v3) {};
\path[edge] (v6) -- (v0) {};
\path[medge] (v6) -- (v1) {};
\path[edge] (v1) -- (v5) {};
\path[medge] (v0) -- (v5) {};
\path[sedge] (v6) -- (v3) {};
\path[edge] (v3) -- (v5) {};
\path[sedge] (v5) -- (v4) {};
\path[edge] (v4) -- (v6) {};
\path[edge] (v5) -- (v6) {};
\path[edge] (v2) -- (v5) {};
\path[edge] (v6) -- (v2) {};
\end{tikzpicture}
}
\subcaptionbox{$\dem = -1$}{
\begin{tikzpicture}[scale=1.2]
\tikzstyle{vertex}=[draw,circle,font=\scriptsize,minimum size=5pt,inner sep=1pt,fill=green!80!white]
\tikzstyle{edge}=[draw,thick,decoration={markings,mark=at position 0.55 with {\arrow{Stealth}}},postaction=decorate]
\tikzstyle{sedge}=[draw,thick,decoration={markings,mark=at position 0.25 with {\arrow{Stealth}}},postaction=decorate]
\tikzstyle{ledge}=[draw,thick,decoration={markings,mark=at position 0.85 with {\arrow{Stealth}}},postaction=decorate]
\tikzstyle{lledge}=[draw,thick,decoration={markings,mark=at position 0.90 with {\arrow{Stealth}}},postaction=decorate]
\node[vertex,label={[yshift=3pt]180:$-1$}] (v1) at (0,1.2) {};
\node[vertex,label={[xshift=3pt]110:$1$}] (v7) at (0,0) {};
\node[vertex,label={[yshift=-3pt]180:$1$}] (v2) at (0,-1.2) {};
\node[vertex,label=90:$1$] (v0) at (1,0.8) {};
\node[vertex,label=0:$-1$] (v6) at (1.5,0) {};
\node[vertex,label=-90:$-1$] (v3) at (1,-0.8) {};
\node[vertex,label=180:$1$,fill=red!80!white] (v4) at (-1.2,0.8) {};
\node[vertex,label=180:$-1$] (v5) at (-1.2,-0.8) {};
\path[edge] (v7) -- (v1) {};
\path[edge] (v7) -- (v2) {};
\path[ledge] (v0) -- (v7) {};
\path[ledge] (v3) -- (v7) {};
\path[edge] (v0) -- (v3) {};
\path[edge] (v3) -- (v6) {};
\path[edge] (v6) -- (v0) {};
\path[sedge] (v2) -- (v6) {};
\path[lledge] (v6) -- (v1) {};
\path[ledge] (v5) -- (v2) {};
\path[sedge] (v5) -- (v3) {};
\path[edge] (v1) -- (v4) {};
\path[sedge] (v4) -- (v0) {};
\path[ledge] (v7) -- (v5) {};
\path[ledge] (v7) -- (v4) {};
\path[ledge] (v1) -- (v5) {};
\path[ledge] (v2) -- (v4) {};
\end{tikzpicture}
}
\subcaptionbox{$\dem = 2/3$}{
\begin{tikzpicture}[scale=1.2]
\clip (-0.6, 1.15) rectangle (4.60, -2.15);
\tikzstyle{vertex}=[draw,circle,font=\scriptsize,minimum size=5pt,inner sep=1pt,fill=green!80!white]
\tikzstyle{edge}=[draw,thick,decoration={markings,mark=at position 0.55 with {\arrow{Stealth}}},postaction=decorate]
\tikzstyle{sedge}=[draw,thick,decoration={markings,mark=at position 0.25 with {\arrow{Stealth}}},postaction=decorate]
\tikzstyle{medge}=[draw,thick,decoration={markings,mark=at position 0.45 with {\arrow{Stealth}}},postaction=decorate]
\tikzstyle{ledge}=[draw,thick,decoration={markings,mark=at position 0.7 with {\arrow{Stealth}}},postaction=decorate]
\node[vertex,label=180:$2$] (v2) at (0, 0) {};
\node[vertex,label=90:$1$] (v4) at (1.5, 0) {};
\node[vertex,label={[yshift=-2pt]90:$3$}] (v0) at (3, 0) {};
\node[vertex,label=180:$-2$] (v1) at (0, -1) {};
\node[vertex,label=-90:$1$] (v6) at (1.5, -1) {};
\node[vertex,label={[yshift=2pt]-90:$-1$}] (v3) at (3, -1) {};
\node[vertex,label={[xshift=-3pt]180:$-3$},fill=red!80!white] (v5) at (4.5, -0.5) {};
\path[edge] (v2) -- (v1) {};
\path[sedge] (v2) -- (v6) {};
\path[edge] (v2) -- (v4) {};
\path[sedge] (v1) -- (v4) {};
\path[edge] (v6) -- (v1) {};
\path[sedge] (v4) -- (v3) {};
\path[sedge] (v6) -- (v0) {};
\path[ledge] (v3) -- (v0) {};
\path[edge] (v4) -- (v6) {};
\path[edge] (v6) -- (v3) {};
\path[edge] (v0) -- (v4) {};
\path[edge] (v0) to[bend left=10] (v5) {};
\path[edge] (v5) to[bend right=50] (v4) {};
\path[edge] (v5) to[bend right=55] (v2) {};
\path[edge] (v5) to[bend left=55] (v1) {};
\path[edge] (v5) to[bend left=50] (v6) {};
\path[edge] (v3) to[bend right=10] (v5) {};
\path[edge] (v0) to[bend right=40] (v2) {};
\path[edge] (v1) to[bend right=40] (v3) {};
\end{tikzpicture}
}
\caption{Examples of gadgets with various demands. The root vertex is coloured red. Labels show the vector $\x$ from Definition~\ref{def:theMagicVector}.}
\label{fig:bunchOfGadgets}
\end{figure}

\begin{lemma}[{\cite[Lemma 1.2]{BanjaLukaPaper}}]
\label{lem:rationalRootsDigraphs}
Let $G$ be a digraph. If $\lambda$ is a rational eigenvalue of $G$ then $\lambda$ is an integer.
\end{lemma}

\begin{proof}
The characteristic polynomial $\varphi$ of a digraph is a monic polynomial with integer coefficients. By the Rational Root Theorem, the rational roots of $\varphi$ are integers.
\end{proof}

Lemma~\ref{lem:rationalRootsDigraphs} will be significant for the formulation of the main theorem of the present section (Theorem~\ref{thm:theGrandMultiplierTheorem}), for which
we will require gadgets that have integer demand. The next theorem shows that ambi-nut digraphs provide a ready source of gadgets with demand $-1$.

\begin{theorem}
\label{thm:maxineyGadgets}
Let $H$ be an ambi-nut digraph and let $\widetilde{\x} \in \ker H$ be a non-trivial vector. Let $(u_1, u_2) \in E(H)$ such that $\widetilde{\x}(u_1) = \widetilde{\x}(u_2)$.
Let $G$ be a digraph with the vertex set $V(G) = V(H) \cup \{ u_0 \}$ and the arc-set $E(G) = E(H) \setminus \{ (u_1, u_2) \} \cup \{ (u_1, u_0), (u_0, u_2) \}$.
Then $(G, u_0)$ is a gadget with $\dem (G, u_0) = -1$.
\end{theorem}

\begin{proof}
Let us define $\x\colon V(G) \to \mathbb{R}$ by $\x(v) = \widetilde{\x}(v)$ for $v \in V(H)$ and $\x(u_0) = \widetilde{\x}(u_1)$. It is clear that $\x$ is full, $\x \in \ker A(G)_{(u_0)}$ and 
$\x \in \ker A(G^R)_{(u_0)}$. Suppose that $\y \in \ker A(G)_{(u_0)}$. Let us define $\widetilde{\y}\colon V(H) \to \mathbb{R}$ by $\widetilde{\y}(v) = \y(v)$ for $v \in V(H)$.
By Lemma~\ref{lem:ext}, $\widetilde{\y}(v) \in \ker H$. But $H$ is an ambi-nut digraph, so $\widetilde{\y} = \mu \widetilde{\x}$ for some $\mu \in \mathbb{R}$.
Now, $\mu\widetilde{\x}$ satisfies the out-local condition at vertex $u_1$ in the graph $H$, i.e.\
\begin{equation}
\label{eq:proofThm34_1} 
\Su_{\mu\widetilde{\x}}^+(u_1) = \sum_{v \in H^+(u_1)}\mu\widetilde{\x}(v) = \sum_{v \in H^+(u_1) \setminus \{u_2\}}\mu\widetilde{\x}(v) + \mu\widetilde{\x}(u_2) = 0.
\end{equation}
Similarly, $\y$ satisfies
the out-local condition at vertex $u_1$ in the graph $G$, i.e.\ 
\begin{equation}
\label{eq:proofThm34_2}
\Su_{\y}^+(u_1) = \sum_{v \in G^+(u_1)}\y(v) = \sum_{v \in G^+(u_1) \setminus \{u_0\}}\y(v) + \y(u_0) = 0.
\end{equation}
From \eqref{eq:proofThm34_1} and \eqref{eq:proofThm34_2} we obtain $\y(u_0) = \mu\widetilde{\x}(u_2)$. This implies that $\y = \mu \x$ and thus $\ker A(G)_{(u_0)}$ is $1$-dimensional.
By analogy we can show that $\ker A(G^R)_{(u_0)}$ is also $1$-dimensional, so $\x$ satisfies the requirements of Definition~\ref{def:theMagicVector}, and thus $(G, u_0)$ is a gadget.
As $u_0$ has out-degree one, connecting only to $u_2$, it is straightforward to see that $\dem (G, u_0) = -1$.
\end{proof}

Note that the condition in Theorem~\ref{thm:maxineyGadgets}, namely the existence of an arc $(u_1, u_2) \in E(H)$ such that $\widetilde{\x}(u_1) = \widetilde{\x}(u_2)$, is not obeyed by all ambi-nut digraphs. The smallest
examples among oriented graphs that are unsuitable for the theorem have $10$ vertices; there are exactly $6$ of them~\cite{GitHubSupplement}. All ambi-nut digraphs among oriented graphs up to order $9$ contain at least one `suitable' arc.

In light of the following lemma, gadgets with fractional demand can also be useful.

\begin{lemma}
\label{lem:compositeLemma}
Let $(G_1, r_1)$ and $(G_2, r_2)$ be gadgets. Then the digraph $(G_1, r_1) \odot (G_2, r_2)$ with the root being the merged vertex is a gadget with demand $\dem(G_1, r_1) + \dem(G_2, r_2)$.
\end{lemma} 

\begin{proof}
Let $H \coloneqq (G_1, r_1) \odot (G_2, r_2)$.
Let $\x^{(1)} \in \ker A(G_1)_{(r_1)} \cap \ker A(G_1^R)_{(r_1)}$ and $\x^{(2)} \in \ker A(G_2)_{(r_2)} \cap \ker A(G_2^R)_{(r_2)}$, which exist by Definition~\ref{def:theMagicVector}. 
We can assume that $\x^{(1)}(r_1) = \x^{(2)}(r_2) \neq 0$. Define $\y \colon V(H) \to \mathbb{R}$ by
$\y(u) = \x^{(1)}(u)$ if $u \in V(G_1)$ and $\y(u) = \x^{(2)}(u)$ if $u \in V(G_2)$. Vector $\y$ is full and $\y \in \ker A(H)_{(r_1)} \cap \ker A(H^R)_{(r_1)}$. 

Let $\z \in \ker A(H)_{(r_1)}$. The restriction of $\z$ to $G_1$ (resp.\ $G_2$) is in $\ker A(G_1)_{(r_1)}$ (resp.\ $\ker A(G_2)_{(r_2)}$), 
which is $1$-dimensional by Definition~\ref{def:theMagicVector},
hence the restriction of $\z$ to $G_1$ (resp.\ $G_2$) is a multiple of $\x^{(1)}$ (resp.\ $\x^{(2)}$).
This implies that $\z$ is a multiple of $\y$ and hence, $\ker A(H)_{(r_1)}$ is $1$-dimensional. By similar reasoning, we can show that  $\ker A(H^R)_{(r_1)}$ is also $1$-dimensional.

Moreover, 
\begin{equation}
\begin{aligned}
\dem(H, r_1) & = - \frac{\sum_{u \in H^+(r_1)} \y(u)}{\y(r_1)} \\
 & = - \frac{\sum_{u \in G_1^+(r_1)}\y(u)}{\y(r_1)} - \frac{\sum_{u \in G_2^+(r_2)} \y(u)}{\y(r_2)} = \dem(G_1, r_1) + \dem(G_2, r_2). 
\end{aligned} 
\end{equation}
\end{proof}

A gadget that can be obtained by (possibly repeated) use of Lemma~\ref{lem:compositeLemma} will be called a \emph{composite} gadget. Note that this allows
gadgets with integer demand to be obtained by coalescing gadgets of fractional demand; the gadget in Figure~\ref{fig:compositeGadgetAmbiNutExample}(a) is an example.
 Note that if we combine Lemma~\ref{lem:compositeLemma} with Proposition~\ref{prop:gadgetsAreAmbi}, we have yet another avenue for obtaining large ambi-nut
 digraphs. If we coalesce, at their common root, a set of gadgets for which the demands add to $0$, the result is an ambi-nut digraph; the ambi-nut digraph in
 Figure~\ref{fig:compositeGadgetAmbiNutExample}(b) is an example.
 
\begin{figure}[!htbp]
\centering
\subcaptionbox{}{
\begin{tikzpicture}[scale=1.1]
\tikzstyle{vertex}=[draw,circle,font=\scriptsize,minimum size=5pt,inner sep=1pt,fill=green!80!white]
\tikzstyle{uertex}=[draw,circle,font=\scriptsize,minimum size=5pt,inner sep=1pt,fill=cyan!80!white]
\tikzstyle{edge}=[draw,thick,decoration={markings,mark=at position 0.55 with {\arrow{Stealth}}},postaction=decorate]
\tikzstyle{sedge}=[draw,thick,decoration={markings,mark=at position 0.35 with {\arrow{Stealth}}},postaction=decorate]
\tikzstyle{ssedge}=[draw,thick,decoration={markings,mark=at position 0.25 with {\arrow{Stealth}}},postaction=decorate]
\tikzstyle{ledge}=[draw,thick,decoration={markings,mark=at position 0.85 with {\arrow{Stealth}}},postaction=decorate]
\tikzstyle{lledge}=[draw,thick,decoration={markings,mark=at position 0.90 with {\arrow{Stealth}}},postaction=decorate]
\node[vertex,label={90:$2$},fill=red!80!white] (v3) at (0, 0) {};
\node[vertex,label={0:$-1$}] (v2) at (1, 0) {};
\node[vertex,label={0:$2$}] (v0) at (2, 0) {};
\node[vertex,label={0:$-2$}] (v1) at (3, 0) {};
\node[vertex,label={90:$1$}] (v7) at (1.5, 1) {};
\node[vertex,label={[yshift=-0.8pt]90:$-1$}] (v5) at (2.5, 1) {};
\node[vertex,label={-90:$1$}] (v6) at (1.5, -1) {};
\node[vertex,label={-90:$-1$}] (v4) at (2.5, -1) {};
\node[uertex,label={90:$1$}] (u6) at (-1, 0.5) {};
\node[uertex,label={-90:$1$}] (u7) at (-1, -0.5) {};
\node[uertex,label={90:$1$}] (u1) at (-2, 1) {};
\node[uertex,label={[yshift=-1pt,xshift=-2pt]90:$-1$}] (u0) at (-2, -0.1) {};
\node[uertex,label={-90:$-1$}] (u2) at (-2, -1) {};
\node[uertex,label={180:$-1$}] (u4) at (-3, 0.5) {};
\node[uertex,label={180:$-1$}] (u5) at (-3, -0.5) {};
\path[edge] (v7) -- (v5) {};
\path[edge] (v4) -- (v6) {};
\path[ledge] (v4) -- (v2) {};
\path[edge] (v7) -- (v2) {};
\path[sedge] (v2) -- (v5) {};
\path[ledge] (v7) -- (v0) {};
\path[ledge] (v5) -- (v0) {};
\path[edge] (v5) -- (v1) {};
\path[ledge] (v6) -- (v1) {};
\path[sedge] (v1) -- (v7) {};
\path[sedge] (v0) -- (v4) {};
\path[edge] (v1) -- (v4) {};
\path[edge] (v2) -- (v6) {};
\path[sedge] (v0) -- (v6) {};
\path[edge] (v3) -- (v7) {};
\path[edge] (v6) -- (v3) {};
\path[edge] (u6) -- (u1) {};
\path[edge] (u1) -- (u4) {};
\path[edge] (u4) -- (u6) {};
\path[edge] (u7) -- (u5) {};
\path[edge] (u2) -- (u5) {};
\path[edge] (u2) -- (u7) {};
\path[sedge] (u0) -- (u4) {};
\path[sedge] (u0) -- (u6) {};
\path[ledge] (u5) -- (u0) {};
\path[ledge] (u7) -- (u0) {};
\path[ssedge] (u5) -- (u1) {};
\path[lledge] (u1) -- (u7) {};
\path[ssedge] (u6) -- (u2) {};
\path[ssedge] (u4) -- (u2) {};
\path[edge] (u7) -- (v3) {};
\path[edge] (v3) -- (u6) {};
\end{tikzpicture}
}
\subcaptionbox{}{
\begin{tikzpicture}[scale=1.1]
\tikzstyle{vertex}=[draw,circle,font=\scriptsize,minimum size=5pt,inner sep=1pt,fill=green!80!white]
\tikzstyle{uertex}=[draw,circle,font=\scriptsize,minimum size=5pt,inner sep=1pt,fill=orange!80!white]
\tikzstyle{edge}=[draw,thick,decoration={markings,mark=at position 0.55 with {\arrow{Stealth}}},postaction=decorate]
\tikzstyle{sedge}=[draw,thick,decoration={markings,mark=at position 0.35 with {\arrow{Stealth}}},postaction=decorate]
\tikzstyle{ssedge}=[draw,thick,decoration={markings,mark=at position 0.25 with {\arrow{Stealth}}},postaction=decorate]
\tikzstyle{ledge}=[draw,thick,decoration={markings,mark=at position 0.85 with {\arrow{Stealth}}},postaction=decorate]
\tikzstyle{lledge}=[draw,thick,decoration={markings,mark=at position 0.90 with {\arrow{Stealth}}},postaction=decorate]
\node[vertex,label={90:$2$},fill=red!80!white] (v3) at (0, 0) {};
\node[vertex,label={0:$-1$}] (v2) at (1, 0) {};
\node[vertex,label={0:$2$}] (v0) at (2, 0) {};
\node[vertex,label={0:$-2$}] (v1) at (3, 0) {};
\node[vertex,label={90:$1$}] (v7) at (1.5, 1) {};
\node[vertex,label={[yshift=-0.8pt]90:$-1$}] (v5) at (2.5, 1) {};
\node[vertex,label={-90:$1$}] (v6) at (1.5, -1) {};
\node[vertex,label={-90:$-1$}] (v4) at (2.5, -1) {};
\node[uertex,label={90:$-1$}] (u6) at (-1, 0.5) {};
\node[uertex,label={-90:$-1$}] (u7) at (-1, -0.5) {};
\node[uertex,label={[yshift=-2pt]90:$1$}] (u1) at (-2, 0) {};
\node[uertex,label={-90:$-1$}] (u0) at (-2, -1) {};
\node[uertex,label={90:$-1$}] (u2) at (-2, 1) {};
\node[uertex,label={180:$1$}] (u4) at (-3, 0.5) {};
\node[uertex,label={180:$1$}] (u5) at (-3, -0.5) {};
\path[edge] (v7) -- (v5) {};
\path[edge] (v4) -- (v6) {};
\path[ledge] (v4) -- (v2) {};
\path[edge] (v7) -- (v2) {};
\path[sedge] (v2) -- (v5) {};
\path[ledge] (v7) -- (v0) {};
\path[ledge] (v5) -- (v0) {};
\path[edge] (v5) -- (v1) {};
\path[ledge] (v6) -- (v1) {};
\path[sedge] (v1) -- (v7) {};
\path[sedge] (v0) -- (v4) {};
\path[edge] (v1) -- (v4) {};
\path[edge] (v2) -- (v6) {};
\path[sedge] (v0) -- (v6) {};
\path[edge] (v3) -- (v7) {};
\path[edge] (v6) -- (v3) {};
\path[edge] (u6) -- (u7) {};
\path[edge] (u4) -- (u5) {};
\path[edge] (u4) -- (u2) {};
\path[ledge] (u7) -- (u2) {};
\path[sedge] (u2) -- (u5) {};
\path[edge] (u2) -- (u6) {};
\path[sedge] (u1) -- (u7) {};
\path[sedge] (u1) -- (u4) {};
\path[ledge] (u6) -- (u1) {};
\path[ledge] (u5) -- (u1) {};
\path[edge] (u7) -- (u0) {};
\path[edge] (u5) -- (u0) {};
\path[sedge] (u0) -- (u4) {};
\path[sedge] (u0) -- (u6) {};
\path[edge] (u7) -- (v3) {};
\path[edge] (v3) -- (u6) {};
\end{tikzpicture}
}
\caption{(a): A composite gadget of demand $-1$ obtained from two non-isomorphic gadgets with demands $-1/2$.  (b): A composite gadget of demand 0, which is an ambi-nut digraph, obtained from two non-isomorphic gadgets with respective demands $1/2$ and $-1/2$.}
\label{fig:compositeGadgetAmbiNutExample}
\end{figure}
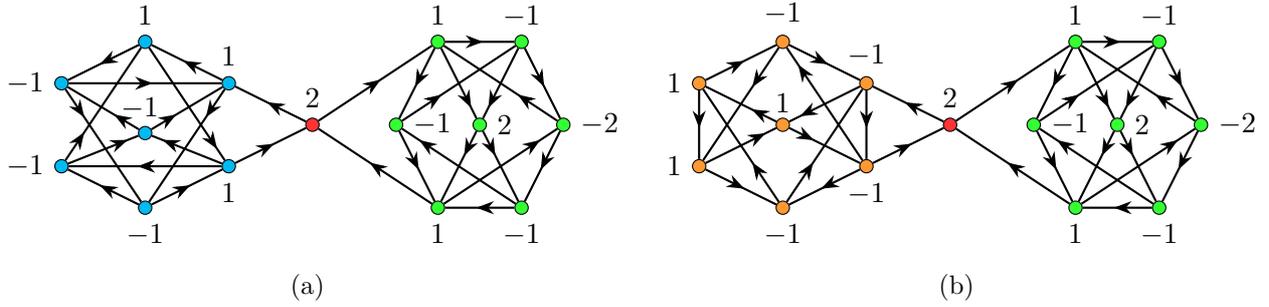
 
We now give a theorem that can be used to produce ambi-nut digraphs by attaching gadgets to a \emph{base digraph}.
 
\begin{theorem}
\label{thm:theGrandMultiplierTheorem}
Let $G$ be the base digraph, i.e.\ a digraph of order $n$ with $V(G) = \{v_1, v_2, \ldots, v_n\}$ and an integer eigenvalue $\lambda$ such that the $\lambda$-eigenspaces of $G$ and
 $G^R$ are both $1$-dimensional and spanned by the same full vector~$\x$.
Let $(\Gamma_1, r_1), (\Gamma_2, r_2), \ldots, (\Gamma_n, r_n)$ be gadgets, such that $\dem (\Gamma_i, r_i) = \lambda$ for all $i = 1, \ldots, n$.
Then the digraph $M$ obtained from the disjoint union $G \sqcup \Gamma_1 \sqcup \Gamma_2 \sqcup \cdots \sqcup \Gamma_n$ by identifying $v_i$ with $r_i$ for all $i = 1,\ldots, n$
 is ambi-nut.
\end{theorem}

\begin{proof}
For each $i$, let $\x^{(i)}$ be the vector that spans the kernel of $A(\Gamma_i)_{(r_i)}$. 
We can assume that $\x^{(i)}(r_i) = \x(v_i)$. Let $\y \colon V(M) \to \mathbb{R}$, such that for each $i$ and for each $v \in V(\Gamma_i)$,  $\y(v) = \x^{(i)}(v)$.
From Definition~\ref{def:theMagicVector} it follows that every vertex $v \in V(M) \setminus V(G)$ satisfies $\Su_{\y}^{+}(v) = \Su_{\y}^{-}(v) = 0$. Moreover,
$\Su_{\y}^{+}(v_i) = \sum_{u \in M^+(v_i)} \y(u) = \sum_{u \in G^+(v_i)} \y(u) +  \sum_{u \in \Gamma_i^+(v_i)} \y(u) = \lambda \x(v_i) - \dem(\Gamma_i, r_i) \x(v_i) = 0$.
Similarly, $\Su_{\y}^{-}(v_i) = 0$. It follows that $\y \in \ker M$ and $\y \in \CoKer M$. Since $\x$ and all $\x^{(i)}$ are full, $\y$ is also full.

Let $\z \in \ker M$. The vector $\z$ restricted to $\Gamma_i$ is in $\ker A(\Gamma_i)_{(r_i)}$ and therefore it is a multiple of $\x^{(i)}$. The
local condition at vertex $r_i$ is
\begin{equation}
\label{eq:localGadgetBaseGraphBurek}
\sum_{u \in M^+(r_i)} \z(u) = \sum_{u \in G^+(r_i)} \z(u) + \sum_{u \in \Gamma_i^+(r_i)} \z(u) = 0.
\end{equation}
Recall that $\sum_{u \in \Gamma_i^+(r_i)} \z(u) = -\z(r_i) \dem(\Gamma_i, r_i)$. Equation \eqref{eq:localGadgetBaseGraphBurek} yields
\begin{equation}
\sum_{u \in G^+(r_i)} \z(u) -\z(r_i) \dem(\Gamma_i, r_i) = 0.
\end{equation}
Since $\dem(\Gamma_i, r_i) = \lambda$, we further obtain
\begin{equation}
\sum_{u \in G^+(r_i)} \z(u)  = \lambda \z(r_i) .
\end{equation}
This implies that $\z$ restricted to $G$ is an eigenvector for $\lambda$, so it is a multiple of $\x$. In summary, $\z$ is a multiple of $\y$. We conclude that
$\ker M$ is $1$-dimensional. By similar reasoning, we can see that $\CoKer M$ is also $1$-dimensional.
\end{proof}

Using Theorem~\ref{thm:maxineyGadgets}, Lemma~\ref{lem:compositeLemma} and the examples in Figure~\ref{fig:bunchOfGadgets}, 
it is easy to produce 
an abundance of gadgets with integer demands. 
Following the definition used in \cite{Tuite2019},
a digraph $G$ is \emph{diregular of degree $k$} (or \emph{$k$-diregular}) if $d^+(v) = d^-(v) = k$ for $v \in V(G)$.
(Some authors, e.g.~\cite{Brualdi2010}, use the term $k$-regular digraph to mean what we are calling a $k$-diregular digraph.
However, in the present paper, we have already used  \emph{$k$-regular digraph} to denote a digraph whose underlying
graph is $k$-regular.)
Further examples of gadgets can be found in \cite{GitHubSupplement}. 
For the base digraph  
$G$ in Theorem~\ref{thm:theGrandMultiplierTheorem}, there exist whole classes of examples:
\begin{enumerate}[label=(\roman*)]
\item
Every connected $k$-diregular digraph is a base digraph for $\lambda = k$.
\item
Every connected bipartite $k$-diregular digraph is a base digraph for $\lambda = -k$ (in addition to $\lambda = k$).
\item
Figure~\ref{ref:baseGraphsDifferentEvals} shows examples of base digraphs for $\lambda \in \{-1, 1, 2\}$ that are not diregular.
More can be found in~\cite{GitHubSupplement}.
\end{enumerate}

\begin{figure}[!htbp]
\centering
\subcaptionbox{$\lambda = -1$}{
\begin{tikzpicture}[scale=1.5]
\tikzstyle{vertex}=[draw,circle,font=\scriptsize,minimum size=5pt,inner sep=1pt,fill=green!80!white]
\tikzstyle{edge}=[draw,thick,decoration={markings,mark=at position 0.6 with {\arrow{Stealth}}},postaction=decorate]
\tikzstyle{ledge}=[draw,thick,decoration={markings,mark=at position 0.8 with {\arrow{Stealth}}},postaction=decorate]
\node[vertex,label=90:$1$] (v0) at (1, 1) {};
\node[vertex,label=-90:$1$] (v2) at (1, 0) {};
\node[vertex,label=90:$-1$] (v5) at (0, 1) {};
\node[vertex,label=-90:$-1$] (v4) at (0, 0) {};
\node[vertex,label=90:$1$] (v1) at ($ (0, 1) + (-150:1) $) {};
\node[vertex,label=90:$-1$] (v3) at ($ (1, 0) + (30:1) $) {};
\path[edge] (v0) -- (v3) {};
\path[edge] (v0) -- (v2) {};
\path[edge] (v3) -- (v2) {};
\path[edge] (v5) -- (v4) {};
\path[edge] (v5) -- (v1) {};
\path[edge] (v1) -- (v4) {};
\path[ledge] (v0) -- (v4) {};
\path[ledge] (v2) -- (v5) {};
\path[edge] (v5) -- (v0) {};
\path[edge] (v4) -- (v2) {};
\end{tikzpicture}
}
\subcaptionbox{$\lambda = -1$}{
\begin{tikzpicture}[scale=1.5]
\tikzstyle{vertex}=[draw,circle,font=\scriptsize,minimum size=5pt,inner sep=1pt,fill=green!80!white]
\tikzstyle{edge}=[draw,thick,decoration={markings,mark=at position 0.6 with {\arrow{Stealth}}},postaction=decorate]
\tikzstyle{ledge}=[draw,thick,decoration={markings,mark=at position 0.8 with {\arrow{Stealth}}},postaction=decorate]
\node[vertex,label=90:$-1$] (v5) at (36:1) {};
\node[vertex,label=180:$1$] (v0) at ({36 + 1 * 72}:1) {};
\node[vertex,label={[xshift=-4pt]90:$-1$}] (v4) at ({36 + 2 * 72}:1) {};
\node[vertex,label=180:$1$] (v1) at ({36 + 3 * 72}:1) {};
\node[vertex,label=-90:$-1$] (v6) at ({36 + 4 * 72}:1) {};
\node[vertex,label=90:$1$] (v3) at ($ (v5) + (-30:1) $) {};
\node[vertex,label=90:$1$] (v2) at ($ (v5) + (210:1) $) {};
\path[edge] (v5) -- (v0) {};
\path[edge] (v0) -- (v4) {};
\path[edge] (v4) -- (v1) {};
\path[edge] (v1) -- (v6) {};
\path[edge] (v6) -- (v5) {};
\path[edge] (v3) -- (v5) {};
\path[edge] (v6) -- (v3) {};
\path[edge] (v6) -- (v2) {};
\path[edge] (v2) -- (v5) {};
\end{tikzpicture}
}
\subcaptionbox{$\lambda = -1$}{
\begin{tikzpicture}[scale=0.8]
\tikzstyle{vertex}=[draw,circle,font=\scriptsize,minimum size=5pt,inner sep=1pt,fill=green!80!white]
\tikzstyle{edge}=[draw,thick,decoration={markings,mark=at position 0.6 with {\arrow{Stealth}}},postaction=decorate]
\tikzstyle{ledge}=[draw,thick,decoration={markings,mark=at position 0.8 with {\arrow{Stealth}}},postaction=decorate]
\node[vertex,label=180:$2$] (v0) at (-5, 4.8) {};
\node[vertex,label=0:$-1$] (v3) at (-0.5, 4.8) {};
\node[vertex,label=180:$-2$] (v4) at (-5, 0.6) {};
\node[vertex,label=0:$1$] (v5) at (-0.5, 0.6) {};
\node[vertex,label={[xshift=-2pt,yshift=4pt]0:$1$}] (v1) at ({-2.75-0.7}, 2.0) {};
\node[vertex,label={[xshift=2pt,yshift=4pt]180:$1$}] (v2) at ({-2.75+0.7}, 2.0) {};
\node[vertex,label={[yshift=-3pt]90:$-1$}] (v6) at (-2.75, 4.0) {};
\path[edge] (v0) -- (v3) {};
\path[edge] (v0) -- (v6) {};
\path[edge] (v1) -- (v4) {};
\path[edge] (v1) -- (v5) {};
\path[edge] (v2) -- (v4) {};
\path[edge] (v2) -- (v5) {};
\path[edge] (v3) -- (v5) {};
\path[edge] (v4) -- (v0) {};
\path[edge] (v4) -- (v5) {};
\path[edge] (v4) -- (v6) {};
\path[edge] (v5) -- (v6) {};
\path[edge] (v6) -- (v1) {};
\path[edge] (v6) -- (v2) {};
\path[edge] (v6) -- (v3) {};
\end{tikzpicture}
}

\subcaptionbox{$\lambda = -1$}{
\begin{tikzpicture}[scale=1.5]
\tikzstyle{vertex}=[draw,circle,font=\scriptsize,minimum size=5pt,inner sep=1pt,fill=green!80!white]
\tikzstyle{edge}=[draw,thick,decoration={markings,mark=at position 0.6 with {\arrow{Stealth}}},postaction=decorate]
\tikzstyle{ledge}=[draw,thick,decoration={markings,mark=at position 0.8 with {\arrow{Stealth}}},postaction=decorate]
\node[vertex,label=90:$1$] (v0) at (1, 1) {};
\node[vertex,label=-90:$1$] (v2) at (1, 0) {};
\node[vertex,label=90:$-1$] (v5) at (0, 1) {};
\node[vertex,label=-90:$-1$] (v4) at (0, 0) {};
\node[vertex,label=90:$1$] (v1) at ($ (0, 1) + (-150:1) $) {};
\node[vertex,label=90:$-1$] (v3) at ($ (1, 0) + (30:1) $) {};
\path[edge] (v0) -- (v3) {};
\path[edge] (v0) -- (v2) {};
\path[edge] (v3) -- (v2) {};
\path[edge] (v5) -- (v4) {};
\path[edge] (v1) -- (v5) {};
\path[edge] (v4) -- (v1) {};
\path[ledge] (v0) -- (v4) {};
\path[ledge] (v5) -- (v2) {};
\path[edge] (v5) -- (v0) {};
\path[edge] (v2) -- (v4) {};
\end{tikzpicture}
}
\subcaptionbox{$\lambda = -1$}{
\begin{tikzpicture}[scale=1.5]
\tikzstyle{vertex}=[draw,circle,font=\scriptsize,minimum size=5pt,inner sep=1pt,fill=green!80!white]
\tikzstyle{edge}=[draw,thick,decoration={markings,mark=at position 0.6 with {\arrow{Stealth}}},postaction=decorate]
\tikzstyle{ledge}=[draw,thick,decoration={markings,mark=at position 0.8 with {\arrow{Stealth}}},postaction=decorate]
\node[vertex,label=90:$-1$] (v5) at (36:1) {};
\node[vertex,label=180:$1$] (v0) at ({36 + 1 * 72}:1) {};
\node[vertex,label={[xshift=-4pt]90:$-1$}] (v4) at ({36 + 2 * 72}:1) {};
\node[vertex,label=180:$1$] (v1) at ({36 + 3 * 72}:1) {};
\node[vertex,label=-90:$-1$] (v6) at ({36 + 4 * 72}:1) {};
\node[vertex,label=90:$1$] (v3) at ($ (v5) + (-30:1) $) {};
\node[vertex,label=90:$1$] (v2) at ($ (v5) + (210:1) $) {};
\path[edge] (v5) -- (v0) {};
\path[edge] (v0) -- (v4) {};
\path[edge] (v4) -- (v1) {};
\path[edge] (v1) -- (v6) {};
\path[edge] (v5) -- (v6) {};
\path[edge] (v5) -- (v3) {};
\path[edge] (v3) -- (v6) {};
\path[edge] (v6) -- (v2) {};
\path[edge] (v2) -- (v5) {};
\end{tikzpicture}
}
\subcaptionbox{$\lambda = -1$}{
\begin{tikzpicture}[scale=0.8]
\tikzstyle{vertex}=[draw,circle,font=\scriptsize,minimum size=5pt,inner sep=1pt,fill=green!80!white]
\tikzstyle{edge}=[draw,thick,decoration={markings,mark=at position 0.6 with {\arrow{Stealth}}},postaction=decorate]
\tikzstyle{ledge}=[draw,thick,decoration={markings,mark=at position 0.8 with {\arrow{Stealth}}},postaction=decorate]
\node[vertex,label=180:$-1$] (v0) at (-5, 4.8) {};
\node[vertex,label=0:$2$] (v3) at (-0.5, 4.8) {};
\node[vertex,label=180:$1$] (v4) at (-5, 0.6) {};
\node[vertex,label=0:$-2$] (v5) at (-0.5, 0.6) {};
\node[vertex,label={[xshift=-2pt,yshift=4pt]0:$1$}] (v1) at ({-2.75-0.7}, 2.0) {};
\node[vertex,label={[xshift=2pt,yshift=4pt]180:$1$}] (v2) at ({-2.75+0.7}, 2.0) {};
\node[vertex,label={[yshift=-3pt]90:$-1$}] (v6) at (-2.75, 4.0) {};
\path[edge] (v0) -- (v3) {};
\path[edge] (v0) -- (v6) {};
\path[edge] (v1) -- (v6) {};
\path[edge] (v2) -- (v6) {};
\path[edge] (v3) -- (v5) {};
\path[edge] (v4) -- (v0) {};
\path[edge] (v4) -- (v5) {};
\path[edge] (v4) -- (v1) {};
\path[edge] (v4) -- (v2) {};
\path[edge] (v5) -- (v1) {};
\path[edge] (v5) -- (v2) {};
\path[edge] (v6) -- (v4) {};
\path[edge] (v6) -- (v5) {};
\path[edge] (v6) -- (v3) {};
\end{tikzpicture}
}

\subcaptionbox{$\lambda = 1$}{
\begin{tikzpicture}[scale=1.2]
\tikzstyle{vertex}=[draw,circle,font=\scriptsize,minimum size=5pt,inner sep=1pt,fill=green!80!white]
\tikzstyle{edge}=[draw,thick,decoration={markings,mark=at position 0.55 with {\arrow{Stealth}}},postaction=decorate]
\node[vertex,label={[yshift=-2pt]90:$1$}] (v6) at (0, -0.1) {};
\node[vertex,label={[yshift=2pt]-90:$1$}] (v4) at (0, -0.9) {};
\node[vertex,label=180:$1$] (v0) at (-1, 0.5) {};
\node[vertex,label=180:$1$] (v2) at (-1, -1.5) {};
\node[vertex,label={[xshift=-4pt,yshift=-2pt]90:$-1$}] (v3) at (1, -0.1) {};
\node[vertex,label={[xshift=-4pt,yshift=2pt]-90:$-1$}] (v1) at (1, -0.9) {};
\node[vertex,label=0:$-1$] (v5) at (2, 0.5) {};
\node[vertex,label=0:$-1$] (v7) at (2, -1.5) {};
\path[edge] (v0) -- (v4) {};
\path[edge] (v2) -- (v6) {};
\path[edge] (v0) -- (v2) {};
\path[edge] (v2) -- (v4) {};
\path[edge] (v4) -- (v6) {};
\path[edge] (v6) -- (v0) {};
\path[edge] (v3) -- (v7) {};
\path[edge] (v1) -- (v5) {};
\path[edge] (v3) -- (v5) {};
\path[edge] (v5) -- (v7) {};
\path[edge] (v7) -- (v1) {};
\path[edge] (v1) -- (v3) {};
\path[edge] (v3) -- (v6) {};
\path[edge] (v0) -- (v5) {};
\path[edge] (v2) -- (v7) {};
\path[edge] (v1) -- (v4) {};
\end{tikzpicture}
}
\subcaptionbox{$\lambda = 1$}{
\begin{tikzpicture}[scale=1.2]
\tikzstyle{vertex}=[draw,circle,font=\scriptsize,minimum size=5pt,inner sep=1pt,fill=green!80!white]
\tikzstyle{edge}=[draw,thick,decoration={markings,mark=at position 0.55 with {\arrow{Stealth}}},postaction=decorate]
\tikzstyle{ledge}=[draw,thick,decoration={markings,mark=at position 0.80 with {\arrow{Stealth}}},postaction=decorate]
\node[vertex,label=90:$1$] (v1) at (0, 0) {};
\node[vertex,label=-90:$1$] (v3) at (0, -1) {};
\node[vertex,label=90:$1$] (v7) at (1, 0) {};
\node[vertex,label=-90:$1$] (v5) at (1, -1) {};
\node[vertex,label={[xshift=-2pt]-90:$-1$}] (v6) at (-1, 0) {};
\node[vertex,label=-90:$-1$] (v0) at (2, 0) {};
\node[vertex,label=0:$-2$] (v4) at (0.5, 1.6) {};
\node[vertex,label=-90:$-2$] (v2) at (0.5, 1.0) {};
\path[edge] (v4) -- (v6) {};
\path[edge] (v6) -- (v2) {};
\path[edge] (v4) -- (v0) {};
\path[edge] (v0) -- (v2) {};
\path[ledge] (v2) -- (v4) {};
\path[edge] (v1) -- (v6) {};
\path[ledge] (v6) -- (v3) {};
\path[edge] (v7) -- (v0) {};
\path[ledge] (v0) -- (v5) {};
\path[edge] (v1) -- (v3) {};
\path[edge] (v3) -- (v5) {};
\path[edge] (v5) -- (v7) {};
\path[edge] (v7) -- (v1) {};
\path[ledge] (v7) -- (v3) {};
\path[ledge] (v1) -- (v5) {};
\end{tikzpicture}
}
\subcaptionbox{$\lambda = 1$}{
\begin{tikzpicture}[scale=1.2]
\tikzstyle{vertex}=[draw,circle,font=\scriptsize,minimum size=5pt,inner sep=1pt,fill=green!80!white]
\tikzstyle{edge}=[draw,thick,decoration={markings,mark=at position 0.55 with {\arrow{Stealth}}},postaction=decorate]
\tikzstyle{ledge}=[draw,thick,decoration={markings,mark=at position 0.8 with {\arrow{Stealth}}},postaction=decorate]
\node[vertex,label=90:$-2$] (v5) at (0, 0) {};
\node[vertex,label=-90:$-2$] (v4) at (0, -1) {};
\node[vertex,label=90:$-1$] (v1) at (1, 0) {};
\node[vertex,label=-90:$-1$] (v2) at (1, -1) {};
\node[vertex,label=90:$1$] (v6) at (2, 0) {};
\node[vertex,label=-90:$1$] (v7) at (2, -1) {};
\node[vertex,label=90:$2$] (v3) at (3, 0) {};
\node[vertex,label=-90:$2$] (v0) at (3, -1) {};
\path[edge] (v4) -- (v5) {};
\path[edge] (v5) -- (v1) {};
\path[edge] (v1) -- (v6) {};
\path[edge] (v3) -- (v6) {};
\path[edge] (v0) -- (v3) {};
\path[edge] (v7) -- (v0) {};
\path[edge] (v2) -- (v7) {};
\path[edge] (v2) -- (v4) {};
\path[ledge] (v5) -- (v2) {};
\path[ledge] (v1) -- (v4) {};
\path[ledge] (v6) -- (v2) {};
\path[ledge] (v7) -- (v1) {};
\path[ledge] (v6) -- (v0) {};
\path[ledge] (v3) -- (v7) {};
\end{tikzpicture}
}

\subcaptionbox{$\lambda = 1$}{
\begin{tikzpicture}[scale=1.2]
\tikzstyle{vertex}=[draw,circle,font=\scriptsize,minimum size=5pt,inner sep=1pt,fill=green!80!white]
\tikzstyle{edge}=[draw,thick,decoration={markings,mark=at position 0.55 with {\arrow{Stealth}}},postaction=decorate]
\tikzstyle{ledge}=[draw,thick,decoration={markings,mark=at position 0.8 with {\arrow{Stealth}}},postaction=decorate]
\node[vertex,label=90:$2$] (v5) at (0, 0) {};
\node[vertex,label=-90:$2$] (v4) at (0, -1) {};
\node[vertex,label=180:$2$] (v2) at ({-sqrt(3)/2},-0.5) {};
\node[vertex,label=90:$-1$] (v7) at (2, 0) {};
\node[vertex,label=-90:$-1$] (v6) at (2, -1) {};
\node[vertex,label=90:$-2$] (v0) at (3, 0) {};
\node[vertex,label=-90:$-2$] (v3) at (3, -1) {};
\node[vertex,label=90:$1$] (v1) at (1, -0.6) {};
\path[edge] (v4) -- (v5) {};
\path[edge] (v2) -- (v4) {};
\path[edge] (v5) -- (v2) {};
\path[edge] (v5) -- (v7) {};
\path[edge] (v6) -- (v4) {};
\path[edge] (v6) -- (v7) {};
\path[edge] (v7) -- (v0) {};
\path[edge] (v3) -- (v6) {};
\path[edge] (v0) -- (v3) {};
\path[ledge] (v6) -- (v0) {};
\path[ledge] (v3) -- (v7) {};
\path[edge] (v5) -- (v1) {};
\path[edge] (v1) -- (v6) {};
\path[edge] (v7) -- (v1) {};
\path[edge] (v1) -- (v4) {};
\end{tikzpicture}
}
\subcaptionbox{$\lambda = 1$}{
\begin{tikzpicture}[scale=1.2]
\tikzstyle{vertex}=[draw,circle,font=\scriptsize,minimum size=5pt,inner sep=1pt,fill=green!80!white]
\tikzstyle{edge}=[draw,thick,decoration={markings,mark=at position 0.55 with {\arrow{Stealth}}},postaction=decorate]
\node[vertex,label=90:$-1$] (v7) at (0, 0) {};
\node[vertex,label=-90:$-1$] (v5) at (0, -1) {};
\node[vertex,label=90:$1$] (v4) at (3, 0) {};
\node[vertex,label=-90:$1$] (v6) at (3, -1) {};
\node[vertex,label=0:$-1$] (v3) at ({sqrt(3)/2},-0.5) {};
\node[vertex,label=180:$1$] (v1) at ({3 - sqrt(3)/2},-0.5) {};
\node[vertex,label=180:$-1$] (v2) at ({-sqrt(3)/2},-0.5) {};
\node[vertex,label=0:$1$] (v0) at ({3 + sqrt(3)/2},-0.5) {};
\path[edge] (v7) -- (v4) {};
\path[edge] (v4) -- (v6) {};
\path[edge] (v6) -- (v5) {};
\path[edge] (v5) -- (v7) {};
\path[edge] (v7) -- (v3) {};
\path[edge] (v3) -- (v5) {};
\path[edge] (v7) -- (v2) {};
\path[edge] (v2) -- (v5) {};
\path[edge] (v6) -- (v1) {};
\path[edge] (v1) -- (v4) {};
\path[edge] (v6) -- (v0) {};
\path[edge] (v0) -- (v4) {};
\end{tikzpicture}
}

\subcaptionbox{$\lambda = 2$}{
\begin{tikzpicture}[scale=1.5]
\tikzstyle{vertex}=[draw,circle,font=\scriptsize,minimum size=5pt,inner sep=1pt,fill=green!80!white]
\tikzstyle{edge}=[draw,thick,decoration={markings,mark=at position 0.55 with {\arrow{Stealth}}},postaction=decorate]
\tikzstyle{ledge}=[draw,thick,decoration={markings,mark=at position 0.8 with {\arrow{Stealth}}},postaction=decorate]
\node[vertex,label=90:$2$] (v5) at (36:1) {};
\node[vertex,label=180:$2$] (v0) at ({36 + 1 * 72}:1) {};
\node[vertex,label={[xshift=-4pt]90:$2$}] (v4) at ({36 + 2 * 72}:1) {};
\node[vertex,label=180:$2$] (v3) at ({36 + 3 * 72}:1) {};
\node[vertex,label=-90:$2$] (v6) at ({36 + 4 * 72}:1) {};
\node[vertex,label=90:$1$] (v1) at ($ (v5) + (0:1) $) {};
\node[vertex,label=90:$1$] (v2) at ($ (v6) + (0:1) $) {};
\path[edge] (v0) -- (v5) {};
\path[edge] (v4) -- (v0) {};
\path[edge] (v3) -- (v4) {};
\path[edge] (v3) -- (v6) {};
\path[edge] (v5) -- (v4) {};
\path[edge] (v5) -- (v3) {};
\path[edge] (v0) -- (v3) {};
\path[edge] (v4) -- (v6) {};
\path[edge] (v6) -- (v0) {};
\path[edge] (v6) -- (v2) {};
\path[ledge] (v6) -- (v1) {};
\path[ledge] (v2) -- (v5) {};
\path[edge] (v1) -- (v5) {};
\end{tikzpicture}
}
\subcaptionbox{$\lambda = 2$}{
\begin{tikzpicture}[scale=1.5]
\tikzstyle{vertex}=[draw,circle,font=\scriptsize,minimum size=5pt,inner sep=1pt,fill=green!80!white]
\tikzstyle{edge}=[draw,thick,decoration={markings,mark=at position 0.55 with {\arrow{Stealth}}},postaction=decorate]
\tikzstyle{ledge}=[draw,thick,decoration={markings,mark=at position 0.8 with {\arrow{Stealth}}},postaction=decorate]
\node[vertex,label=90:$2$] (v5) at (36:1) {};
\node[vertex,label=180:$2$] (v0) at ({36 + 1 * 72}:1) {};
\node[vertex,label={[xshift=-4pt]90:$2$}] (v4) at ({36 + 2 * 72}:1) {};
\node[vertex,label=180:$2$] (v3) at ({36 + 3 * 72}:1) {};
\node[vertex,label=-90:$2$] (v6) at ({36 + 4 * 72}:1) {};
\node[vertex,label=90:$1$] (v1) at ($ (v5) + (0:1) $) {};
\node[vertex,label=90:$1$] (v2) at ($ (v6) + (0:1) $) {};
\path[edge] (v5) -- (v0) {};
\path[edge] (v0) -- (v4) {};
\path[edge] (v3) -- (v4) {};
\path[edge] (v3) -- (v6) {};
\path[edge] (v4) -- (v5) {};
\path[edge] (v5) -- (v3) {};
\path[edge] (v0) -- (v3) {};
\path[edge] (v4) -- (v6) {};
\path[edge] (v6) -- (v0) {};
\path[edge] (v6) -- (v2) {};
\path[ledge] (v6) -- (v1) {};
\path[ledge] (v2) -- (v5) {};
\path[edge] (v1) -- (v5) {};
\end{tikzpicture}
}
\subcaptionbox{$\lambda = 2$}{
\begin{tikzpicture}[scale=1.2]
\tikzstyle{vertex}=[draw,circle,font=\scriptsize,minimum size=5pt,inner sep=1pt,fill=green!80!white]
\tikzstyle{edge}=[draw,thick,decoration={markings,mark=at position 0.55 with {\arrow{Stealth}}},postaction=decorate]
\tikzstyle{sedge}=[draw,thick,decoration={markings,mark=at position 0.35 with {\arrow{Stealth}}},postaction=decorate]
\tikzstyle{ledge}=[draw,thick,decoration={markings,mark=at position 0.80 with {\arrow{Stealth}}},postaction=decorate]
\node[vertex,label={[yshift=-4pt]120:$2$}] (v6) at (150:0.5) {};
\node[vertex,label={[yshift=4pt]-60:$2$}] (v7) at (150+180:0.5) {};
\node[vertex,label={90:$1$}] (v0) at (0:1.5) {};
\node[vertex,label={0:$1$}] (v1) at (60:1.5) {};
\node[vertex,label={180:$1$}] (v2) at (120:1.5) {};
\node[vertex,label={90:$1$}] (v3) at (180:1.5) {};
\node[vertex,label={180:$1$}] (v4) at (240:1.5) {};
\node[vertex,label={0:$1$}] (v5) at (300:1.5) {};
\path[ledge] (v6) -- (v5) {};
\path[edge] (v5) -- (v7) {};
\path[edge] (v6) -- (v4) {};
\path[edge] (v4) -- (v7) {};
\path[edge] (v7) -- (v0) {};
\path[sedge] (v0) -- (v6) {};
\path[ledge] (v7) -- (v3) {};
\path[edge] (v3) -- (v6) {};
\path[ledge] (v7) -- (v2) {};
\path[edge] (v2) -- (v6) {};
\path[edge] (v7) -- (v1) {};
\path[edge] (v1) -- (v6) {};
\path[edge] (v6) -- (v7) {};
\end{tikzpicture}
}
\caption{Non-diregular base digraphs for Theorem~\ref{thm:theGrandMultiplierTheorem} with $\lambda \in \{-1, 1, 2\}$. The labels show entries of the unique (up to scalar multiplication)
eigenvector for the chosen eigenvalue $\lambda$.}
\label{ref:baseGraphsDifferentEvals}
\end{figure}
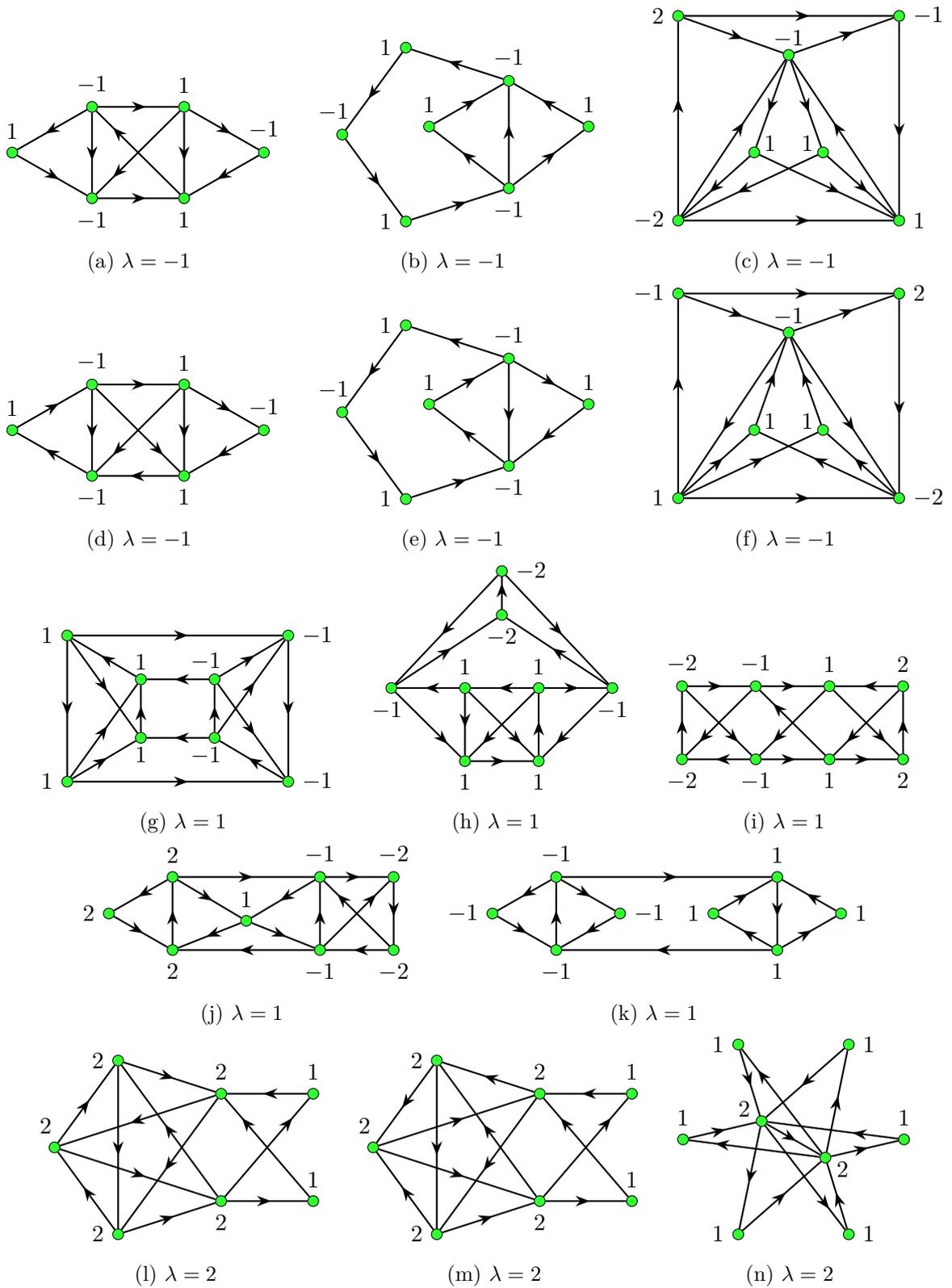

It is not difficult to see the following.

\begin{proposition}
\label{prop:everyDiregularStrongly}
Every connected $k$-diregular digraph ($k \geq 0$) is strongly connected.
\end{proposition}

\begin{proof}
Let $G$ be a connected $k$-diregular digraph and let $\widetilde{G}$ be the condensation of $G$.
Recall that the vertices of $\widetilde{G}$  are the strongly connected components $S_i$ of $G$ with $(S_i, S_j) \in E(\widetilde{G})$ if there exist vertices $u \in S_i$ and $v \in S_j$ such that $(u, v) \in E(G)$~\cite{Harary1971}. If $\widetilde{G}$ contains
exactly one vertex, there is nothing to prove, as the whole graph $G$ is a strongly connected component.
Suppose that $\widetilde{G}$ contains at least two vertices. As $\widetilde{G}$ is a DAG (directed acyclic graph),
there exists a vertex $S^* \in V(\widetilde{G})$ with $d^-(S^*) = 0$, which implies $$\sum_{\substack{v \to u \\ u \in S^*, v \notin S^*}} 1 = 0.$$ Since $G$ is connected, $d^+(S^*) > 0$, which implies
$$ \sum_{\substack{u \to v \\ u \in S^*, v \notin S^*}} 1 > 0.$$
For $U \subseteq V(G)$, let us define
\begin{equation}
\Delta(U) = \sum_{u \in U} d^+(u) - \sum_{u \in U} d^-(u).
\end{equation}
Since $G$ is diregular, we obtain $\Delta(S^*) = 0$.
On the other hand, we can write
\begin{equation}
\label{eq:deltaFunctionProp}
\Delta(U) = \sum_{\substack{u \to v \\ u \in U}} 1 \;\;-\;\; \sum_{\substack{v \to u \\ u \in U}} 1 = 
\sum_{\substack{u \to v \\ u, v \in U}}  \hspace{-0.3em} 1\;\; + \sum_{\substack{u \to v \\ u \in U, v \notin U}} \hspace{-0.8em} 1 \;\; - \sum_{\substack{v \to u \\ u, v \in U}} \hspace{-0.3em} 1 \;\; -  \sum_{\substack{v \to u \\ u \in U, v \notin U}} \hspace{-0.8em} 1 = 
{\sum_{\substack{u \to v \\ u \in U, v \notin U}} \hspace{-0.8em} 1}\;\;  -  {\sum_{\substack{v \to u \\ u \in U, v \notin U}} \hspace{-0.8em}  1}.
\end{equation}
Choosing $U = S^*$ in \eqref{eq:deltaFunctionProp} gives that $\Delta(S^*) > 0$, a contradiction. Therefore, $\widetilde{G}$ cannot
have two or more vertices.
\end{proof}

By Proposition~\ref{prop:everyDiregularStrongly}, digraphs from (i) and (ii) above are strongly regular, i.e., have irreducible
adjacency matrices and the Perron-Frobenius theory for non-symmetric matrices applies. If $G$ is a $k$-diregular digraph, then the spectral radius $\rho(A(G)) = k$ by \cite[Theorem~8.3.6]{BrualdiCvetkovicBook}.
Moreover, by~\cite[Theorem~8.3.4]{BrualdiCvetkovicBook}, $G$ has the Perron eigenvalue $k$ with (algebraic and geometric) multiplicity $1$. The corresponding eigenspace is, in fact,
spanned by the all-ones vector. Finally, by \cite[Theorem~3.1]{Brualdi2010}, the spectrum of a bipartite digraph is invariant under multiplication by $-1$.

Theorem~\ref{thm:theGrandMultiplierTheorem} also applies to undirected graphs (viewed as digraphs where for every arc the opposite arc
is also present) and hence generalises Propositions 17--19 from \cite{NutOrbitPaper}. For convenience, we give the undirected versions of Definition~\ref{def:theMagicVector} and Theorem~\ref{thm:theGrandMultiplierTheorem}:

\begin{definition}
\label{def:theMagicVectorUndirected}
An \emph{undirected gadget} is a pair $(G, r)$, where $G$ is a graph and $r \in V(G)$, with the property that there exists a full vector $\x$ that spans the 
($1$-dimensional) kernel of $A(G)_{(r)}$. The vertex $r$ will be called the \emph{root}.
\end{definition}

Note that the triangle and pentagon graphs used in Section 4.2 of~\cite{NutOrbitPaper} are, in fact, undirected gadgets with respective demands $2$ and $-2$.
The various bouquets used in that section correspond to composite undirected gadgets.

\begin{theorem}
\label{thm:theGrandMultiplierTheoremUndirected}
Let $G$ be the base graph, i.e.\ a graph of order $n$ with $V(G) = \{v_1, v_2, \ldots, v_n\}$ and an integer eigenvalue $\lambda$ such that the ($1$-dimensional) $\lambda$-eigenspace of $G$ is spanned 
by a full vector~$\x$.
Let $(\Gamma_1, r_1), (\Gamma_2, r_2), \ldots, (\Gamma_n, r_n)$ be undirected gadgets, such that $\dem (\Gamma_i, r_i) = \lambda$ for all $i = 1, \ldots, n$.
Then the graph $M$ obtained from the disjoint union $G \sqcup \Gamma_1 \sqcup \Gamma_2 \sqcup \cdots \sqcup \Gamma_n$ by identifying $v_i$ with $r_i$ for all $i = 1,\ldots, n$
 is a nut graph.
\end{theorem}

\section{Perspective}

To give an indication of how digraphs and hence nut digraphs may have applications to physics and chemistry, 
we first note some applications of graphs. 
Adjacency matrices have many applications to physics and chemistry, where it is often appropriate to use a system of linear equations. 
For example, in the H\"{u}ckel (tight-binding) model of $\pi$ systems,
 linear combinations of basis functions on atomic sites (vertices of the molecular graph)
are used to form molecular orbitals, which are then used to construct a many-electron wavefunction for the physical state of the system. 
Interactions are limited to adjacent sites, and in the simplest cases the one-electron Hamiltonian determining the molecular orbitals 
reduces to a scaled and shifted adjacency matrix of the molecular graph, and $\pi$ electron distribution and $\pi$ energy 
follow from the eigenvectors and eigenvalues of this matrix.

Adjacency matrices also arise in approximate descriptions of photonic, optomechanical, or phononic devices. 
In these systems a set of oscillatory modes (for example, a photon at resonant frequency in a cavity) defines a basis. 
Modes may be coupled to one another (for instance, due to the spatial proximity of cavities) with couplings defining edges. 
Eigenvalues then correspond to the resonant frequencies of the entire coupled system, and eigenvectors to the spatial profile of the many-mode wavefunctions at that frequency.

Typically $H$ for an isolated system is Hermitian to ensure real energy eigenvalues (a requirement of the postulates of quantum mechanics), 
but connecting the system to a larger environment in a controlled non-conservative way allows an effective non-Hermitian matrix 
(including directional arc weights~\cite{ReservoirEngineeringPaper, MicrowaveNonreciprocal}) to model the system \cite{HatanoNelson,HatanoNelson2,NonHermitianSSHmodel,BosonicKitaevChain} and is readily realisable experimentally \cite{ArbitraryBosonicNetworks_Exps,slim2024optomechanical,OptomechanicalSqueezingNHexps,li2024observation,zhang2021observation}.  
This generalisation implies that many chemical and physical applications may be usefully formulated in terms of digraphs.

In topological physics, non-Hermitian model Hamiltonians may undergo so-called \textit{topological phase transitions}, 
where Hamiltonians are split into distinct equivalence classes, and may be mapped from one class to another 
by some continuous adjustment of system parameters 
(for example, altering arc weights of a digraph to change the parity of the number of negative-eigenvalue 
states \cite{ClassificationTable1}). 
In some approaches to the classification of such phases, ambi- and bi- nut digraphs may describe systems that 
have two distinct topological phases \cite{ClassificationTable1, Classification2, Classification3} and 
(with unit arc weights) 
may lie on a topological phase boundary. 
This connection gives a strong hint that nut digraphs may also be applied to the classification of topological phases. 

Ambi- and bi-nut digraphs also have potential applications in chemical physics, for example in models of molecular conduction. 
In the eponymous source-and-sink model, where a device in a circuit is modelled by connection of 
vertices of the molecular graph to special source and sink vertices, graph nullity is a crucial factor. 
In particular, a molecular graph can be a \emph{strong omni-conductor}, such that
connection to any vertex or pair of vertices leads to a prediction of transmission 
at the Fermi level~\cite{SSPmodel,Omniconductors}.
The nut graphs are exactly the strong omni-conductors of nullity 1, as the property
depends on the fact that all first minors of the adjacency matrix are non-zero, 
a requirement fulfilled by all ambi- and bi-nut digraphs. (By Corollary~\ref{lem:extendedSciriha}, all
principal minors are non-zero, but this result is easily extended to all first minors.)

Laevo- and dextro-nut digraphs could be useful in the description of novel physical phenomena. Left and right null-vectors are not simultaneously 
full in such digraphs, and the Greens function of the adjacency matrix \cite{DoubledGreenFunctionsExample, FormalGreensFunctions} 
 indicates that the digraph will display (asymmetric) directional transport.
Furthermore, such digraphs are prime candidates for exhibiting a form of \textit{topological protection} 
(which has been experimentally observed for undirected graphs in \cite{LinearPaper} and will be subjected to further theoretical exploration). 

Our discussion indicates that applications of
laevo- and dextro-nut digraphs will be distinct from those of  ambi- and bi-nut digraphs, 
but we expect that the construction of larger digraphs using the various forms of nut digraphs as building blocks  will result in systems 
with novel physics. 

We suspect that there are many more applications of nut digraphs, beyond the few outlined above.

\section*{Acknowledgements}

We thank Dr James Tuite (The Open University) for asking a stimulating question at the IWONT2023 
meeting held at the International Centre for Mathematical Sciences, Edinburgh, July 2023.
The work of Nino Bašić and Primo\v{z} Poto\v{c}nik is supported in part by the Slovenian Research Agency (Research Program P1-0294).
Maxine McCarthy thanks the EPSRC for a PhD studentship (Project Reference 2482664),
and Dr Clara Wanjura (MPL) for useful discussions on non-Hermitian physics.
Patrick Fowler thanks the Leverhulme Trust for an Emeritus Fellowship on the theme of {\lq Modelling molecular currents, conduction and aromaticity\rq}.

\end{document}